\documentclass[oneside,english]{amsart}
\usepackage[T1]{fontenc}
\usepackage[latin9]{inputenc}
\usepackage{geometry}
\usepackage{color}
\usepackage{babel}
\usepackage{float}
\usepackage{units}
\usepackage{amsthm}
\usepackage{amsbsy}
\usepackage{amstext}
\usepackage{amssymb}
\usepackage{mathdots}
\usepackage{graphicx}
\usepackage{xargs}[2008/03/08]
\usepackage[all]{xy}
\usepackage[unicode=true,pdfusetitle,
 bookmarks=false,
 breaklinks=true,pdfborder={0 0 1},backref=false,colorlinks=true]
 {hyperref}

\makeatletter

\newcommand{\noun}[1]{\textsc{#1}}
\providecommand{\tabularnewline}{\\}

\numberwithin{equation}{section}
\numberwithin{figure}{section}
\newcommand{\lyxrightaddress}[1]{
\par {\raggedleft \begin{tabular}{l}\ignorespaces
#1
\end{tabular}
\vspace{1.4em}
\par}
}
 \newcommand\thmsname{\protect\theoremname}
 \newcommand\nm@thmtype{theorem}
 \theoremstyle{plain}
 
 \newenvironment{namedthm}[1][Undefined Theorem Name]{
   \ifx{#1}{Undefined Theorem Name}\renewcommand\nm@thmtype{theorem*}
   \else\renewcommand\thmsname{#1}\renewcommand\nm@thmtype{namedtheorem}
   \fi
   \begin{\nm@thmtype}}
   {\end{\nm@thmtype}}
  \theoremstyle{plain}
  \newtheorem*{prop*}{\protect\propositionname}
  \theoremstyle{plain}
  \newtheorem*{fact*}{\protect\factname}
\theoremstyle{plain}
\newtheorem{thm}{\protect\theoremname}[section]
  \theoremstyle{definition}
  \newtheorem{defn}[thm]{\protect\definitionname}
  \theoremstyle{plain}
  \newtheorem{prop}[thm]{\protect\propositionname}
  \theoremstyle{remark}
  \newtheorem{rem}[thm]{\protect\remarkname}
  \theoremstyle{remark}
  \newtheorem*{rem*}{\protect\remarkname}
  \theoremstyle{plain}
  \newtheorem{lem}[thm]{\protect\lemmaname}
  \theoremstyle{plain}
  \newtheorem{cor}[thm]{\protect\corollaryname}
  \theoremstyle{plain}
  \newtheorem{fact}[thm]{\protect\factname}

\usepackage{kpfonts}
\usepackage{esint}

\AtBeginDocument{

}

\makeatother

  \providecommand{\corollaryname}{Corollary}
  \providecommand{\definitionname}{Definition}
  \providecommand{\factname}{Fact}
  \providecommand{\lemmaname}{Lemma}
  \providecommand{\propositionname}{Proposition}
  \providecommand{\remarkname}{Remark}
  \providecommand{\theoremname}{Theorem}
\providecommand{\theoremname}{Theorem}

\begin{document}
\global\long\def\re#1{\Re\left(#1\right)}
\global\long\def\im#1{\Im\left(#1\right)}
\global\long\def\tx#1{\mathrm{#1}}
\global\long\def\ww#1{\mathbb{#1}}
\global\long\def\nf#1#2{\nicefrac{#1}{#2}}
\global\long\def\pp#1#2{\frac{\partial#1}{\partial#2}}
\global\long\def\ii{\tx i}
\newcommandx\p[4][usedefault, addprefix=\global, 1=]{\frac{\partial^{#1}#2}{\partial#3^{#1}}|_{#4}}
\global\long\def\germ#1{\ww C\left\{  #1\right\}  }
\global\long\def\pol#1#2{\ww C\left[#1\right]_{#2}}
\global\long\def\sing#1{\tx{Sing}\left(#1\right)}
\global\long\def\diff#1{\tx{Diff}\left(\ww C^{#1},0\right)}
\global\long\def\frml#1{\ww C\left[\left[#1\right]\right]}
\global\long\def\dd#1{\tx d#1}
\global\long\def\adh#1{\tx{adh}\left(#1\right)}
\global\long\def\cc#1{\tx{cc}\left(#1\right)}
\global\long\def\id{\tx{Id}}
\global\long\def\norm#1#2{\left|\left|#1\right|\right|_{#2}}
\global\long\def\fgrm#1{\widehat{\germ{#1}}}
\global\long\def\igerm#1{\underrightarrow{\germ{#1}}}
\global\long\def\pgerm#1{\underleftarrow{\germ{#1}}}
\global\long\def\coef#1#2#3{\left[#1\right]_{#2}^{#3}}
\global\long\def\tang#1#2#3{\mathcal{T}_{#1}^{#2}#3}
\global\long\def\rat#1{\ww C\left(#1\right)}
\global\long\def\mero#1{\ww C\left(\left\{  #1\right\}  \right)}
\global\long\def\pmero#1{\mero{#1}_{0}}
\global\long\def\rad#1{\mathcal{R}\left(#1\right)}
\global\long\def\sone{\ww S^{1}}
\newcommandx\Tayl[2][usedefault, addprefix=\global, 1=k]{\mathcal{T}_{#1}\left(#2\right)}
\global\long\def\bb#1{\boldsymbol{#1}}
\global\long\def\haat#1#2{\triangleleft_{#1}#2}
\newcommandx\ball[3][usedefault, addprefix=\global, 1=a, 2=0, 3=r]{B_{#1}\left(#2,#3\right)}
\newcommandx\polyd[2][usedefault, addprefix=\global, 1=0, 2=r]{D\left(#1,#2\right)}
\global\long\def\hol#1{\tx{Hol}\left(#1\right)}
\global\long\def\holo#1{\frak{h}_{#1}}
\global\long\def\car#1{\mathfrak{c}_{#1}}
\global\long\def\amp#1{\mathfrak{a}_{#1}}
\global\long\def\ugerm#1{\mathcal{O}\left(\left(\ww C,0\right)\to#1\right)}
\global\long\def\flow#1#2{\Phi_{#1}^{#2}}
\global\long\def\quot{\tx{Quot}}
\global\long\def\rank#1{\tx{rank}\left(#1\right)}
\global\long\def\img#1{\tx{Im\,}#1}
\global\long\def\coef#1{\frak{C}_{#1}}

\title{Analyticity in spaces of convergent power series and applications$^{\dagger}$}

\maketitle

\lyxrightaddress{$^{\dagger}$Preprint version: May 2015}

\author{Loïc Teyssier}

\date{August 2013}

\address{Laboratoire I.R.M.A., Université de Strasbourg}

\email{\texttt{teyssier@math.unistra.fr}}

\thanks{This work has been partially supported by French national grant ANR-13-JS01-0002-01.}
\begin{abstract}
We study the analytic structure of the space of germs of an analytic
function at the origin of $\ww C^{m}$, namely the space $\germ{\mathbf{z}}$
where $\mathbf{z}=\left(z_{1},\cdots,z_{m}\right)$, equipped with
a convenient locally convex topology. We are particularly interested
in studying the properties of analytic sets of $\germ{\mathbf{z}}$
as defined by the vanishing loci of analytic maps. While we notice
that $\germ{\mathbf{z}}$ is not Baire we also prove it enjoys the
analytic Baire property: the countable union of proper analytic sets
of $\germ{\mathbf{z}}$ has empty interior. This property underlies
a quite natural notion of a generic property of $\germ{\mathbf{z}}$,
for which we prove some dynamics-related theorems. We also initiate
a program to tackle the task of characterizing glocal objects in some
situations. 
\end{abstract}

\keywords{Infinite-dimensional holomorphy, complex dynamical systems, holomorphic
solutions of differential equations, Liouvillian integrability of
foliations}

\subjclass[2000]{[MSC2010] 46G20, 58B12, 34M99, 37F75}

\section{Introduction}

This article purports to provide a context in which the following
results make sense:
\begin{namedthm}[Corollary A]
\textbf{\emph{}}Fix $m\in\ww N$.\textbf{ }The generic finitely generated
subgroup $G<\diff m$\emph{ }of germs of biholomorphisms fixing $0\in\ww C^{m}$,
identified with a tuple of generators $\left(\Delta_{1},\ldots,\Delta_{n}\right)$,
is free. Besides the set of non-solvable subgroups of $\diff{}$ generated
by two elements is Zariski-full, in the sense that it contains (and,
as it turns out, is equal to) an open set defined as the complement
of a proper analytic subset of $\diff{}^{n}$. The latter result persists
in the case of $n$-generated subgroups of $\diff{}$ everyone of
which is tangent to the identity.
\end{namedthm}
\bigskip{}

\begin{namedthm}[Corollary B]
\textbf{\emph{}}The generic germ of a holomorphic function $f\in\germ z$
near the origin of $\ww C$ is not solution of any ordinary differential
equation $f^{\left(k+1\right)}\left(z\right)=P\left(z,f\left(z\right),\ldots,f^{\left(k\right)}\left(z\right)\right)$,
where $P$ is an elementary function of all its variables, differentiable
at $\left(0,f\left(0\right),\ldots,f^{\left(k\right)}\left(0\right)\right)$,
and $k\in\ww N$. 
\end{namedthm}
\bigskip{}

\begin{namedthm}[Corollary C]
\textbf{\emph{}}The set of coprime $P,\text{ }Q\in\germ{x,y}$, with
zero first-jet, such that 
\begin{align*}
y' & =\frac{P\left(x,y\right)}{Q\left(x,y\right)}
\end{align*}
is not solvable in «closed form» constitutes a Zariski-full set.
\end{namedthm}
\bigskip{}

The main concern of the article is proposing a framework in which
the analytic properties of $\germ z$ can easily be manipulated. For
that reason the above results should be understood as showcase consequences
of more general theorems stated below. Other consequences in the realm
of dynamics will also be detailed in due time, for there lies the
original motivation of this work. In the meantime the objects involved
above must be outlined, and we postpone more formal definitions to
the body of the article.

\subsection{Statement of the main results}

~

Roughly speaking a property $\mathcal{P}$ expressed on the set of
germs of holomorphic functions is \emph{generic}\footnote{An element satisfying a generic property will also be called «generic»
for convenience, as in the three corollaries stated at the very beginning
of the introduction.} if the subset where $\mathcal{P}$ does not hold is contained in
countably many (proper) analytic sets. This concept supports a genericity
\emph{à la }Baire embodied in a metrizable, locally convex topology
on $\germ{\mathbf{z}}$, where $\mathbf{z}=\left(z_{1},\cdots,z_{m}\right)$.
The topology is induced by a convenient family of norms. Such spaces
will be referred to as \emph{normally convex} spaces. Special norms
will be of interest, those defined by
\begin{align*}
\norm{\sum_{\mathbf{n}\in\ww N^{m}}f_{\mathbf{n}}\mathbf{z^{n}}}a & :=\sum_{\mathbf{n}\in\ww N^{m}}a_{\mathbf{n}}\left|f_{\mathbf{n}}\right|\,,\,\tag{\ensuremath{\star}}
\end{align*}
 where $a=\left(a_{\mathbf{n}}\right)_{\mathbf{n}\in\ww N^{m}}$ is
a multi-sequence of positive real numbers satisfying the additional
growth condition
\begin{align*}
\lim_{\left|\mathbf{n}\right|\to\infty}a_{\mathbf{n}}^{\nf 1{\left|\mathbf{n}\right|}} & =0\,,
\end{align*}
which ensures the convergence of the series~$\left(\star\right)$.
This particular choice of a topology on $\germ{\mathbf{z}}$, instead
of the <<usual>> ones, is motivated by the theory of analyticity
between locally convex spaces, developed in its modern form during
the $70$'s by various mathematicians (\noun{S.~Dineen, L.~Nachbin
}and \noun{J.~Sebasti\~ao~e~Silva} among others) in the wake of
the works of J.-P.~\noun{Ramis} for Banach analytic spaces. A brief
history of the theory of infinite-dimensional holomorphy in locally
convex space is conducted in~\cite[p101-104]{Didi}. In that setting
a map $\Lambda\,:\,E\to F$ is \emph{analytic} if it is the sum of
a <<convergent power series>> whose term of <<homogeneous degree>>
$p\in\ww N$ is some $p$-linear, continuous mapping\footnote{We refer to the books by \noun{J.-A.~Barroso}~\cite{Barro} and
\noun{P.~Mazet}~\cite{Mazette} for a presentation of the different
(equivalent) types of analyticity and related analytic sets, of which
we give a short summary in Section~\ref{sec:usual_holo}.}. 

\bigskip{}

To keep matters brief we wish to manipulate relations like
\begin{align*}
f\circ g & =g\circ f
\end{align*}
or, for the sake of example,
\begin{align*}
\left(f''\right)^{2} & =1+\exp f'+f\times f^{\left(3\right)}\,,
\end{align*}
as analytic relations on the corresponding space of holomorphic germs.
While this is true for the inductive topology on $\germ{\mathbf{z}}$,
it is not too difficult to prove that the standard differentiation
$f\mapsto\pp fz$ or the composition mapping 
\begin{eqnarray*}
\left(f,g\right) & \longmapsto & f\circ g
\end{eqnarray*}
are analytic if one endows $\germ{\mathbf{z}}$ with a convenient
sub-family of norms $\left(\star\right)$. We would like to point
out that not only is the continuity of the linear left-composition
operator $g^{*}\,:\,f\mapsto f\circ g$ ascertained, provided $g\left(0\right)=0$,
but so is the analyticity of the right-composition mapping $f_{*}\text{ }:\text{ }g\mapsto f\circ g$,
for given $f$. The Taylor expansion
\begin{eqnarray*}
f\circ g & = & \sum_{p=0}^{\infty}\frac{f^{\left(p\right)}\left(g\left(0\right)\right)}{p!}\left(g-g\left(0\right)\right)^{p}
\end{eqnarray*}
is indeed convergent on the open set 
\begin{eqnarray*}
 & \left\{ g\in\germ{\mathbf{z}}\,:\,\left|g\left(0\right)\right|<\mathcal{R}\left(f\right)\right\} 
\end{eqnarray*}
(see Proposition~\ref{prop:compo_is_ana}) where $\mathcal{R}\left(f\right)$
stands for the radius of convergence of the Taylor series of $f$
at $0$. It should be noticed straight away that $\mathcal{R}$ is
lower semi-continuous but cannot be positively lower-bounded on any
domain of $\germ z$, which constitutes of course a source of trouble
and, as a by-product, reveals that the composition mapping $\left(f,g\right)\mapsto f\circ g$
cannot be defined on any domain of $\germ z\times\germ{\mathbf{z}}$
(for given $g$ with $g\left(0\right)\neq0$ there will always exist
$f$ not defined at $g\left(0\right)$).

\bigskip{}

Once we are granted the notion of analyticity we can speak of \emph{analytic
sets}, closed subsets locally defined by the vanishing of a collection
of analytic functions. As we recall later the space $\germ{\mathbf{z}}$
is not Baire and at first glance the notion of $G_{\delta}$-genericity
must arguably be discarded. Yet it is possible to salvage this concept
by hardening the rules Baire's closed sets are required to play by:
\begin{namedthm}[Theorem A]
\textbf{\emph{}}The space $\germ{\mathbf{z}}$ enjoys the analytical
Baire property: any \emph{analytically meager} set (at most countable
unions of proper analytic sets) of $\germ{\mathbf{z}}$ have empty
interior.
\end{namedthm}
\bigskip{}

This theorem bolsters the well-foundedness of the concept of genericity
in $\germ{\mathbf{z}}$ as introduced above. It has been originally
proved in a weaker form in~\cite{GenzTey} while encompassing issues
of glocality, which is a subject we come back to afterward. In fact
we can prove a stronger refinement of the main result of the cited
reference:
\begin{namedthm}[Theorem B]
\textbf{\emph{}}Let $\Lambda\,:\,U\subset E\to\germ{\mathbf{w}}$
be a map analytic on a relative open set in a linear subspace $E<\germ{\mathbf{z}}$
of at most countable dimension. Then the range of $\Lambda$ is analytically
meager.
\end{namedthm}
The analytically meager set $\Lambda\left(U\right)$ seems smaller
than other analytically meager sets. We call \emph{countably meager}
a set obtained in that way, that is an at-most-countable union of
images of at-most-countable-dimensional spaces by analytic mappings.
For instance the proper analytic set $\left\{ f~:~f\left(0\right)=0\right\} $
is not countably meager. By definition, the image of a countably meager
set by an analytic mapping is again countably meager. It is not known
whether the image of an arbitrary analytically meager set by an analytic
map remains analytically meager.

\bigskip{}

Such a result is typically useful in conjunction with Theorem~A to
derive existential properties by proving, say, that the image of the
set of polynomials by a given analytic map cannot cover a domain of
$\germ{\mathbf{w}}$ (in particular it cannot be locally onto). That
was the targeted objective of~\cite{GenzTey} through the use of
a different concept of analyticity (the so-called here \emph{quasi-strong
analyticity}) and analytic sets, which we also come back to later.
The topology on $\germ z$ was there induced by the norm $\norm{\bullet}{\left(\nicefrac{1}{n!}\right)_{n}}$.
This topology, though, is too coarse to be used in Corollaries~A,~B
and C. Instead one can consider the metrizable topology induced by
the collection $\left(\norm{\bullet}{\left(n!^{-\nicefrac{1}{k}}\right)_{n}}\right)_{k\in\ww N_{>0}}$,
which we name \emph{factorial topology}. Other choices, called \emph{useful}
\emph{topologies,} are possible to carry out most proofs; roughly
speaking any collection of norms making multiplication, composition
and differentiation continuous can do. Notice that no finite collection
can be useful: the presence of normally convex spaces equipped with
an infinite collection of norms (instead of \emph{e.g. }normed space)
is a necessity when dealing with the composition or the differentiation.
The major drawback is that no differential geometry (\emph{e.g.} Fixed-Point
Theorem\footnote{We explain in the course of the article why Nash-Moser theorem cannot
be applied in the generality of the spaces under study.}) exists in such a general framework, although this is precisely what
one would like to use in conjunction with the Fréchet calculus on
first-order derivatives. In order to make a constructive use of the
latter we need some rudiments of analytical geometry in $\germ{\mathbf{z}}$,
which is unfortunately not available for now.

\subsection{Deduction of the Corollaries from the Theorems}

~

\subsubsection{Corollary~A}

~

Any algebraic relation between $n$ generators $\left(\Delta_{1},\cdots,\Delta_{n}\right)$
of $G$ writes
\begin{align*}
\bigcirc_{\ell=1}^{k}\,\Delta_{j_{\ell}}^{\circ n_{\ell}}-\id & =0
\end{align*}
for a given depth $k\in\ww N_{>0}$ and a collection of couples $\left(j_{\ell},n_{\ell}\right)\in\ww N_{\leq n}\times\ww Z_{\neq0}$.
Each one of these countably many choices defines an analytic set\footnote{Since each diffeomorphism $\Delta_{j}$ fixes $0$ the composition
mappings $\left(\Delta_{j},\Delta_{k}\right)\mapsto\Delta_{j}\circ\Delta_{k}$
are analytic on the whole $\diff m^{2}$, as stated in Section~\ref{sub:anal_space}.
We also refer to this section for the definition of analytic mappings
from and into the non-linear space $\diff m$.} of $\diff m^{n}$, which is proper if the relation does not boil
down to the trivial relation $\Delta\circ\Delta^{\circ-1}-\id=0$
(after recursive simplification of said trivial relations). Abelian
groups form the analytic set defined by the (non-trivial) relations
\begin{align*}
\left[\Delta_{j},\Delta_{k}\right]-\id & =0\,\,\,,\,j\neq k\,,
\end{align*}
where as usual $\left[f,g\right]:=f^{\circ-1}\circ g^{\circ-1}\circ f\circ g$.
In the unidimensional case \noun{D.~Cerveau} and \noun{R.~Moussu}
on the one hand, \noun{Y.~Il'Yashenko} and \noun{A.~Sherbakov} on
the other hand, have studied the structure of the finitely-generated
subgroups of $\diff{}$, proving results later refined by \noun{F.~Loray}~\cite{theseFrank}
which are as follows.
\begin{itemize}
\item A finitely generated subgroup $G$ of $\diff{}$ is solvable if, and
only if, it is meta-Abelian (\emph{i.e. }its second derived group
$\left[\left[G,G\right],\left[G,G\right]\right]$ is trivial).
\item $G:=\left\langle f,g\right\rangle $ is solvable if, and only if,
\begin{align*}
\left[f,\left[f,g^{\circ2}\right]\right]-\id & =0\,.
\end{align*}
This again is an analytic relation and the complement in $\diff{}^{2}$
of the corresponding analytic set is a Zariski-full open set.
\item If all the generators of $G$ are tangent to the identity then $G$
is solvable if, and only if, it is Abelian.
\end{itemize}
All these points put together prove the corollary.

\subsubsection{Corollary~B}

~

We actually prove a stronger result than Corollary~B.
\begin{prop*}
Being given $k\in\ww N$ and $\mathcal{E}_{k}$ a countably meager
subset of $\germ{z,\delta_{0},\ldots,\delta_{k}}$, the set $\mathcal{S}$
of germs $f$ in $\germ z$ that are solutions of some differential
equation 
\begin{align*}
P\left(z,f\left(z\right)-f\left(0\right),\ldots,f^{\left(k\right)}\left(z\right)-f^{\left(k\right)}\left(0\right)\right) & =0~,
\end{align*}
with 
\begin{align*}
P & \in\mathcal{E}_{k}^{*}:=\mathcal{E}_{k}\backslash\left\{ Q\,:\,\frac{\partial Q}{\partial\delta_{k}}\left(\mathbf{0}\right)=0\right\} \,,
\end{align*}
is countably meager. 
\end{prop*}
We explain at the end of this paragraph what we mean by an elementary
functions in the variable $\mathbf{z}=\left(z_{1},\ldots,z_{m}\right)$,
and why the trace of their set on $\germ{\mathbf{z}}$ is countably
meager. We thus obtain Corollary~B from the proposition by setting
$\mathbf{z}:=\left(z,\delta_{0},\ldots,\delta_{k}\right)$.

\bigskip{}

The proposition is a consequence of Theorem~B above and Theorem~C
below. For the sake of concision set 
\begin{align*}
F_{k}^{f}\left(z\right): & =\left(z,f\left(z\right)-f\left(0\right),\ldots,f^{\left(k\right)}\left(z\right)-f^{\left(k\right)}\left(0\right)\right)\,.
\end{align*}
The natural approach is to consider the vanishing locus $\Omega$
of the analytic map
\begin{align*}
\germ{z,\delta_{0},\ldots,\delta_{k}}^{*}\times\germ z & \longrightarrow\germ z\\
\left(P,f\right) & \longmapsto P\circ F_{k}^{f}
\end{align*}
(this map is analytic if both source and range spaces are given a
useful topology). The set of those germs satisfying at least one differential
equation of the requested type is therefore the sub-analytic set given
by the canonical projection of $\Omega\cap\left(\mathcal{E}_{k}^{*}\times\germ z\right)$
on the second factor $\germ z$. To guarantee that this projection
has empty interior we provide a parameterized covering by the space
of equations and initial conditions, using:
\begin{namedthm}[Theorem~C]
Fix $m\in\ww N$ and consider the space $\mathtt{VF}$ of germs at
$\mathbf{0}\in\ww C^{m}$ of a holomorphic vector field, identified
with $\germ{\mathbf{z}}^{m}$, endowed with the factorial topology.
For $X\in\mathtt{VF}$ we name $\flow X{}$ the flow of $X$, that
is the unique germ of a holomorphic mapping near $\left(\mathbf{0},0\right)$
\begin{align*}
\flow X{}\,:\,\ww C^{m}\times\ww C & \longrightarrow\ww C^{m}\\
\left(\mathbf{p},t\right) & \longmapsto\Phi_{X}^{t}\left(\mathbf{p}\right)
\end{align*}
that is solution of the differential system 
\begin{align*}
\begin{cases}
\dot{\mathbf{z}}\left(\mathbf{p},t\right) & =X\left(\mathbf{z}\left(\mathbf{p},t\right)\right)\\
\mathbf{z}\left(\mathbf{p},0\right) & =\mathbf{p}\,.
\end{cases}
\end{align*}
Then the <<flow mapping>>
\begin{align*}
\mathtt{VF} & \longrightarrow\germ{\mathbf{z},t}^{m}\\
X & \longmapsto\flow X{}\,,
\end{align*}
where the target space is also given the factorial topology, is analytic.
\end{namedthm}
Take now $f\in\germ z$ and $P\in\germ{z,\delta_{0},\ldots,\delta_{k}}^{*}$.
Using the usual trick of differentiating once more the equation we
obtain $\frac{\dd{P\circ F_{k}^{f}}}{\dd z}\left(z\right)=0$, which
we rewrite $\hat{P}\circ F_{k+1}^{f}=0$ with: 
\begin{align*}
\hat{P}\left(z,\delta_{0},\ldots,\delta_{k+1}\right) & :=\frac{\partial P}{\partial z}\left(z,\delta_{0},\ldots,\delta_{k}\right)+\sum_{j=0}^{k}\frac{\partial P}{\partial\delta_{j}}\left(z,\delta_{0},\ldots,\delta_{k}\right)\delta_{j+1}\,.
\end{align*}
From this we deduce an explicit non-trivial $\left(k+1\right)^{\tx{th}}$-order
differential equation, whose solutions are obtained through the flow
of the companion vector field 
\begin{align*}
\frak{X}\left(P\right) & =\pp{}z+\sum_{j=0}^{k-1}\delta_{j+1}\pp{}{\delta_{j}}-\frac{\frac{\partial P}{\partial z}+\sum_{j=0}^{k-1}\frac{\partial P}{\partial\delta_{j}}\delta_{j+1}}{\pp P{\delta_{k}}}\pp{}{\delta_{k}}\,,
\end{align*}
which is holomorphic for $\frac{\partial P}{\partial\delta_{k}}\left(\mathbf{0}\right)\neq0$.
Obviously $P\in\germ{z,\delta_{0},\ldots,\delta_{k}}^{*}\mapsto\frak{X}\left(P\right)$
is analytic for any useful topology. Consider the map
\begin{align*}
\Phi_{k}\,:\,\pol z{\leq k}\times\germ{z,\delta_{0},\ldots,\delta_{k}}^{*} & \longrightarrow\germ z\\
\left(J,P\right) & \longmapsto\left(z\mapsto J\left(z\right)+\Pi\circ\flow{\frak{X}\left(P\right)}z\left(\mathbf{0}\right)\right)
\end{align*}
where $\Pi$ denotes the natural projection $\left(z,\delta_{0},\ldots,\delta_{k}\right)\mapsto\delta_{0}$
and $\pol z{\leq k}$ the vector space of complex polynomials of degree
at most $k$. Theorem~C asserts the analyticity of the map if the
space $\germ z$ is equipped with the factorial topology. By construction,
if $f$ is given satisfying an equation $P\in\germ{z,\delta_{0},\ldots,\delta_{k}}^{*}$
then $f=\Phi_{k}\left(J,P\right)$ where $J$ is the $k^{\tx{th}}$-jet
of $f$ at $0$ (the canonical projection of $f\in\germ z$ on $\pol z{\leq k}$).
The set $\mathcal{S}$ must therefore be included in the countable
union of sets $\left(\Phi_{k}\left(\pol z{\leq k}\times\mathcal{E}_{k}^{*}\right)\right)_{k\in\ww N}$,
each one of which is countably meager accounting for Theorem~B. This
completes the proof of the proposition.

\bigskip{}

Let us conclude this paragraph by proving our claim regarding elementary
functions. Denote by ${\tt DiffAlg}_{\ww K}\left(x\right)$ the abstract
set of all solutions of at least one autonomous, polynomial differential
equations of finite order in the variable $x$ with coefficients in
a subfield $\ww K\leq\ww C$. By ${\tt DiffAlg}_{\ww K}\left(x\right)\cap\germ x$
we mean the concrete representation of those elements of ${\tt DiffAlg}_{\ww K}\left(x\right)$
defining actual germs of analytic functions at $0$. 

Define the increasing union of fields
\begin{align*}
\ww E\left(\mathbf{z}\right) & :=\bigcup_{n\in\ww N}\ww E_{n}\left(\mathbf{z}\right)
\end{align*}
where 
\begin{align}
\ww E_{0}\left(\mathbf{z}\right) & :=\rat{\mathbf{z}}\nonumber \\
\ww E_{n+1}\left(\mathbf{z}\right) & :=\rat{\mathbf{z},~{\tt DiffAlg}_{\ww C}\left(x\right)\circ\ww E_{n}\left(\mathbf{z}\right)}~.\label{eq:elementary_tower}
\end{align}
Hence $\ww E_{n+1}\left(\mathbf{z}\right)$ is the field generated
by all variables $z_{j}$ and all compositions of an element of ${\tt DiffAlg}\left(x\right)$
with an element of $\ww E_{n}\left(\mathbf{z}\right)$. Notice that
$\ww E\left(z\right)$ contains also most special functions, with
the notable exception of the Gamma function (Hölder's theorem, see
also~\cite{Moo}). The field ${\tt Elem}\left(\mathbf{z}\right)$
obtained by replacing ${\tt DiffAlg}_{\ww C}\left(x\right)$ with
$\left\{ \exp x,\log x,\arcsin x\right\} $ in~\eqref{eq:elementary_tower},
is the field of \emph{elementary functions} in $\mathbf{z}$, which
is naturally a subfield of $\ww E\left(\mathbf{z}\right)$. 

\bigskip{}

We wish to establish the following fact by induction on $m$.
\begin{fact*}
$\ww E\left(\mathbf{z}\right)\cap\germ{\mathbf{z}}$ is countably
meager 
\end{fact*}
Let us first deal with the case $m=1$, that is $\mathbf{z}=\left(z\right)$.
\noun{E.~H.~Moore} proved the following result~:
\begin{namedthm}[{Lemma \cite[p54-55]{Moo}}]
 $\ww E\left(z\right)={\tt DiffAlg}_{\ww C}\left(z\right)$. Moreover
this field is stable by differentiation and compositional inversion.
\end{namedthm}
Pick $P\in\pol{\delta_{0},\ldots,\delta_{k}}{}$ and $f\in\ww E\left(x\right)$
such that $P\left(f,f',\ldots,f^{\left(k\right)}\right)=0$ (for short,
let us say that $f$ is solution of $P=0$). We may write $\delta=\left(\delta_{0},\ldots,\delta_{k}\right)$,
$\delta^{f}:=\left(f,f',\ldots,f^{\left(k\right)}\right)$ and $P\left(\delta\right)=\sum_{\left|\alpha\right|\leq d}\lambda_{\alpha}\delta^{\alpha}$
with $\lambda_{\alpha}\in\ww C$ and $\alpha\in\ww N^{k+1}$. By isolating
a single variable $\lambda_{\beta}$ appearing in $P$, that is 
\begin{align*}
\lambda_{\beta} & =-\frac{\sum_{\alpha\neq\beta~,~\left|\alpha\right|\leq d}\lambda_{\alpha}\left(\delta^{f}\right)^{\alpha}}{\left(\delta^{f}\right)^{\beta}}~,
\end{align*}
and differentiating with respect to $x$ we obtain that $f$ is solution
of another polynomial, autonomous differential equation of order $k+1$
which does not depend on $\lambda_{\beta}$ and in which the coefficients
$\left(\lambda_{\alpha}\right)_{\alpha\neq\beta}$ continue to appear
linearly. By induction we actually prove that there exists $\tilde{P}\in\ww Q\left[\delta_{0},\ldots,\delta_{k'}\right]$
for some $k'\geq k$ such that $f$ is solution of $\tilde{P}=0$,
so that
\begin{align*}
{\tt DiffAlg}_{\ww Q}\left(x\right) & ={\tt DiffAlg}_{\ww C}\left(x\right)~.
\end{align*}
Let us fix an enumeration $\left\{ Q_{n}~:~n\in\ww N\right\} $ of
the countable set $\mathcal{E}:=\ww Q\left[\left(\delta_{k}\right)_{k\in\ww N}\right]$
and fix $k_{n}\in\ww N$ in such a way that $Q_{n}\in\ww Q\left[\left(\delta_{k}\right)_{k\leq k_{n}}\right]$.
According to Moore's lemma and the previous proposition, the set $\ww E\left(z\right)\cap\germ z$
is included in $\bigcup_{n\in\ww N}\Phi_{k_{n}}\left(\pol z{\leq k_{n}}\times\left\{ Q_{n}\right\} \right)$.
Hence $\ww E\left(z\right)\cap\germ z$ is countably meager, settling
the case $m=1$. 

\bigskip{}

Assume now that the claim is true at any rank less than $m\geq1$.
For $f\in\ww E\left(\mathbf{z}\right)\cap\germ{\mathbf{z}}$ and for
given $\tilde{\mathbf{z}}=\left(z_{2},\ldots,z_{m}\right)$ small
enough consider the partial function defined by $\tilde{f}_{\tilde{\mathbf{z}}}\left(z\right):=f\left(z,z_{2},\ldots,z_{m}\right)$,
which belongs to $\ww E\left(z\right)$. According to the lemma there
exists $k=k\left(\tilde{\mathbf{z}}\right)\in\ww N$ and a polynomial
$P_{\tilde{\mathbf{z}}}\in\ww Q\left[\delta_{0},\ldots,\delta_{k}\right]$
such that $\tilde{f}_{\tilde{\mathbf{z}}}$ is solution of $P_{\tilde{\mathbf{z}}}=0$.
We can split a small domain $\Omega\subset\ww C^{m-1}$ containing
$\mathbf{0}$ into countably many sets $\Omega_{n}$, consisting precisely
of those points $\tilde{\mathbf{z}}$ for which $P_{\tilde{\mathbf{z}}}$
equals the $n^{\tx{th}}$ element $Q_{n}$ of $\mathcal{E}$. According
to the principle of analytic continuation we may assume without loss
of generality that each $\Omega_{n}$ is closed in $\Omega$. From
Baire's theorem we deduce that at least one $\Omega_{n}$ has non-empty
interior, thus (again by analytic continuation) $\tilde{f}_{\tilde{\mathbf{z}}}$
is solution of $Q_{n}=0$ for any $\tilde{\mathbf{z}}\in\Omega$.
Therefore $f$ is solution of
\begin{align*}
Q_{n}\left(f,\pp f{z_{1}},\pp{^{2}f}{z_{1}^{2}},\ldots,\pp{^{k_{n}}f}{z_{1}^{k_{n}}}\right) & =0
\end{align*}
and $f$ belongs to $A_{n}:=\Phi_{k_{n}}\left(\left(\ww E\left(\tilde{\mathbf{z}}\right)\cap\germ{\tilde{\mathbf{z}}}\right)^{k_{n}+1}\times\left\{ Q_{n}\right\} \right)$,
being uniquely defined by the boundary value $\left(f\left(0,\tilde{\mathbf{z}}\right),\pp f{z_{1}}\left(0,\tilde{\mathbf{z}}\right),\ldots,\pp{^{k_{n}}f}{z_{1}^{k_{n}}}\left(0,\tilde{\mathbf{z}}\right)\right)$.
The induction hypothesis and the proposition guarantee that $A_{n}$
is countably meager, hence $\ww E\left(\mathbf{z}\right)\cap\germ{\mathbf{z}}$
also is.

\subsubsection{Corollary~C}

~

A (germ of a) meromorphic, order one differential equation in $\ww C^{2}$
\begin{align*}
y' & =\frac{P\left(x,y\right)}{Q\left(x,y\right)}\tag{\ensuremath{\Diamond}}
\end{align*}
induces a (germ of a) foliation at the origin of the complex plane.
Roughly speaking, the leaves of such a foliation are the connected
Riemann surfaces corresponding to <<maximal>> solutions. By Cauchy-Lipschitz's
theorem if $P$ or $Q$ does not vanish at some point\footnote{Such a point is called regular.}
then the foliation is locally conjugate to a product of two discs.
On the contrary at a singularity of the foliation, which we locate
at $\left(0,0\right)$ for convenience, a whole range of complex behaviors
can turn up. An obvious fact is that the generic germ of a foliation\footnote{For the sake of clarity we identify the set of germs of foliations
with $\germ{x,y}^{2}=\left\{ \left(P,Q\right)\right\} $, voluntarily
forgetting that proportional couples induce the same foliation; in
particular the singular locus of a holomorphic foliation in $\ww C^{2}$
is always isolated. This technicality will be dealt with in due time.} is regular, since singular ones correspond to the analytic set of
$\germ{x,y}^{2}$ defined by 
\begin{align*}
\mathtt{Sing} & :=\left\{ \left(P,Q\right)\in\germ{x,y}^{2}\,:\,P\left(0,0\right)=Q\left(0,0\right)=0\right\} \,.
\end{align*}
From now on we solely work in $\mathtt{Sing}$, which is given the
analytic structure of $\germ{x,y}^{2}$ through the continuous, affine
and onto mapping 
\begin{align*}
\left(P,Q\right)\mapsto\left(P-P\left(0,0\right),Q-Q\left(0,0\right)\right) & \text{ .}
\end{align*}

\bigskip{}

An important question regarding foliations is that of finding solutions
of $\left(\Diamond\right)$ in «closed form», which was originally
formulated by J.~\noun{Liouville} in terms of consecutive quadratures
and exponentiation of quadratures of meromorphic functions. In the
modern framework of differential Galois theory this notion translates
as the request that every germ of a solution near every regular points
admit an analytic continuation coinciding with a determination of
an abstract solution lying in a finite tower of consecutive extensions
of differential fields $\ww K_{0}<\cdots<\ww K_{n}$ of the following
kind:
\begin{itemize}
\item we start from the field $\ww K_{0}$ of germs of meromorphic functions
near $\left(0,0\right)$,
\item $\ww K_{n+1}=\ww K_{n}\left\langle f\right\rangle $ where $f$ is
algebraic over $\ww K_{n}$,
\item $\ww K_{n+1}=\ww K_{n}\left\langle f\right\rangle $ where $f'=a\in\ww K_{n}$
(a quadrature),
\item $\ww K_{n+1}=\ww K_{n}\left\langle f\right\rangle $ where $f'=af$
with $a\in\ww K_{n}$ (an exponentiation of a quadrature).
\end{itemize}
For the sake of example let us wander a little away from the path
we are currently treading, and consider the case of a linear differential
system, where $Q$ is the $n\times n$ identity matrix and $P\left(x,y\right)=P\left(x\right)y$
is obtained from a $n\times n$ matrix $P\left(x\right)$ with entries
rational in $x$, and $y$ is a vector in $\ww C^{n}$. To simplify
further imagine that 
\begin{align*}
P\left(x\right) & =\sum_{\ell=1}^{k}\frac{D_{\ell}}{x-x_{\ell}}
\end{align*}
for some finite collection of constant, diagonal matrices $\left(D_{\ell}\right)_{\ell\leq k}$
and distinct points $x_{\ell}$ of $\ww C$, so that the system is
Fuchsian. Then the solutions are multi-valued mappings 
\begin{align*}
x\mapsto y\left(x\right) & =\prod_{\ell=1}^{k}\left(x-x_{\ell}\right)^{D_{\ell}}\times C
\end{align*}
where $C\in\ww C^{n}$ is the vector of <<initial conditions>>.
The multi-valuedness of solutions is directly related to $\left(D_{\ell}\right)_{\ell}$
and is measured by the monodromy group, obtained by performing successive
local analytic continuations, starting from some transverse line $\left\{ x=\mbox{cst}\right\} $
outside the singular locus $\bigcup_{\ell}\left\{ x=x_{\ell}\right\} $,
and returning to it after winding around the singularities (a kind
of first-return mapping acting on $C$). The monodromy group is a
representation of the fundamental group of the punctured sphere $\overline{\ww C}\backslash\left\{ x_{\ell}\,:\,\ell\leq k\right\} $
into a linear algebraic subgroup in $\tx{GL}_{n}\left(\ww C\right)$.
\noun{E.~Kolchin} (see for instance \cite{SingVDP}) related the
Liouvillian integrability of the system to the solvability of the
(connected component of identity of the Zariski-closure of the) monodromy
group. It is well known that the generic linear algebraic group is
non-solvable.

\bigskip{}

Back to the non-linear setting we would like to generalize this non-solvability
result. The candidate replacement differential Galois theory has been
introduced in a recent past by \noun{B.~Malgrange~\cite{Malg}}
and subsequently developed by \noun{G.~Casale} (we refer to \noun{\cite{Casa}}
for matters regarding our present study\noun{)}. The monodromy group
is replaced by the groupoid of holonomy, but its geometric construction
is the same up to replacement of the fundamental group of the base
space by the fundamental groupoid. Although the tools introduced here
are not powerful enough to deal with such a generality we can nonetheless
say something in the generic case.

In this paper when we speak of a \emph{reduced} singularity we mean
that the linearized differential equation at the singularity, identified
with the $2$-dimensional square matrix
\begin{align*}
L\left(P,Q\right) & :=\left[\begin{array}{cc}
\frac{\partial P}{\partial x}\left(0,0\right) & \frac{\partial P}{\partial y}\left(0,0\right)\\
\frac{\partial Q}{\partial x}\left(0,0\right) & \frac{\partial Q}{\partial y}\left(0,0\right)
\end{array}\right]\,,
\end{align*}
possesses at least one non-zero eigenvalue. It is obvious again that
the generic element of ${\tt Sing}$ is reduced, since non-reduced
foliations can be discriminated by the characteristic polynomial of
their linear part, and therefore form the analytic set 
\begin{align*}
\mathtt{Sing}\cap\left\{ \left(P,Q\right)\,:\,\det L\left(P,Q\right)=\tx{tr}L\left(P,Q\right)=0\right\}  & \,.
\end{align*}
Liouvillian integrability of foliations with $L\left(A,B\right)\neq0$
is now a well-studied topic~\cite{BerTouze,Casa} so we dismiss this
case and consider only germs of singular foliations belonging to the
proper analytic set 
\begin{align*}
\mathtt{ZLP} & :=\left\{ \left(P,Q\right)\,:\,L\left(P,Q\right)=0\right\} 
\end{align*}
(${\tt ZLP}$ standing for \emph{zero linear part}).

An old result, formalized by \noun{A.~Seidenberg~\cite{Seiden},}
states that every germ $\mathcal{F}\in\mathtt{Sing}$ of a holomorphic
foliation with a singularity at $0\in\ww C^{2}$ can be reduced, that
is: there exists a complex surface $\mathcal{M}$ and a proper rational
morphism $\pi\,:\,\mathcal{M}\to\left(\ww C^{2},0\right)$, obtained
as successive blow-ups of singular points, such that
\begin{itemize}
\item $\mathcal{E}:=\pi^{-1}\left(0\right)$, called the exceptional divisor,
is a finite, connected union of normally-crossing copies of $\ww P_{1}\left(\ww C\right)$,
\item the restriction of $\pi$ to $\mathcal{M}\backslash\mathcal{E}$ is
a biholomorphism,
\item the pulled-back foliation $\pi^{*}\mathcal{F}$ has only reduced singularities,
located on the exceptional divisor.
\end{itemize}
Either a component $\mathcal{D}$ of the exceptional divisor is transverse
to all but finitely many leaves of $\pi^{*}\mathcal{F}$, in which
case we are confronted to a \emph{dicritic component}, or $\mathcal{D}$
is a leaf of $\pi^{*}\mathcal{F}$. To a non-dicritic component $\mathcal{D}$
and any (small enough) transversal disk $\Sigma$, not meeting the
(finite) singular locus $\sing{\mathcal{F}}$ of the foliation, we
associate the <<projective>> holonomy group $\tx{Hol}\left(\mathcal{D},\Sigma\right)$
of germs of invertible holomorphic first-return maps obtained by following
(lifting) cycles of $\mathcal{D}\backslash\tx{Sing}\left(\pi^{*}\mathcal{F}\right)$,
with base-point $\mathcal{D}\cap\Sigma$, in a leaf of $\pi^{*}\mathcal{F}$.
By endowing $\Sigma$ with an analytic chart, so that $\mathcal{D}\cap\Sigma$
correspond to $0$, the group $\tx{Hol}\left(\mathcal{D},\Sigma\right)$
is naturally identified with a finitely-generated sub-group of $\diff{}$.
As the domain of definition of an element $\Delta$ may not equal
the whole $\Sigma$ this is not a sub-group of the biholomorphisms
of $\Sigma$. For this reason such objects are usually referred to
as pseudo-groups in the literature, although we consider them as groups
of germs (that is, without considering a geometric realization). If
one chooses another transverse $\tilde{\Sigma}$ then $\tx{Hol}\left(\mathcal{D},\Sigma\right)$
and $\tx{Hol}\left(\mathcal{D},\tilde{\Sigma}\right)$ are biholomorphically
conjugate. 

In the case where $\mathcal{F}$ is reduced after a single blow-up
the (conjugacy class of the) holonomy group embodies all the information
about Liouvillian integrability. In particular it is solvable if the
equation $\left(\Diamond\right)$ is integrable. Corollary~C then
follows from combining Corollary~A with the facts that such foliations
$\mathcal{F}$ are Zariski-full and that we can build an analytic,
open mapping
\begin{align*}
\mathcal{F} & \longmapsto\tx{Hol}\left(\mathcal{D},\Sigma\right)\,.
\end{align*}
A formal proof will be given in the body of the article.

\subsection{Strong analyticity}

~

To the extent of my knowledge the notion presented now has never been
thoroughly studied so far, which is a pity since it is fairly common
in actual problems. It runs as follows : a continuous map $\Lambda\,:\,U\subset\germ{\mathbf{z}}\to\germ{\mathbf{w}}$
is \emph{strongly analytic} if for any finite-dimensional family of
functions $\left(f_{\mathbf{x}}\right)_{\mathbf{x}\in\ww D^{n}}$
such that $\left(\mathbf{x},\mathbf{z}\right)\mapsto f_{\mathbf{x}}\left(\mathbf{z}\right)$
is analytic near $\mathbf{0}\in\ww C^{n}\times\ww C^{m}$, the corresponding
family $\left(\Lambda\left(f_{\mathbf{x}}\right)\right)_{\mathbf{x}}$
is also an analytic function of $\left(\mathbf{x},\mathbf{w}\right)$
near $\mathbf{0}$. This property can be easily checked\footnote{When I say that the property is easily checked I mean informally that
in actual problems it should not be more difficult to prove strong
analyticity than plain analyticity.} and ensures that the composition of two source/range compatible strongly
analytic maps is again strongly analytic. For instance the composition
and differentiation mappings trivially satisfy this property. The
next result enables all the previous theorems to apply in the case
of strongly analytic relations:
\begin{namedthm}[Theorem D]
\textbf{\emph{}}Let $\Lambda\,:\,U\subset\germ{\mathbf{z}}\to\germ{\mathbf{w}}$
be a strongly analytic map. Then $\Lambda$ is analytic.
\end{namedthm}
We will provide a criterion characterizing strongly analytic maps
among analytic ones. Notice also that in Theorem~C the flow mapping
is actually strongly analytic.

\bigskip{}

The continuity condition on $\Lambda$ can be notably relaxed and
only \emph{ample boundedness} (analogous to local boundedness as in
the characterization of continuous linear mappings) is actually required.
This is a very natural condition to impose when dealing with analyticity
in locally convex space, as the whole theory relies on making sense
of the Cauchy's formula in an infinite dimensional space (which is
the argument the proof of this theorem also is based upon).

\subsection{Applications to dynamics}

~

Having a <<nice>> composition and differentiation comes in handy
in dynamics, as we already illustrated with the introductory corollaries.
We introduce also the \emph{glocal }problem: how can one recognize
that a local dynamical system is the local trace of a global one.
We establish a tentative program to deal with this problem, which
will demand the making of a differential geometry in $\germ{\mathbf{z}}$.

\subsubsection{Infinitely-renormalizable parabolic diffeomorphisms}

~

The usual framework of unidimensional (discrete) dynamics is the iteration
of rational maps $\Delta\,:\,\ww C\to\ww C$. However to understand
the local structure of, say, the Julia set of $\Delta$ one is often
led to study its local analytic conjugacy class, which can be quite
rich indeed. For instance near a parabolic fixed-point the space of
equivalence classes under local changes of coordinates is huge\footnote{Although it has the same cardinality as $\ww R$, its algebraic dimension
is infinite.}. It is, roughly speaking, isomorphic to a product of finitely many
copies of $\diff{}$ through a one-to-one map 
\begin{align*}
\tx{\acute{E}V}\,:\,\Delta & \longmapsto\tx{\acute{E}V}\left(\Delta\right)\,,
\end{align*}
known as the <<Écalle-Voronin invariants>> mapping. Germs of biholomorphisms
have same Écalle-Voronin invariants if, and only if, they are locally
conjugate by a germ of a biholomorphism.

Understanding how the local dynamics (\emph{i.e. }the local invariants)
varies as $\Delta$ does is usually a hard but rewarding task, undertaken
for example by M.~\noun{Shishikura~}\cite{ShiShi} to prove that
the boundary of the quadratic Mandelbrot set has Hausdorff-dimension
$2$. This citation is particularly interesting since one ingredient
of the proof consists in exploiting the strong analyticity of $\tx{\acute{E}V}$.
The result also relies on the existence of infinitely-renormalizable
maps, in the following way. Suppose for the sake of simplicity that
$\Delta$ is tangent to the identity and 
\begin{align*}
\Delta\left(z\right) & =z+\alpha z^{2}+\ldots~~~~,~\alpha\neq0
\end{align*}
 so that $\tx{\acute{E}V}\left(\Delta\right)\in\diff{}^{2}$. The
invariant $\tx{\acute{E}V}\left(\Delta\right)$ is well-defined up
to a choice of a linear chart on $\ww C$, and therefore we can always
arrange that a component $\tilde{\Delta}$ of $\tx{\acute{E}V}\left(\Delta\right)$
itself is tangent to the identity. With corresponding notations, if
$\tilde{\alpha}\neq0$ then $\tx{\acute{E}V}\left(\tilde{\Delta}\right)\in\diff{}^{2}$
is again well-defined and $\Delta$ is renormalizable once. The set
of germs of diffeomorphisms which are renormalizable $n$ times is
consequently a Zariski-full open set. We particularly derive the
\begin{namedthm}[Corollary D]
\textbf{\emph{}}The generic tangent-to-the-identity germ of a diffeomorphism
of the complex line is infinitely-renormalizable.
\end{namedthm}
I'm grateful to M.~\noun{Yampolsky} for suggesting this application
to me.

\subsubsection{Application of Corollary~A to the topology of foliations}

~

Non-solvability of finitely-generated subgroups of $\diff{}$ is a
key point in studying rigidity properties of holomorphic foliations
on compact complex surfaces. In the case of \emph{e.g. }a foliation
on $\ww P_{2}\left(\ww C\right)$ it measures (through the holonomy
representation of the line at infinity $\tx{Hol}\left(\Sigma,L_{\infty}\right)$)
how the leaves are mutually entangled: if the dynamics is sufficiently
<<mixing>> then topological conjugations between such foliations
turn out to be holomorphic (\emph{i.e. }homographies). In that respect
we should cite a consequence of the Nakai theorem exploited by \noun{A.~Lins~Neto},
\noun{P.~Sad} and \noun{B.~Sc\'{a}rdua }in~\cite{NetoSadSca}.
They show that the set of topologically rigid foliations of given
degree in $\ww P_{2}\left(\ww C\right)$ contains an open and dense
set, by proving that holonomy groups $\tx{Hol}\left(\Sigma,L_{\infty}\right)$
corresponding to those foliations are non-solvable\footnote{And therefore, according to Nakai's theorem, have dense orbits in
a <<big>> domain of the transversal $\Sigma$.}. The <<generic>> freeness of $\tx{Hol}\left(\Sigma,L_{\infty}\right)$,
for a fixed foliation degree, is proved by \noun{Y.~Il'Yashenko}
and \noun{A.~Pyartli~}\cite{IlyaPyr}. We point out that the context
of both results is of a different nature from ours: for a fixed degree
the space of foliations is a finite-dimensional complex projective
space.

In the context of germs of singular foliations \noun{J.-F.~Mattei},
\noun{J.~Rebelo} and \noun{H.~Reis~}\cite{MatRebRe} devise a result
comparable to Corollary~A. It is stronger in the sense that any subgroup
$G=\left\langle \Delta_{\ell}\right\rangle _{1\leq\ell\leq n}$ can
be perturbed by the action of $\diff{}^{n}$ 
\begin{align*}
\varphi=\left(\varphi_{\ell}\right)_{1\leq\ell\leq n} & \longmapsto\varphi^{*}G:=\left\langle \varphi_{\ell}^{*}\Delta_{\ell}\right\rangle _{1\leq\ell\leq n}\,,
\end{align*}
in such a way that for a <<generic>> choice of $\varphi$ the corresponding
subgroup $\varphi^{*}G$ is free. They deduce from this result a statement
about the corresponding pseudo-group (obtained by realizing the group
on a common domain of definition of the generators) that the generic
foliation has at most countably many non-simply-connected leaves.
Here <<generic>> refers to $G_{\delta}$-genericity in $\diff{}$
for the \emph{analytic topology}, introduced by \noun{F.~Takens}~\cite{Takens},
which we present in Section~\ref{sec:topologies}. This topology
is Baire but otherwise severely flawed: it does not turn $\diff{}$
into a topological vector space (and does not enjoy a continuous composition)
which, ironically enough, forbids any reasonably interesting analytic
structure on $\diff{}$. I believe that an effective analytic geometry
in $\germ{\mathbf{z}}$ would allow to obtain much the same kind of
result, and perhaps more constructively. Indeed the Fréchet calculus
allows to identify directions transverse with the tangent space of
an analytic set (\emph{e.g.} the set of generators satisfying a given
non-trivial algebraic relation), therefore pointing directions along
which the algebraic relation between generators of $G$ will be broken
by perturbation. One should now check that these transverse directions
can be realized as tangent spaces of curves embedded in the analytic
set describing locally the property of being in the same conjugacy
class, which is a statement reaching beyond the limits of the present
article. However it is related to what comes now.

\subsubsection{The glocal problem for diffeomorphisms}

~

It is not clear how to distinguish which local objects are actually
global objects having been processed through a local change of coordinates.
The \emph{glocal problem} in the context of germs of diffeomorphisms
refers to the following question:\bigskip{}

\begin{center}
\emph{<< Is any element of $\diff{}$ locally conjugate to a rational
one ? If not, how do we recognize that some of them are ? >>}
\par\end{center}

\bigskip{}

In his thesis \noun{A.~Epstein~} underlines that $\tx{\acute{E}V}\left(\Delta\right)$,
for rational $\Delta$, behave (dynamically) very much like a rational
map itself (a <<finite-type map>>), except for the fact that it
is transcendental and must admit a frontier for analytic continuation,
so the answer to the former question is <<no>>. Yet no answer to
the latter one is known. Somehow glocality must be readable in the
map $\tx{\acute{E}V}$, but this task is a difficult one hindered
by the fact that obtaining general properties on the invariants map
is hard work.

Another approach consists in acknowledging that being glocal is a
sub-analytic property. Indeed the set of biholomorphisms locally conjugate
to a polynomial of degree $d$ is the projection on the second factor
of the analytic set $\Omega$ defined by the zero-locus of 
\begin{align*}
\diff{}\times\diff{} & \longrightarrow\germ z\\
\left(\varphi,\Delta\right) & \longmapsto\left(\mathbf{\id}-J_{d}\right)\left(\varphi^{\circ-1}\circ\Delta\circ\varphi\right)\,,
\end{align*}
where $J_{d}$ is the $d^{\tx{th}}$-jet of a germ and $\id$ stands
for the identity mapping from $\germ z$ to $\germ z$. Therefore
it should be possible to gain knowledge from the study of the tangent
space of $\Omega$. This can be done using the Fréchet calculus detailed
in this paper, although we need a more powerful tool to derive existential
(or explicit) results by geometrical arguments. An article is in preparation
regarding the glocal problem in $\diff{}$.

\subsubsection{The glocal problem for foliations}

~

A similar problem can be stated in the realm of germs of singular
holomorphic foliations in $\ww C^{2}$:

\bigskip{}

\begin{center}
\emph{<< Is any element of $\mathtt{Sing}$ locally conjugate to
a polynomial one ? >>}
\par\end{center}

\bigskip{}

A non-constructive, negative answer is given to the question in~\cite{GenzTey}:
the generic local conjugacy class of germs of saddle-node\footnote{The linear part at $\left(0,0\right)$ of $\left(\diamond\right)$
has exactly one non-zero eigenvalue.} singularities do not contain polynomial representative. The methods
used there generalize flawlessly to the case of resonant-saddle singularities\footnote{The linear part at $\left(0,0\right)$ of $\left(\diamond\right)$
has two non-zero eigenvalues whose ratio lies in $\ww Q_{\leq0}$,
not formally linearizable.}, which are 2-dimensional counterparts to parabolic diffeomorphisms
through the holonomy correspondence. Using the enhanced theorems that
we prove here, we can be slightly more precise:
\begin{namedthm}[Corollary E]
\textbf{\emph{}}The generic germ of a $2$-dimensional resonant-saddle
foliation is not glocal. More precisely, glocal resonant-saddle foliations
form a countably meager set in the space of all foliations.
\end{namedthm}
\bigskip{}

We point out that, unlike the discrete case, no explicit example of
non glocal foliation is yet known, as no characterization (even partial
ones) of glocal foliations exists. I hope that from the developments
to come for discrete dynamics will emerge a general framework in which
explicit examples and/or criteria can be devised using perturbations
along non-tangential directions provided by the Fréchet calculus.

\subsection{Structure of the paper and table of contents}

~
\begin{itemize}
\item In Section~\ref{sec:topologies} we review <<usual>> topologies
on $\germ{\mathbf{z}}$ and introduce the factorial and <<useful>>
normally convex topologies we will use. We compare their relative
thinness. We also explicit the sequential completion of these spaces.
\item In Section~\ref{sec:usual_holo} we give a short survey of the standard
definitions and general properties of analytic maps between locally
convex spaces. We also introduce strong analyticity and Theorem~D
is proved.
\item In Section~\ref{sec:anal_sets} we introduce analytic sets and prove
the Baire analyticity and related properties of $\germ{\mathbf{z}}$,
including Theorem~A and Theorem~B.
\item In Section~\ref{sec:calculus} we introduce and give examples of
the Fréchet calculus in $\germ{\mathbf{z}}$.
\item In Section~\ref{sec:application_analysis} we present more applications
of the main results to complex analysis, particularly by equipping
the field of germs of meromorphic functions with an analytical structure
modelled on $\germ{\mathbf{z}}$. In doing so we study coprimality
in the ring $\germ{\mathbf{z}}$ and prove that coprime families $\left(f_{1},\ldots,f_{k}\right)\in\germ{\mathbf{z}}^{k}$
form an open, Zariski-dense set. 
\item In Section~\ref{sec:application_ODE} we give a complete proof to
Theorem~C and Corollary~C.
\end{itemize}
\tableofcontents{}

\subsection{Notations and conventions}

~
\begin{itemize}
\item The usual sets of numbers $\ww N,\,\ww Z,\,\ww Q,\,\ww R,\,\ww C$
are used with the convention that $\ww N$ is the set of non-negative
integers.
\item Let $m\in\ww N$ ; we use bold-typed letters to indicate $m$-dimensional
vectors $\mathbf{z}=\left(z_{1},\ldots,z_{m}\right)$ and multi-indexes
$\mathbf{j}=\left(j_{1},\ldots,j_{m}\right)$. We define as usual 

\begin{itemize}
\item $\left|\mathbf{j}\right|:=\sum_{\ell=1}^{m}j_{\ell}$,
\item $\mathbf{j}!:=\prod_{\ell=1}^{m}j_{\ell}!$,
\item $\mathbf{z}^{\mathbf{j}}:=\prod_{\ell=1}^{m}z_{\ell}^{j_{\ell}}$.
\end{itemize}
\item $\ww N^{\left(\ww N\right)}$ denotes the set of all finitely-supported
sequences of non-negative integers, that can be identified with $\coprod_{m\geq0}\ww N^{m}$.
\item We also define the \textbf{insertion symbol}
\begin{eqnarray*}
\mathbf{z}\haat ja & := & \left(z_{1},\cdots,z_{j-1},a,z_{j+1},\cdots,z_{m}\right)\,.
\end{eqnarray*}

\item We use the notation <<$\oplus$>> to concatenate vectors
\begin{align*}
\left(z_{1},\cdots,z_{k}\right)\oplus\left(z_{k+1},\cdots,z_{m}\right) & :=\left(z_{1},\cdots,z_{m}\right)\,,
\end{align*}
and by a convenient abuse of notations we set $\mathbf{z}\oplus a:=\mathbf{z}\oplus\left(a\right)=\left(z_{1},\cdots,z_{m},a\right)$.
In the same way we define
\begin{align*}
\mathbf{z}^{\oplus0} & :=\left(\right)\\
\mathbf{z}^{\oplus k+1} & :=\mathbf{z}^{\oplus k}\oplus\mathbf{z}\,.
\end{align*}

\item The symbol <<$\bullet$>> stands for the argument of a mapping,
for instance $\exp\bullet$ stands for the function $z\mapsto\exp z$.
We use it when the context renders the notation unambiguous.
\item If $X$ and $Y$ are topological spaces with the same underlying set,
we write $X\geq Y$ if the identity mapping $\id\,:\,X\to Y$ is continuous.
We write $X>Y$ if in addition the spaces are not homeomorphic.
\item A complex locally convex space whose topology is induced by a family
$\left(\norm{\bullet}a\right)_{a\in A}$ of norms will be called a
\textbf{normally convex} \textbf{space}. 
\item $\frml{\mathbf{z}}$ is the complex algebra of formal power series,
$\germ{\mathbf{z}}$ the sub-algebra consisting of those that converge
(nontrivial domain of convergence). We distinguish between the formal
power series $\sum_{n\geq0}f_{n}z^{n}\in\frml z$ with its sum, when
it exists, $z\mapsto\sum_{n=0}^{\infty}f_{n}z^{n}$ understood as
a holomorphic function on a suitable domain.
\item $\pol{\mathbf{z}}{\leq d}$ is the complex vector space of all polynomial
of degree at most $d$, while $\pol{\mathbf{z}}{=d}$ is the Zariski-full
open set consisting of those with degree exactly $d$.
\item For the sake of keeping notations as simple as possible we often identify
the symbol <<$\mathbf{z}$>> both with an element of $\ww C^{m}$
and with the identity mapping of the ambient space. This ambiguity
will resolve itself according to the context. It may not be orthodoxically
sound but it will prove quite convenient in some places. 
\item In the context of the previous item the notation $\id$ will be kept
for the linear operator of an underlying, implicit vector space such
as $\pol{\mathbf{z}}{}$ or $\germ{\mathbf{z}}$.
\item $\diff m$ is the group of germs of biholomorphisms fixing the origin
of $\ww C^{m}$.
\item Throughout the article the notation $\ww D$ stands for the open unit
disc of $\ww C$. 
\item A \textbf{domain} of a topological space is a non-empty, connected
open set.
\end{itemize}

\section{\label{sec:topologies}About useful topologies on $\protect\germ{\mathbf{z}}$}

For the sake of simplicity we only present the case $m=1$, from which
the general case is easily derived. We say that the space $\germ z$
of germ of a holomorphic function at $0\in\ww C$ is a \textbf{compositing
differential algebra} when it is endowed with a structure of topological
algebra for which the differentiation $\pp{}z$ and the right- (\emph{resp.}
left-) composition $g^{*}\,:\,f\mapsto f\circ g$ with a given germ
$g$ vanishing at $0$ (\emph{resp. }$f_{*}\,:\,g\mapsto f\circ g$
with a given germ $f$), are continuous operations. We review here
some topologies usually put on the space of convergent power series.
We prove that some of them fail to be <<useful>> in the sense that
they do not induce a structure of compositing differential algebra.
For the sake of example if $\germ z$ is given a normed topology then
it cannot be a compositing differential algebra. The sequence $z\mapsto\exp\left(kz\right)$,
for $k\in\ww N$, indeed shows that differentiation is not continuous,
while $z\mapsto z^{k}$ provides a sequence for which the right-composition
with \emph{e.g. }$z\mapsto2z$ does not satisfy the axiom of continuity. 

Notice that the <<most>> natural topology on $\germ z$, the inductive
topology, is useful. Although it is not so easy to handle as compared
to the sequential topologies we introduce (in particular because the
inductive topology is not metrizable), we establish that the inductive
topology coincides with the one obtained by taking all sequential
norms. We end this section by proving that the factorial topologies
(for instance) are useful.

\subsection{\label{sub:topo_inductive}The inductive topology}

~

The usual definition of the space of germs of holomorphic functions
at $0\in\ww C$ introduces $\germ z$ as an inductive space
\begin{eqnarray*}
\germ z & := & \underrightarrow{\lim}\left(\mathcal{B}_{r}\right)_{r>0}\,,
\end{eqnarray*}
where $\mathcal{B}_{r}$ denotes the Banach space of bounded holomorphic
functions on the disk $r\ww D$ equipped with the $\sup$ norm 
\begin{eqnarray*}
\norm f{r\ww D} & := & \sup_{\left|z\right|<r}\left|f\left(z\right)\right|\,,
\end{eqnarray*}
transition maps $\mathcal{B}_{r}\to\mathcal{B}_{r'}$ for $r'\leq r$
being defined by the restriction morphisms 
\begin{align*}
\iota_{r\to r'}\,:\,f\in\mathcal{B}_{r} & \longmapsto f|_{r'\ww D}\in\mathcal{B}_{r'}\,.
\end{align*}
 An element $f\in\germ z$ is therefore understood as en equivalence
class of all couples $\left(r,f_{r}\right)\in\coprod_{r>0}\left\{ r\right\} \times\mathcal{B}_{r}$
such that
\begin{eqnarray*}
\left(r,f_{r}\right)\leftrightsquigarrow\left(r',f_{r'}\right) & \Longleftrightarrow & \left(\,\exists\,0<\rho\leq\min\left(r,r'\right)\,\right)\,\,\,f_{r}|_{\rho\ww D}=f_{r'}|_{\rho\ww D}\,.
\end{eqnarray*}
We denote by
\begin{align*}
\underline{\left(r,f_{r}\right)}
\end{align*}
the equivalence class of $\left(r,f_{r}\right)$. The <<convergence
radius>> function
\begin{eqnarray*}
\mathcal{R}\,:\,\germ z & \longrightarrow & ]0,\infty]\\
f & \longmapsto & \sup\left\{ r>0\,:\,\exists f_{r}\in\mathcal{B}_{r},\,\underline{\left(r,f_{r}\right)}=f\right\} 
\end{eqnarray*}
 is well-defined.
\begin{defn}
\label{def_inductive_germ}The topological space obtained as the direct
limit $\underrightarrow{\lim}\left(\mathcal{B}_{r}\right)_{r>0}$
equipped with the inductive topology will be denoted by 
\begin{eqnarray*}
\igerm z & := & \underrightarrow{\lim}\left(\mathcal{B}_{r}\right)_{r>0}\,.
\end{eqnarray*}
Notice that this direct limit is the same as that obtained by considering
the countable family $\left(\mathcal{B}_{\nf 1k}\right)_{k\in\ww N_{>0}}$. 
\end{defn}
The fact that this topology is actually that of a locally convex one
is not straightforward, and can be deduced from a general result of
\noun{H.~Komatsu}~\cite[Theorem 6' p375]{Koma} regarding countable
inductive systems of Banach spaces and compact maps.
\begin{prop}
The topological space $\igerm z$ is naturally endowed with a structure
of a locally convex vector space.
\end{prop}
Such a locally convex structure will be described afterwards, when
we establish the fact that $\igerm z$ is homeomorphic to the locally
convex topology obtained by gathering all sequential norms (Proposition~\ref{prop:topo_compare}).
\begin{rem}
By construction a mapping $\Lambda\,:\,\igerm z\to X$, where $X$
is a topological space, is continuous if, and only if, for all $r>0$
the map $\Lambda\circ\underline{\bullet}\,:\,\left\{ r\right\} \times\mathcal{B}_{r}\to X$
is.
\end{rem}
It is possible to show that the inductive topology on $\igerm z$
is non-metrizable, ultrabornological and nuclear. It also satisfies
the next additional property, which will follow from its counterpart
for sequential spaces: Propositions~\ref{prop:topo_completion},~\ref{prop:convergence_radius}
and~\ref{prop:all_norms_is_useful}.
\begin{prop}
The space $\left(\igerm z,\cdot,+,\times,\pp{}z\right)$ is a compositing
differential algebra, which is complete but not Baire.
\end{prop}

\subsection{\label{sub:topo_anal}The analytic topology}

~

Another topology on $\germ z$ has been introduced by F.~\noun{Takens~}\cite{Takens}.
It is worth noticing that although $\germ z$ becomes a Baire space
it is not a topological vector space, and for that reason will not
be of great interest to us in the sequel; we only mention it for the
sake of being as complete as possible. This topology is spanned by
the following collection of neighborhoods of $0$ (and their images
by the translations of $\germ z$):
\begin{align*}
U_{\rho,\delta} & :=\left\{ f\in\germ z\,:\,\exists\left(\rho,f_{\rho}\right)\in f\mbox{ and }\norm{f_{\rho}}{\rho\ww D}<\delta\right\} \,\,\,\,\,,\,\,\,\,\,\rho,\,\delta>0\,.
\end{align*}
The resulting topological space is written $\germ z^{\omega}$. This
space is not a topological vector space since no $U_{\rho,\delta}$
is absorbing\footnote{If the radius of convergence of $f$ is strictly less than $\rho$
then $f$ does not belong to any $\lambda U_{\rho,\delta}$ whatever
$\lambda\in\ww C$ may be. }. Let us conclude this paragraph by mentioning the
\begin{prop}
$\germ z^{\omega}$ is a Baire space. The multiplication, differentiation
or the right-composition with a dilatation are not continuous operations. 
\end{prop}

\subsection{\label{sub:topo_proj}The projective topology}

~

Because we work in the holomorphic world the space $\germ z$ can
also be analyzed through the Taylor linear one-to-one (but not onto)
mapping 
\begin{eqnarray*}
\igerm z & \overset{\mathcal{T}}{\longrightarrow} & \frml z\\
\underline{\left(r,f_{r}\right)} & \longmapsto & \sum_{n\geq0}\frac{f_{r}^{\left(n\right)}\left(0\right)}{n!}z^{n}\,.
\end{eqnarray*}
The space $\germ z$ is therefore isomorphic to the sub-algebra of
the formal power series 
\begin{align*}
\frml z & :=\left\{ \sum_{n\geq0}f_{n}z^{n}\,:\,\left(f_{n}\right)_{n}\in\ww C^{\ww N}\right\} ~,
\end{align*}
which differs from $\ww C^{\ww N}$ by the choice of the Cauchy product
instead of the Hadamard (term-wise) product, characterized by the
condition 
\begin{eqnarray*}
 & \liminf\left|f_{n}\right|^{-\nf 1n}>0 & \,.
\end{eqnarray*}
Hadamard's formula stipulates that this value is nothing but the radius
of convergence $\rad{\sum_{n\geq0}f_{n}z^{n}}$ of the corresponding
germ. The latter is given by the evaluation map
\begin{eqnarray*}
\germ z & \overset{\mathcal{E}}{\longrightarrow} & \igerm z\\
f=\sum_{n\geq0}f_{n}z^{n} & \longmapsto & \underline{\left(\frac{\rad f}{1+\rad f}\,,\,z\mapsto\sum_{n=0}^{\infty}f_{n}z^{n}\right)}\,.
\end{eqnarray*}
There is no special meaning in taking $\frac{\rad f}{1+\rad f}$,
any finite number less than $\rad f$ would work (\emph{e.g. }$\min\left\{ 1,\frac{1}{2}\rad f\right\} $).

We can therefore equip $\germ z$ with the projective topology inherited
from $\frml z$ and defined by the $N^{\tx{th}}$-jet projectors:
\begin{eqnarray*}
J_{N}\,:\,\germ z & \longrightarrow & \pol z{\leq N}\\
\sum_{n\geq0}f_{n}z^{n} & \longmapsto & \sum_{n\leq N}f_{n}z^{n}
\end{eqnarray*}
where the topology on each $\pol z{\leq N}$ is the standard normed
one. It will also be convenient to introduce the Taylor-coefficient
map of degree $N$ as: 
\begin{align*}
T_{N}\,:\,\germ z & \longrightarrow\ww C\\
\sum_{n\geq0}f_{n}z^{n} & \longmapsto f_{N}
\end{align*}
so that $J_{N}\left(\bullet\right)=\sum_{n=0}^{N}T_{n}\left(\bullet\right)z^{n}$.
The induced topology is that of the product topology on $\prod_{N\geq0}\ww C$;
this coarseness renders the projective topology almost useless for
our purposes in this article\footnote{We also mention that $\germ z$ could be equipped with the restriction
of the normed topology offered by the Krull distance on $\frml z$,
but this topology is rougher yet and even less interesting.}. 
\begin{defn}
\label{def_projective_germ}The topological space obtained as the
restriction to $\germ z$ of the inverse limit $\underleftarrow{\lim}\left(\pol z{\leq N}\right)_{N\in\ww N}$,
equipped with the projective topology, will be denoted by 
\begin{eqnarray*}
\pgerm z & := & \underleftarrow{\lim}\left(\pol z{\leq N}\right)_{N\in\ww N}\cap\germ z\,.
\end{eqnarray*}
\end{defn}
\begin{prop}
$\left(\pgerm z,+,\times\right)$ is a non-Baire, non-complete topological
algebra. Neither is it a compositing differential algebra.\end{prop}
\begin{proof}
This space clearly is not complete. Besides the decomposition $\germ z=\bigcup_{N\in\ww N}F_{N}$,
where
\begin{align}
F_{N} & :=\left\{ \sum_{n\geq0}f_{n}z^{n}\,:\,\left|f_{n}\right|\leq N^{n}\right\} =\bigcap_{n\in\ww N}\left|T_{n}\right|^{-1}\left(\left[0,N^{n}\right]\right)\label{eq:F_N_not_Baire}
\end{align}
is a closed set with empty interior, shows the space cannot be Baire. 
\end{proof}

\subsection{\label{sub:topo_sequential}Sequential topologies}

~

We define a norm on $\germ z$ by making use of the Taylor coefficients
at $0$ of a germ at $0$. Being given a sequence $a=\left(a_{n}\right)_{n\in\ww N}$
of positive numbers we can formally define 
\begin{eqnarray*}
\norm{\sum_{n\geq0}f_{n}z^{n}}a & := & \sum_{n=0}^{\infty}a_{n}\left|f_{n}\right|.
\end{eqnarray*}
It is a genuine norm on $\germ z$ if, and only if, $a$ is asymptotically
sufficiently flat, \emph{i.e. }that it belongs to 
\begin{eqnarray*}
\mathcal{A} & := & \left\{ a\in\ww R_{>0}^{\ww Z_{\geq0}}~~:~~\lim_{n\to\infty}a_{n}^{\nf 1n}=0\right\} \,.
\end{eqnarray*}

\begin{defn}
\label{def_A_topo}~
\begin{enumerate}
\item For any $a\in\mathcal{A}$ the above norm will be called the \textbf{$a$-norm
}on $\germ z$.
\item The entire function
\begin{eqnarray*}
\car a\,:\,x\in\ww C & \longmapsto & \sum_{n=0}^{\infty}a_{n}x^{n}
\end{eqnarray*}
is called the \textbf{comparison function} of $a$. The \textbf{amplitude
}of $a$ is the function
\begin{eqnarray*}
\amp a\,:\,r>0 & \longmapsto & \max_{n\in\ww N}\left\{ a_{n}r^{n}\right\} <\infty.
\end{eqnarray*}

\item For every non-empty subset $A\subset\mathcal{A}$ we define the \textbf{$A$-topology}
of $\germ z$ as the normally convex topology associated to the family
of norms $\left(\norm{\bullet}a\right)_{a\in A}$. The topological
vector space $\left(\germ z,\left(\norm{\bullet}a\right)_{a\in A}\right)$
will be written $\germ z_{A}$. Such a topology is also called a \textbf{sequential
topology}.
\item Two collections $A,\,A'\subset\mathcal{A}$ will be deemed \textbf{equivalent}
if they induce equivalent topologies, that is the identity mapping
is a homeomorphism $\germ z_{A}\to\germ z_{A'}$. For all intents
and purposes we then say that both sequential topologies are the same
but we write $A\simeq A'$ for precision.
\end{enumerate}
\end{defn}
\begin{rem*}
With a similar construction it is possible to derive norms induced
by hermitian inner products on $\germ z\times\germ z$ as defined
by 
\begin{eqnarray*}
\left\langle \sum_{n\geq0}f_{n}z^{n},\sum_{n\geq0}g_{n}z^{n}\right\rangle _{a} & := & \sum_{n=0}^{\infty}a_{n}f_{n}\overline{g_{n}}\,.
\end{eqnarray*}
Since we have not felt the need to use the associated extra structure
this viewpoint provides, we will not particularly develop it (although
the results presented here should continue to hold).
\end{rem*}
Every $A$-topology is spanned by the family of finite intersections
of open $a$-balls of some radius $\varepsilon_{a}>0$ and center
$0$
\begin{eqnarray*}
\ball[a][0][\varepsilon_{a}] & := & \left\{ f\in\germ z:\norm fa<\varepsilon_{a}\right\} 
\end{eqnarray*}
for $a\in A$. We recall that a linear mapping $L\,:\,\germ z_{A}\to\germ w_{B}$
is continuous if, and only if, for all $b\in B$ there exists a finite
set $F\subset A$ and some $C_{b}\geq0$ such that
\begin{eqnarray*}
\left(\forall f\in\germ z\right) &  & \norm{L\left(f\right)}b\leq C_{b}\max_{a\in F}\norm fa.
\end{eqnarray*}
We denote by $\mathcal{L}\left(\germ z_{A}\to\germ w_{B}\right)$
the space of all linear, continuous mappings and endow it with the
natural locally convex topology induced by that of the source and
range spaces.

We finish this section by mentioning the following easy lemma, based
on the observation that if $f\left(z\right)=\sum_{n\ge\nu}f_{n}z^{n}\in\germ z$
with $f_{\nu}\neq0$ then for every $a\in\mathcal{A}$ and $g\in\ball[a][f][\frac{a_{\nu}}{\left|f_{\nu}\right|}]$
either the derivative $\frac{\dd{^{\nu}g}}{\dd{z^{\nu}}}\left(0\right)$
does not vanish or some derivative of lesser order does not.
\begin{lem}
\label{lem:valuation}For any $A$-topology on $\germ z$ the valuation
map 
\begin{align*}
\germ z_{A}\backslash\left\{ 0\right\}  & \longrightarrow\ww N\\
\sum_{n\geq0}f_{n}z^{n}\neq0 & \longmapsto\inf\left\{ n\text{ }:\text{ }f_{n}\neq0\right\} \text{ }
\end{align*}
 is lower semi-continuous. 
\end{lem}

\subsubsection{Naive polydiscs}

~
\begin{defn}
\label{def_naive_polydisc}Let $r:=\left(r_{n}\right)_{n\in\ww N}$
be a sequence of positive real numbers and take $f\in\germ z$. The
\textbf{naive-polydisc} of center $f$ and poly-radius $r$ is the
set
\begin{align*}
\polyd[f][r] & :=\left\{ g\in\germ z\,:\,\left(\forall n\in\ww N\right)\,\left|T_{n}\left(f-g\right)\right|<r_{n}\right\} \,.
\end{align*}
\end{defn}
\begin{lem}
\label{lem:naive-polyd_open} A naive-polydisc $\polyd[f][r]$ contains
an open $a$-ball of center $f$ if, and only if, 
\begin{align*}
\liminf_{n\to\infty}a_{n}r_{n} & >0\,.
\end{align*}
This particularly means that $\frac{1}{r}\in\mathcal{A}$ and $\lim r=\infty$.
Beside the above limit represents the maximum radius of $a$-balls
that can be included in $D\left(f,r\right)$.
\begin{proof}
We can suppose that $f=0$ and write $r=\left(r_{n}\right)_{n}$.
Assume first that there exists $\varepsilon>0$ such that $\ball[a][0][\varepsilon]\subset\polyd[0][r]$.
In particular for a given $n\in\ww N$ the polynomial $\eta z^{n}$
belongs to $\polyd[0][r]$ for every $\eta\in\ww C$ such that $\left|\eta\right|a_{n}<\varepsilon$,
which implies $\varepsilon\leq r_{n}a_{n}$. Conversely if $\liminf a_{n}r_{n}>0$
then there exists $\varepsilon>0$ such that $a_{n}r_{n}\geq\varepsilon$
for every $n\in\ww N$. If we choose $f\in\ball[][][\varepsilon]$
then for all integer $n$ we particularly have the estimate $a_{n}\left|f_{n}\right|\leq\varepsilon\leq a_{n}r_{n}$
so that $\ball[a][][\varepsilon]\subset\polyd$.
\end{proof}
\end{lem}

\subsubsection{Topological completion}

~

Any sequential topology induces a uniform structure on $\germ z$,
allowing to contemplate the notion of topological completeness\footnote{In general one only needs sequential completeness, and the corresponding
completed space will be named Cauchy's completion as opposed to Hausdorff's
topological completion of the uniform space $\germ z_{A}$.}. We give without proof the following statements.
\begin{prop}
\label{prop:topo_completion}Consider some $A$-topology on $\germ z$. 
\begin{enumerate}
\item The Cauchy (sequential) completion of $\germ z_{A}$ is canonically
isomorphic, as $\ww C$-algebras, to the locally convex subspace of
$\frml z$ defined by
\begin{eqnarray*}
\fgrm z_{A} & := & \bigcap_{a\in A}\left\{ \sum_{n\geq0}f_{n}z^{n}:\sum_{n=0}^{\infty}a_{n}\left|f_{n}\right|<\infty\right\} 
\end{eqnarray*}
endowed with the family of norms $\left(\norm{\bullet}a\right)_{a\in A}$.
\item Take $f=\sum_{n}f_{n}z^{n}\in\fgrm z_{A}$ and the associated sequence
of jets $J_{N}\left(f\right)=\sum_{n=0}^{N}f_{n}z^{n}$ as $N\in\ww N$.
Then for every $a\in A$ 
\begin{eqnarray*}
\lim_{N\to\infty}\norm{f-J_{N}\left(f\right)}a & = & 0.
\end{eqnarray*}
In particular the subspace of polynomials is dense in $\fgrm z_{A}$.
\item We have
\begin{align*}
\germ z & =\fgrm z_{\mathcal{A}}\,,
\end{align*}
which means the space $\germ z_{\mathcal{A}}$ is sequentially complete.
Besides no other (non-equivalent) $A$-topology can be.
\end{enumerate}
\end{prop}
\begin{rem}
When $A$ is at most countable the space $\fgrm z_{A}$ is a special
case of a Köthe sequential space~\cite{Kotkot}, some of which have
been extensively studied (\emph{e.g.} rapidly decreasing sequences)
in particular regarding the property of tameness to be used in Nash-Moser
local inversion theorem~\cite{Hami}. Unfortunately it is known that
the Köthe spaces presented here do not fulfill Nash-Moser's theorem
hypothesis although, as we see further down, some of them are nuclear. 
\end{rem}

\subsubsection{Radius of convergence and Baire property}

~

We relate now the Baire property on $\germ z_{A}$ to the absence
of a positive lower bound for $\rad{\bullet}$ on any domain. This
relationship was suggested by R.~\noun{Schäfke.}
\begin{prop}
\label{prop:convergence_radius}Fix some $A$-topology on $\germ z$.
\begin{enumerate}
\item $\mathcal{R}$ is upper semi-continuous on $\fgrm z_{A}$ and for
all non-empty open set $U\subset\germ z_{A}$ we have $\mathcal{R}\left(U\right)=]0,\infty]$.
In particular $\mathcal{R}$ can never be positively lower-bounded
on $U$.
\item The space $\germ z_{A}$ can never be Baire (in particular not Fréchet
when $A$ is countable).
\end{enumerate}
\end{prop}
\begin{proof}
~
\begin{enumerate}
\item Each map $H_{N}\,:\,\sum_{n\geq0}f_{n}z^{n}\mapsto\left|f_{N}\right|^{1/N}$
is continuous and ranges in $[0,\infty[$. As a consequence each member
of the sequence of functions indexed by $n\in\ww N$ and defined by
\begin{eqnarray*}
R_{n} & := & \sup_{N\geq n}H_{N}
\end{eqnarray*}
is lower semi-continuous. Therefore $\frac{1}{\rad{\bullet}}$ is
the limit of an increasing sequence of lower semi-continuous functions
and consequently is itself lower semi-continuous. According to Theorem~\ref{prop:topo_completion}~(2)
the value $\infty$ belongs to $\rad U$, so let us now take an arbitrary
$r>0$. For all $\lambda\in\ww C_{\neq0}$ the radius of convergence
of the power series $f_{\lambda}:=\lambda\sum_{n\geq0}r^{-n}z^{n}$
is precisely $r$. By taking $\lambda$ small enough and by picking
$f\in U\cap\pol z{}$ the germ $f+f_{\lambda}$ belongs to $U$ and
its radius of convergence is precisely $r$.
\item $\germ z$ is covered by the countable family of closed sets $F_{N}$
defined in~\eqref{eq:F_N_not_Baire}. Every one of these has empty
interior because of (1).
\end{enumerate}
\end{proof}

\subsection{Comparing the topologies}

~

We particularly prove that the inductive topology is homeomorphic
to the full $\mathcal{A}$-topology. This fact was mentioned to me
by Y.~\noun{Okada} (to which I am indebted also for having spotted
some mistakes in earlier versions of this manuscript).
\begin{prop}
\label{prop:topo_compare}Let $A\subset\mathcal{A}$ be given. We
have the following ordering of topologies 
\begin{eqnarray*}
\pgerm z & <~\germ z_{A}~\leq & \igerm z\,.
\end{eqnarray*}
Moreover $\igerm z$ and $\germ z_{A}$ are homeomorphic if, and only
if, $A\simeq\mathcal{A}$.
\end{prop}
On the one hand $\pgerm z\leq~\germ z_{A}$ because of the next obvious
lemma.
\begin{lem}
\label{lem:jet_is_continuous}The $N^{\tx{th}}$-jet projector $J_{N}\,:\,\germ z_{A}\to\pol z{\leq N}$
is continuous. More precisely for all $a\in\mathcal{A}$ and $f\in\germ z$
\begin{align*}
\norm{J_{N}\left(f\right)}a & \leq\norm fa\,.
\end{align*}

\end{lem}
The fact that $\pgerm z$ is not homeomorphic to $\germ z_{A}$ is
clear enough. On the other hand $\germ z_{A}\leq\igerm z$ because
of Cauchy's estimate, as stated below.
\begin{lem}
\label{lem:Taylor_is_continuous}~The Taylor map $\mathcal{T}\,:\,\igerm z\to\germ z_{A}$
is continuous. More precisely, for all $r>0$, all $f_{r}\in\mathcal{B}_{r}$
and all $a\in\mathcal{A}$: 
\begin{eqnarray*}
\norm{\mathcal{T}\left(f_{r}\right)}a & \leq & \norm{z\mapsto\frac{r}{r-z}}a\norm{f_{r}}{r\ww D}.
\end{eqnarray*}

\end{lem}
Because $\igerm z$ is complete and $\germ z_{A}$ is only when $A\simeq\mathcal{A}$,
the spaces cannot be homeomorphic if $A$ is not equivalent to $\mathcal{A}$.
We prove now that the evaluation mapping $\mathcal{E}\,:\,\germ z_{\mathcal{A}}\to\igerm z$
is continuous. I reproduce here the argument of \noun{Y.~Okada} with
his kind permission. 

First observe the inductive limit of nested linear subspaces $\left(\mathcal{X}_{\rho}\right)_{\rho>0}$
\begin{align*}
\mathcal{X}_{\rho} & :=\left\{ \sum_{n\geq0}f_{n}z^{n}\in\germ z~:~\sup_{n\geq0}\rho^{n}\left|f_{n}\right|<\infty\right\} 
\end{align*}
results in the same topological space $\igerm z$ (again it is sufficient
to take the countable direct system $\left(\mathcal{X}_{\nf 1k}\right)_{k\in\ww N_{>0}}$). 
\begin{lem}
For any balanced, convex neighborhood $U$ of $0$ in $\igerm z$
there exists a sequence $\varepsilon=\left(\varepsilon_{k}\right)_{k\in\ww N}$
of positive numbers such that
\begin{align*}
U\left(\varepsilon\right) & :=\sum_{k\geq0}\ball[k][0][\varepsilon_{k}]\subset U~,
\end{align*}
where $\ball[k][][\varepsilon_{k}]$ stands for the ball of $\mathcal{X}_{\nf 1{1+k}}$
of center $0$ and radius $\varepsilon_{k}$.\end{lem}
\begin{proof}
There clearly exists $\tilde{\varepsilon}_{k}>0$ such that $\ball[k][0][\tilde{\varepsilon}_{k}]\subset U$
since $\igerm z\simeq\underrightarrow{\lim}\mathcal{X}_{\nf 1{1+k}}$.
Because $U$ is convex and contains $0$ the inclusion
\begin{align*}
\sum_{k=0}^{N}2^{-k-1}\ball[k][0][\tilde{\varepsilon}_{k}] & \subset U
\end{align*}
holds for any $N\in\ww N_{\geq0}$. Simply take $\varepsilon_{k}:=2^{-k-1}\tilde{\varepsilon}_{k}$.
\end{proof}
The proposition's proof is completed with the next lemma.
\begin{lem}
There exists $a\in\mathcal{A}$ such that 
\begin{align*}
\ball[a][0][1] & \subset U\left(\varepsilon\right)~.
\end{align*}
\end{lem}
\begin{proof}
Without loss of generality we choose the sequence $\varepsilon$ of
the previous lemma in such a way that $\varepsilon_{k}\leq e^{-k-1}$.
For the sake of clarity write $r_{k}:=\frac{1}{k+1}$. Define for
$n\in\ww N$
\begin{align*}
a_{n} & :=2^{n+3}\inf_{k\geq0}\frac{r_{k}^{n}}{\varepsilon_{k}}
\end{align*}
and let us prove $\ball[a][0][1]\subset U\left(\varepsilon\right)$.
Notice first that by construction we have
\begin{align*}
\limsup_{n\to\infty}a_{n}^{\nf 1n} & \leq\frac{2}{k+1}~~~~~~~~~~~\left(\forall k\in\ww N\right)\\
a_{n} & \geq2^{n+3}\left(\frac{e}{n}\right)^{n}
\end{align*}
from which we deduce $a\in\mathcal{A}$. The latter estimates follows
from
\begin{align}
\frac{r_{k}^{n}}{\varepsilon_{k}}\geq\frac{e^{k+1}}{\left(k+1\right)^{n}}\geq & \left(\frac{e}{n}\right)^{n}~.\label{eq:estim_inductive}
\end{align}

Take now an arbitrary $f\in\ball[a][0][1]$ which we decompose as
\begin{align*}
f\left(z\right) & =\sum_{n<N}f_{n}z^{n}+\sum_{n\geq N}f_{n}z^{n}
\end{align*}
for a suitable $N$ to be determined later, and we shall prove that
both terms belong to $\frac{1}{2}U\left(\varepsilon\right)$. There
exists $k\in\ww N$ such that $f\in\mathcal{X}_{\nf 1{1+k}}$, \emph{i.e.}
\begin{align*}
\sup_{n\geq0}r_{k}^{n}\left|f_{n}\right| & =:M<\infty~.
\end{align*}
Therefore when $N$ is large enough we have
\begin{align*}
\sup_{n\geq N}r_{k+1}^{n}\left|f_{n}\right| & \leq M\left(\frac{k+1}{k+2}\right)^{N}<\frac{1}{2}\varepsilon_{k+1}~,
\end{align*}
so that $\sum_{n\geq N}f_{n}z^{n}\in\frac{1}{2}\ball[k+1][0][\varepsilon_{k+1}]\subset\frac{1}{2}U\left(\varepsilon\right)$
as claimed. We deal now with the other part $\hat{f}\left(z\right):=\sum_{n<N}f_{n}z^{n}$.

We deduce from~\eqref{eq:estim_inductive} the following statement:
for any $n\in\ww N$ there exist at least one $k\in\ww N$ for which
\begin{align*}
a_{n} & >2^{n+3}\frac{r_{k}^{n}}{\varepsilon_{k}}-2^{n+2}\left(\frac{e}{n}\right)^{n}\\
 & =2^{n+2}\frac{r_{k}^{n}}{\varepsilon_{k}}+2^{n+2}\left(\frac{r_{k}^{n}}{\varepsilon_{k}}-\left(\frac{e}{n}\right)^{n}\right)\\
 & \geq2^{n+2}\frac{r_{k}^{n}}{\varepsilon_{k}}~.
\end{align*}
For any $n<N$ pick such a $k=:k\left(n\right)$ and form the finite
set $K:=\left\{ k\left(n\right)~:~n<N\right\} $. For $k\in K$ define
the polynomial 
\begin{align*}
g_{k}\left(z\right) & :=\sum_{\begin{array}{c}
n<N\\
k=k\left(n\right)
\end{array}}f_{n}z^{n}
\end{align*}
 so that $\hat{f}\left(z\right)=\sum_{k\in K}g_{k}\left(z\right)$.
Since $\norm fa<1$ we particularly have $a_{n}\left|f_{n}\right|<1$,
hence for all $n<N$
\begin{align*}
r_{k\left(n\right)}^{n}\left|f_{n}\right| & <2^{-n-2}\varepsilon_{k\left(n\right)}~.
\end{align*}
Thus $g_{k}\in\frac{1}{2}\ball[k][0][\varepsilon_{k}]$ for each $k\in K$,
which finally yields $\hat{f}\in\frac{1}{2}\sum_{k\in K}\ball[k][0][\varepsilon_{k}]\subset\frac{1}{2}U\left(\varepsilon\right)$. 
\end{proof}

\subsection{\label{sub:topo_useful}Useful topologies}

~
\begin{defn}
\label{def_useful_topo}Let $A\subset\mathcal{A}$ be non-empty.
\begin{enumerate}
\item An $A$-topology is \textbf{finite} if it is equivalent to some $A'$-topology,
$A'$ being a finite set.
\item An $A$-topology is \textbf{useful} when $\left(\germ z_{A},\cdot,+,\times,\pp{}z\right)$
is a compositing differential topological algebra. 
\item The \textbf{factorial topology} is the metrizable sequential topology
spanned by the family 
\begin{align*}
\mathtt{AF} & :=\left\{ a\left(\frac{1}{k}\right)\,:\,k\in\ww N\right\} \,,\\
a\left(\alpha\right) & :=\left(n!^{-\alpha}\right)_{n\in\ww N}\,\,,\,\alpha\in\ww R_{>0}.
\end{align*}

\end{enumerate}
\end{defn}
\begin{rem}
\label{rem_facto_topo_vector}In case one wants to study the factorial
topology on $\germ{\mathbf{z}}$ with $\mathbf{z}=\left(z_{1},\cdots,z_{m}\right)$
then one defines for $\alpha>0$ 
\begin{align*}
\norm{\sum_{\mathbf{n}\in\ww N^{m}}f_{\mathbf{n}}\mathbf{z}^{\mathbf{n}}}{a\left(\alpha\right)} & :=\sum_{\mathbf{n}=\mathbf{0}}^{\mathbf{\infty}}\left|f_{\mathbf{n}}\right|\left(\mathbf{n}!\right)^{-\alpha}\text{ }.
\end{align*}

\end{rem}

\subsubsection{General properties of useful topologies}

~
\begin{prop}
\label{prop:all_norms_is_useful}The $\mathcal{A}$-topology is useful,
while no finite $A$-topology can be.\end{prop}
\begin{proof}
The fact that no finite $A$-topology may be useful has been hinted
at in the introduction of this section, the argument being the same
as in the normed case. The remaining of the proposition is mainly
a consequence of the next trivial lemma:
\begin{lem}
\label{lem:estim_for_useful}For every $a\in\mathcal{A}$ there exists
$b,\,c\in\mathcal{A}$ such that for all $m\in\ww N$ and all $\mathbf{j}\in\ww N^{m}$
we have $c_{n}\leq1$ and 
\begin{align*}
a_{\left|\mathbf{j}\right|} & \leq c_{m}\prod_{\ell=1}^{m}b_{j_{\ell}}\,.
\end{align*}

\end{lem}
Back to our proposition, let us first address the continuity of the
multiplication. For $a\in\mathcal{A}$: 
\begin{eqnarray*}
\norm{f\times g}a & = & \sum_{n=0}^{\infty}a_{n}\left|\sum_{p+q=n}f_{p}g_{q}\right|\,.
\end{eqnarray*}
According to the lemma there exist $b,\,c\in\mathcal{A}$ such that
$a_{n}\leq b_{p}b_{q}$ for every $p+q=n$, which means that
\begin{align*}
\norm{f\times g}a & \leq\norm fb\norm gb\,.
\end{align*}
Now consider the action of the derivation: 
\begin{align}
\pp fz & =\sum_{n\geq0}\left(n+1\right)f_{n+1}z^{n}\label{eq:formula_deriv}
\end{align}
so that for every $a,\,b\in\mathcal{A}$
\begin{align*}
\norm{\pp fz}a & =\sum_{n=0}^{\infty}a_{n}\left(n+1\right)\left|f_{n+1}\right|\leq\sum_{n=0}^{\infty}b_{n+1}\left|f_{n+1}\right|\times\frac{\left(n+1\right)a_{n}}{b_{n+1}}\,.
\end{align*}
We can always find $b$ such that the sequence $\left(\frac{\left(n+1\right)a_{n}}{b_{n+1}}\right)_{n\in\ww N}$
is bounded by some\footnote{Take for instance $b_{0}:=1$ and $b_{n+1}:=\sqrt{a_{n}}$.}
$C>0$ so that
\begin{align*}
\norm{\pp fz}a & \leq C\norm fb\,.
\end{align*}
We end the proof by using the composition formula, being given $g\in\germ z$,
\begin{eqnarray}
f\circ g & = & \sum_{n}\left(\sum_{m\leq n}f_{m}\sum_{\mathbf{j}\in\ww N^{m}\,,\,\left|\mathbf{j}\right|=n}\prod_{\ell=1}^{m}g_{j_{\ell}}\right)z^{n}\,.\label{eq:formula_compo}
\end{eqnarray}
For $a\in\mathcal{A}$ w e have
\begin{align*}
\norm{f\circ g}a & \leq\sum_{n=0}^{\infty}a_{n}\sum_{m\leq n}\left|f_{m}\right|\sum_{\mathbf{j}\in\ww N^{m}\,,\,\left|\mathbf{j}\right|=n}\prod_{\ell=1}^{m}\left|g_{j_{\ell}}\right|\,.
\end{align*}
Invoking once more the previous lemma we conclude the existence of
$b,\,c\in\mathcal{A}$ such that 
\begin{align*}
\norm{f\circ g}a & \leq\sum_{m=0}^{\infty}c_{m}\left|f_{m}\right|\sum_{n=0}^{\infty}\sum_{\mathbf{j}\in\ww N^{m}\,,\,\left|\mathbf{j}\right|=n}\prod_{\ell=1}^{m}b_{j_{\ell}}\left|g_{j_{\ell}}\right|\\
 & \leq\sum_{m=0}^{\infty}\sqrt{c_{m}}\left|f_{m}\right|\sqrt{c_{m}}\norm gb^{m}\leq\norm f{\sqrt{c}}\amp{\sqrt{c}}\left(\norm g{}\right)
\end{align*}
where the amplitude function $\amp{\bullet}$ is introduced in Definition~\ref{def_A_topo}.
\end{proof}
We show now a nice feature of useful topologies:
\begin{prop}
\label{prop:sequential_is_nuclear}Assume that the sequential topology
$\germ z_{A}$ is useful. Then $\germ z_{A}$ is a nuclear space.
In particular if $A$ is countable then $\fgrm z_{A}$ is a Montel
space (\emph{i.e. }every closed, bounded\footnote{We recall that a subset $\Omega$ of a locally convex space is bounded
if \emph{every} neighborhood of $0$ can be rescaled to contain $\Omega$.} set of $\fgrm z_{A}$ is compact).\end{prop}
\begin{proof}
To show that $\germ z_{A}$ is nuclear we must show that for every
$a\in A$ the natural embedding $\iota_{a}\,:\,\germ z_{A}\hookrightarrow\fgrm z_{a}$
is nuclear. By definition this boils down to proving the existence
of a sequence $\left(g_{n}\right)_{n\in\ww N}$ of $\fgrm z_{a}$
and a sequence $\left(\varphi_{n}\right)_{n\in\ww N}$ of continuous
linear forms $\germ z_{A}\to\ww C$ such that 
\begin{itemize}
\item there exists $b\in A$ with $\sum_{n=0}^{\infty}\norm{\varphi_{n}}b\norm{g_{n}}a<\infty$
\item $\iota_{a}\left(\bullet\right)=\sum_{n=0}^{\infty}\varphi_{n}\left(\bullet\right)g_{n}$
for the above normal convergence.
\end{itemize}
The natural choice is $g_{n}:=z^{n}$ and $\varphi_{n}:=T_{n}$, so
that $\norm{g_{n}}a=a_{n}$ and $\norm{T_{n}}b\leq\frac{1}{b_{n}}$.
Since $A$ is useful there exists $b\in A$ and $C>0$ such that for
all $f\in\germ z$
\begin{align*}
\norm{f\circ\left(2z\right)}a & \leq C\norm fb\,.
\end{align*}
Take now $f\left(z\right):=z^{n}$ so that for all $n\in\ww N$ 
\begin{align*}
2^{n}a_{n} & \leq Cb_{n}\,,
\end{align*}
and $\sum_{n=0}^{\infty}\frac{a_{n}}{b_{n}}<\infty$ as expected.
It is besides well known~\cite[Section 50]{Treves} that nuclear
Fréchet spaces are Montel spaces.
\end{proof}

\subsubsection{Factorial topology}

~

We study now the factorial topology, which we will use in most applications
because of its nice combinatorial and analytical properties. The choice
of the exponents $\nf 1k$ is rather arbitrary and one could chose
any strictly decreasing to zero sequence of positive numbers in what
follows. The topological completion of the space $\fgrm z_{\mathtt{AF}}$
corresponds to the Köthe sequence space of finite order $\Lambda_{0}\left(n!^{-\nf 1k}\right)$
which has been well studied, and the choice of exponent we make is
therefore <<standard>>\footnote{Another reason why this choice may be deemed natural is that $\fgrm z_{\mathtt{AF}}$
is the space of <<sub-Gevrey>> formal power series, that is those
power series $\sum_{n\geq0}f_{n}z^{n}$ such that $\sum_{n\geq0}\left(n!\right)^{-\alpha}f_{n}z^{n}$
converges for all $\alpha>0$.}. 
\begin{prop}
\label{prop:facto_useful}The factorial topology is useful. Moreover
the next three additional quantitative properties hold.
\begin{enumerate}
\item For every $f,\,g\in\germ z$ and all $\alpha>0$: 
\begin{align*}
\norm{f\times g}{a\left(\alpha\right)} & \leq\norm f{a\left(\alpha\right)}\norm g{a\left(\alpha\right)}\,.
\end{align*}

\item For all $\alpha>\beta>0$ and every $f,\,g\in\germ z$ with $g=0$:
\begin{eqnarray*}
\norm{f\circ g}{a\left(\alpha\right)} & \leq & \norm f{a\left(\beta\right)}\amp{a\left(\alpha-\beta\right)}\left(\norm g{a\left(\alpha\right)}\right)
\end{eqnarray*}
where the  amplitude $\amp{\bullet}$ is defined in Definition~\ref{def_A_topo}.
It is equal to $\sup_{n\in\ww N}n!^{\beta-\alpha}\norm g{a\left(\alpha\right)}^{n}$,
whose maximum is reached for some rank 
\begin{align*}
n_{0} & \in\left\{ -1,0,1\right\} +\left\lceil \norm g{a\left(\alpha\right)}^{\nf 1{\alpha-\beta}}\right\rceil \,.
\end{align*}
This constant is optimal.
\item For all $\alpha>\beta>0$ there exists a sequence of positive real
numbers $D_{k,\alpha,\beta}$ such that for every $f\in\germ z$:
\begin{align*}
\norm{\frac{\partial^{k}f}{\partial z^{k}}}{a\left(\alpha\right)} & \leq D_{k,\alpha,\beta}\norm{f-J_{k}\left(f\right)}{a\left(\beta\right)}\,.
\end{align*}
The optimal constant is given by 
\begin{align*}
D_{k,\alpha,\beta} & =\sup_{n\in\ww N}\frac{\left(n+k\right)!^{\beta+1}}{n!^{\alpha+1}}
\end{align*}
whose maximum is reached for some rank 
\begin{align*}
n_{0} & \in\left\{ -1,0,1\right\} +\left\lceil r_{k,\alpha,\beta}\right\rceil \,,
\end{align*}
where $r_{k,\alpha,\beta}$ is the positive solution of the equation
$\left(x+k\right)^{\beta+1}=x^{\alpha+1}$. We also have
\begin{align*}
r_{k,\alpha,\beta}\sim_{k\to\infty}k^{\nf{\beta+1}{\alpha+1}} & \,\mbox{and }r_{k,\alpha,\beta}>k^{\nf{\beta+1}{\alpha+1}}\,.
\end{align*}
Besides there exists $\chi_{k,\alpha,\beta}\in]r_{k,\alpha,\beta},r_{k,\alpha,\beta}+1[$
such that for all $k$, writing $\chi$ instead of $\chi_{k,\alpha,\beta}$,
\begin{align*}
D_{k,\alpha,\beta} & \leq\frac{e^{\beta+1}}{\left(2\pi\right)^{\nf{\alpha+1}2}}\chi^{k\left(\alpha+1\right)}\exp\left(\left(\alpha-\beta\right)\chi-\left(\beta+1\right)k\right)\,.
\end{align*}
One has 
\begin{align*}
k^{\nf{\beta+1}{\alpha+1}}<\chi_{k,\alpha,\beta} & <e^{\sqrt{\nf{\beta+1}{\alpha-\beta}}}k^{\nf{\beta+1}{\alpha+1}}+1\,.
\end{align*}

\end{enumerate}
\end{prop}
\begin{proof}
The proof first relies on the following trivial combinatorial estimate,
which is a counterpart to Lemma~\ref{lem:estim_for_useful}:
\begin{lem}
\label{lem:facto_estim}Let $m\in\ww N_{>0}$ and $\mathbf{j}\in\ww N^{m}$.
Then
\begin{eqnarray*}
\left|\mathbf{j}\right|! & \geq & \mathbf{j}!
\end{eqnarray*}
and if moreover $j_{\ell}>0$ for all $\ell$ we have 
\begin{eqnarray*}
\left|\mathbf{j}\right|! & \geq & m!\mathbf{j}!\,.
\end{eqnarray*}

\end{lem}
The addition and scalar multiplication are of course continuous. The
three estimates (1), (2) and (3) guarantee the usefulness of the factorial
topology. Take two elements $f$ and $g$ in $\germ z$ and write
them down respectively as $\sum_{n}f_{n}z^{n}$ and $\sum_{n}g_{n}z^{n}$. 
\begin{enumerate}
\item According to the lemma we have $\left(p+q\right)!\geq p!q!$ so that
\begin{eqnarray*}
\norm{f\times g}{a\left(\alpha\right)} & = & \sum_{n=0}^{\infty}n!^{-\alpha}\left|\sum_{p+q=n}f_{p}g_{q}\right|\\
 & \leq & \sum_{n=0}^{\infty}\sum_{p+q=n}p!^{-\alpha}q!^{-\alpha}\left|f_{p}g_{q}\right|=\norm f{a\left(\alpha\right)}\norm g{a\left(\alpha\right)}\,.
\end{eqnarray*}

\item Because of the lemma we can show in the same spirit as Proposition~\ref{prop:all_norms_is_useful}
that 
\begin{eqnarray*}
\norm{f\circ g}{a\left(\alpha\right)} & \leq & \norm f{a\left(\beta\right)}\sup_{m\in\ww N}m!^{\beta-\alpha}\norm g{a\left(\alpha\right)}^{m}\,.
\end{eqnarray*}
We can try to estimate $\amp{a\left(\beta-\alpha\right)}\left(\norm g{a\left(\alpha\right)}\right)$
using the Gamma and Digamma functions, through the study of the auxiliary
function
\begin{align*}
\phi\,:\,\ww R_{>1} & \longrightarrow\ww R\\
x & \longmapsto\Gamma\left(x\right)^{-\delta}r^{x}
\end{align*}
 where $\delta,\,r>0$ are given. This function admits a unique maximum,
located at the zero of its logarithmic derivative (expressed in terms
of the Digamma function $\psi=\frac{\Gamma'}{\Gamma}$)
\begin{align*}
\frac{\phi'}{\phi}\left(x\right) & =-\delta\psi\left(x\right)+\ln r\,.
\end{align*}
It is well known that for $x>1$
\begin{align}
\ln\left(x-1\right) & \leq\psi\left(x\right)\leq\ln x\label{eq:digamma_estim}
\end{align}
so that the equality $\psi\left(x\right)=\frac{1}{\delta}\ln r$ happens
only when
\begin{align*}
x & \in\left[r^{1/\delta},r^{1/\delta}+1\right]\,.
\end{align*}
Since this interval has length $1$ the result follows by setting
$r:=\norm g{a\left(\alpha\right)}$ and $\delta:=\alpha-\beta$.
\item We apply formula~\eqref{eq:formula_deriv} repeatedly so that
\begin{align*}
\norm{\pp{^{k}f}{z^{k}}}{a\left(\alpha\right)} & =\sum_{n=0}^{\infty}a\left(\alpha\right)_{n}\frac{\left(n+k\right)!}{n!}\left|f_{n+k}\right|\\
 & \leq\sum_{n=0}^{\infty}\left(n+k\right)!^{-\beta}\left|f_{n+k}\right|\times\frac{\left(n+k\right)!^{\beta+1}}{n!^{\alpha+1}}\\
 & \leq D_{k,\alpha,\beta}\norm{f-J_{k}\left(f\right)}{a\left(\beta\right)}\,.
\end{align*}
The constant is optimal: take for $f$ the monomial $z^{d}$ where
$d$ is any integer such that $\frac{\left(d+k\right)!^{\beta+1}}{d!^{\alpha+1}}$
equals $D_{k,\alpha,\beta}$. In order to derive the final estimate
we need to determine $d$; for this we study the auxiliary function
\begin{align*}
\varphi\,:\,\ww R_{>1} & \longrightarrow\ww R\\
x & \longmapsto\frac{\Gamma\left(x+k\right)^{\beta+1}}{\Gamma\left(x\right)^{\alpha+1}}
\end{align*}
much in the same way as we did just before. This function admits a
unique maximum, which is the positive zero of its logarithmic derivative
\begin{align*}
\frac{\varphi'}{\varphi}\left(x\right) & =\left(\beta+1\right)\psi\left(x+k\right)-\left(\alpha+1\right)\psi\left(x\right)\\
 & =\left(\beta-\alpha\right)\psi\left(x\right)+\left(\beta+1\right)\sum_{j=0}^{k-1}\frac{1}{x+j}\,.
\end{align*}
We therefore seek the positive real $\chi$ such that
\begin{align}
\delta\psi\left(\chi\right) & =\sum_{j=0}^{k-1}\frac{1}{\chi+j}\label{eq:digamma_eq}
\end{align}
where 
\begin{align*}
\delta: & =\frac{\alpha-\beta}{\beta+1}>0\\
\delta+1 & =\frac{\alpha+1}{\beta+1}\,.
\end{align*}
It is completely elementary that for $x>1$ we have 
\begin{align*}
\ln\frac{x+k}{x} & <\sum_{j=0}^{k-1}\frac{1}{x+j}<\ln\frac{x+k-1}{x-1}\,.
\end{align*}
This estimate, coupled with~\eqref{eq:digamma_estim}, implies particularly
that the equality~\eqref{eq:digamma_eq} is possible only if $\ln\frac{x+k}{x}<\delta\ln x$
and $\ln\frac{x+k-1}{x-1}<\delta\ln\left(x-1\right)$. As a consequence
$x_{*}$ lies in the interval $]r,r+1[$ where $r$ is the positive
solution of
\begin{align*}
\left(x+k\right)^{\beta+1} & =x^{\alpha+1}\,
\end{align*}
or, equivalently,
\begin{align*}
x^{\delta+1}-x & =k\,.
\end{align*}
Since this interval has length $1$ one has
\begin{align*}
d & \in\left\{ -1,0,1\right\} +\left\lceil r\right\rceil \,.
\end{align*}
Plugging $x:=k^{\nf 1{\delta+1}}$ in the previous expression shows
that 
\begin{align*}
k^{\nf 1{\delta+1}} & <r
\end{align*}
(in particular the sequence $r$ is unbounded). The relation $\frac{k}{r}=r^{\delta}-1$
proves $r=o\left(k\right)$. Both relations hold also for $\chi$
instead of $r$. Moreover let $\sigma>0$ be given such that
\begin{align*}
\sigma\geq e^{\nf 1{\sqrt{\delta}}}>1 & \,.
\end{align*}
Then 
\begin{align*}
\delta & \geq\frac{1}{\left(\ln\sigma\right)^{2}}>\frac{1}{\sigma\ln\sigma}\geq\frac{\ln\left(1+\nf 1{\sigma}\right)}{\ln\sigma}\\
\delta+1 & >\frac{\ln\left(\sigma+1\right)}{\ln\sigma}\\
1 & >\frac{\sigma}{\sigma^{\delta+1}-1}\\
k\geq1 & >\left(\frac{\sigma}{\sigma^{\delta+1}-1}\right)^{\nf{\delta+1}{\delta}}\\
k\left(\sigma^{\delta+1}-1\right) & >\sigma k^{\nf 1{\delta+1}}\\
\left(\sigma k^{\nf 1{\delta+1}}\right)^{\delta+1}-\sigma k^{\nf 1{\delta+1}} & >k
\end{align*}
so that 
\begin{align*}
r & <\sigma k^{\nf 1{\delta+1}}\,.
\end{align*}
We use finally the classical relation for $x>0$
\begin{align}
\sqrt{2\pi}x^{x-\nf 12}e^{-x}\leq & \Gamma\left(x\right)\leq e^{1-x}x^{x-\nf 12}\label{eq:stirling_estim}
\end{align}
and the fact that $\chi^{\alpha+1}>\left(\chi+k\right)^{\beta+1}$
to derive
\begin{align*}
\varphi\left(\chi\right) & \leq\frac{e^{\beta+1}}{\left(2\pi\right)^{\nf{\alpha+1}2}}\chi^{k\left(\alpha+1\right)}\exp\left(\left(\alpha-\beta\right)\chi-\left(\beta+1\right)k\right)\,.
\end{align*}

\end{enumerate}
\end{proof}

\section{\label{sec:usual_holo}Analyticity}

In this section we consider two Hausdorff, locally convex spaces $\left(E,\left(\norm{\bullet}{a\in A}\right)\right)$
and $\left(F,\left(\norm{\bullet}{b\in B}\right)\right)$, \emph{i.e.
}topological linear spaces whose topologies are induced by a collection
of separating semi-norms. We begin with giving general definitions
and properties, summarizing some important results of the references
\cite{Barro} and \cite{Mazette}. We then present specifics for the
spaces $\germ z_{A}$ and $\igerm z$, as well as introducing the
notion of strong analyticity. Notice that because of Proposition~\ref{prop:topo_compare}
and Definition~\ref{def_anal} below, as soon as a map $\Lambda\text{ }:\text{ }\germ z\to\germ x$
is proved analytic for some sequential $A$-topology, which is comparatively
an easier task to perform, it will automatically be analytic for the
inductive topology.

A notion we will need is that of ample boundedness, which mimics the
criterion for the continuity of linear mappings. 
\begin{defn}
\label{def_amply_bounded}A mapping $\Lambda~:~U\to F$ from an open
set $U$ of $E$ is \textbf{amply bounded} if for every $f\in U$
and all $b\in B$ there exists a neighborhood $W\subset E$ of $0$
such that 
\begin{eqnarray*}
\sup_{h\in W}\norm{\Lambda\left(f+h\right)}b & < & \infty.
\end{eqnarray*}

\end{defn}
Continuous mappings are amply bounded.

\subsection{Polynomials, power series and analyticity}

~
\begin{defn}
\label{def_anal}~
\begin{enumerate}
\item A \textbf{polynomial} on $E$ of degree at most $d$ with values in
$F$ is a finite sum 
\begin{eqnarray*}
P\left(f\right) & = & \sum_{p=0}^{d}P_{n}\left(f^{\oplus p}\right)
\end{eqnarray*}
of continuous, symmetric $p$-linear mappings $P_{p}\in\mathcal{L}_{p}\left(E\to F\right)$.
The least $d$ for which the above expansion holds for $P\neq0$ is
its degree $\deg P$. As usual we conventionally set $\deg0:=-\infty$,
and recall that 
\begin{align*}
f^{\oplus p} & =\left(f,\cdots,f\right)\in E^{p}\,.
\end{align*}

\item A \textbf{formal power series} $\Phi$ from $E$ to $F$ is a series
built from a sequence of continuous, symmetric \textbf{$p$}-linear
mappings $\left(P_{p}\right)_{p\in\ww N}\in\prod_{p\in\ww N}\mathcal{L}_{p}\left(E\to F\right)$
:
\begin{eqnarray*}
\Phi\left(f\right) & := & \sum_{p\geq0}P_{p}\left(f^{\oplus p}\right).
\end{eqnarray*}
The space of all such objects is a complex algebra with the standard
sum and Cauchy product operations on series.
\item A mapping $\Lambda$ defined on some neighborhood $U$ of $f\in E$
, with values in $F$, is said to be \textbf{analytic at $f$} if
for all $b\in B$ there exists a neighborhood $W$ of $0\in E$ such
that 
\begin{eqnarray*}
\lim_{N\to\infty}\norm{\Lambda\left(f+h\right)-\sum_{p=0}^{N}P_{p}\left(h^{\oplus p}\right)}b & = & 0
\end{eqnarray*}
uniformly in $h\in W$.
\item We say that $\Lambda$ is \textbf{analytic} on $U$ if it is analytic
at any point of $U$. If $\Lambda$ is analytic on the whole $E$
we say that $\Lambda$ is \textbf{entire}.
\item The $p$-linear mapping $p!P_{p}$ is called the $p$-th \textbf{Taylor
coefficient} of $\Lambda$ at $f$ and is written as
\begin{eqnarray*}
\p[p]{\Lambda}ff & := & p!P_{p}.
\end{eqnarray*}
The formal power series $\sum_{p\geq0}\frac{1}{p!}\p[p]{\Lambda}ff$
is called the \textbf{Taylor series} of $\Lambda$ at $f$.
\end{enumerate}
\end{defn}
\begin{thm}
\label{thm:anal_gene_prop} Let $U\subset E$ be an open set. In the
following hatted spaces denotes their Hausdorff topological completion
(as uniform spaces).
\begin{enumerate}
\item \emph{(}\textbf{\emph{\cite[p177]{Barro}}}\emph{)} If $\Lambda$
is analytic at some point $f\in U$ then its Taylor coefficients are
unique. In that case $\Lambda\left(f\right)=\p[0]{\Lambda}ff$.
\item \emph{(}\cite[p195]{Barro}\emph{)} Analytic mappings $U\to F$ are
continuous. 
\item \emph{(}\cite[p196]{Barro}\emph{)} An amply bounded mapping $U\to F$
is analytic if, and only if, for every $f\in U$ there exists a sequence
$\left(P_{p}\right)_{p\in\ww N}$ of $p$-linear mappings, not necessarily
continuous but satisfying nonetheless the rest of condition (3) in
the above definition. Particularly in that case each $P_{p}$ is \emph{a
posteriori} continuous.
\item \emph{(}\cite[p187]{Barro}\emph{)} Polynomials with values in $F$
are entire. If $\Phi$ is a polynomial then $\p[p]{\Phi}ff=0$ for
all $p>\deg P$ and $f\in E$.
\item \emph{(}\cite[p192]{Barro}\emph{)} Let $\Lambda:U\to F$ be analytic
and $\Xi:F\to G$ be a continuous linear mapping with values in a
normally convex space $G$. Then $\Xi\circ\Lambda$ is analytic and
\begin{eqnarray*}
\p[p]{\Xi\circ\Lambda}ff & = & \Xi\circ\p[p]{\Lambda}ff.
\end{eqnarray*}

\item \emph{(}\cite[p63]{Mazette}\emph{)} More generally the composition
of source/range compatible analytic maps remains analytic.
\item \emph{(}\cite[p68]{Mazette}\emph{) }Any analytic map $\Lambda\,:\,U\to F$
extends in a unique fashion to an analytic map $\hat{\Lambda}\,:\,\hat{U}\to\hat{F}$
from an open set of $\hat{E}$. Notice that $\hat{U}$ might not be
a topological completion of $U$.
\end{enumerate}
\end{thm}
\begin{quote}
Notice that from (1) is deduced the identity theorem, using a standard
connectedness argument:\end{quote}
\begin{cor}
\label{cor:anal_PPA}\textbf{\emph{(Identity Theorem \cite[p241]{Barro})}}
Let $U\subset E$ be a domain and $\Lambda$ analytic on $U$ such
that $\Lambda$ vanishes on some open subset of $U$. Then $\Lambda$
is the zero map.
\end{cor}

\subsection{\label{sub:anal_space}Analytical spaces}

~

As usual with analyticity, the definition is given for maps defined
on open sets. Yet in practice we would like to speak of analyticity
of maps defined on sets with empty interior (\emph{e.g.} proper analytic
sets). Although there exists a notion of analytic variety~\cite{Mazette}
we do not want to venture into those territories. As we will be in
general able to parameterize the sets of interest to us, we introduce
instead the notion of induced analytic structure:
\begin{defn}
\label{def_anal_space}Let $X$ be a topological space, $\Psi\,:\,U\subset E\twoheadrightarrow X$
be continuous and onto from an open set $U$ of $E$. 
\begin{enumerate}
\item We say that a map $\Lambda\,:\,X\to F$ is \textbf{analytic with respect
to the analytic structure induced by} $\Psi$ if $\Psi^{*}\Lambda=\Lambda\circ\Psi\,:\,U\to F$
is analytic.
\item Let $V\subset F$ be an open set. We say that a map $\Lambda\,:\,V\to X$
is \textbf{analytic with respect to the analytic structure induced
by} $\Psi$ if there exists an analytic map $\tilde{\Lambda}\,:\,V\to U$
such that $\tilde{\Lambda}^{*}\Psi=\Lambda$.
\item We say in this case that $X$ is an \textbf{analytic space} \textbf{modeled
on} $E$, equipped with the analytic structure induced by $\Psi$.
\item These definitions allow to speak of analytic maps between analytic
spaces $X$ and $Y$, each one equipped with an induced analytic structure,
respectively $\Psi_{X}~:~U_{X}\twoheadrightarrow X$ and $\Psi_{Y}~:~U_{Y}\twoheadrightarrow Y$.
A map $\Lambda\,:\,X\to Y$ is \textbf{analytic} if there exists an
analytic map $\tilde{\Lambda}\,:\,U_{X}\to U_{Y}$ such that the following
diagram commutes: 
\begin{align*}
\xymatrix{X\ar[r]^{\Lambda} & Y\\
U_{X}\ar[r]_{\tilde{\Lambda}}\ar@{->>}[u]^{\Psi_{X}} & U_{Y}\ar@{->>}[u]_{\Psi_{Y}}
}
\end{align*}

\end{enumerate}
\end{defn}
Basic examples include vector subspaces $F<E$ with a continuous projector
$E\twoheadrightarrow F$ and, more generally, quotient spaces with
respect to equivalence relations $\sim$ whose canonical map $E\twoheadrightarrow\nf E{\sim}$
is continuous. Also the space of diffeomorphisms $\diff{}$ is the
range of the continuous map (see Proposition~\ref{prop:compo_is_ana})
\begin{align*}
\Psi~:~\germ z & \longrightarrow\diff{}\subset\germ z\\
f & \longmapsto\left(z\mapsto z\exp f\left(z\right)\right)
\end{align*}
and is therefore an analytic space for the analytic structure induced
by $\Psi$. According to Proposition~\ref{prop:compo_is_ana} the
group law
\begin{align*}
\left(\Delta_{1},\Delta_{2}\right)\in\diff{}^{2} & \longmapsto\Delta_{1}\circ\Delta_{2}\in\diff{}
\end{align*}
is continuous since it corresponds \emph{via }$\Psi$ to the entire
mapping
\begin{align*}
\left(f_{1},f_{2}\right) & \longmapsto f_{2}+f_{1}\circ\left(z\exp f_{2}\right)~.
\end{align*}
Corresponding statements for the $m$-dimensional case follow.

\subsection{Cauchy integrals and estimates}

~

In the finite-dimensional setting a key element of the theory of analytic
functions is the Cauchy's formula and related estimates. The infinite-dimensional
setting is no exception to that rule and those estimates likewise
play a major role in all the theory. If $\Lambda$ is analytic on
some open set $U$ then, for fixed $f\in U$ and $h$ small enough
we can pick $\eta>0$ in such a way that the mapping 
\begin{eqnarray*}
x\in\eta\ww D & \longmapsto & \Lambda\left(f+xh\right)\in F
\end{eqnarray*}
is also analytic (for the standard normed topology on $\ww C$). In
particular it is continuous and the integral $\oint_{\left|x\right|=\eta}\Lambda\left(f+xh\right)\frac{\dd x}{x^{p+1}}$
makes perfectly sense in some topological completion of $F$. Yet
the very nature of the formula below implies \emph{a posteriori} that
this integral belongs to $F$.
\begin{prop}
\label{prop:anal_Cauchy_integ} \cite[p210-212]{Barro} Let $\Lambda~:~U\to F$
be analytic and pick $f\in U$. For every $h\in E$, $p\in\ww N$
and $\eta>0$ such that $f+\eta\overline{\ww D}h\subset U$ we have
\begin{eqnarray*}
\p[p]{\Lambda}ff\left(h^{\oplus p}\right) & = & \frac{1}{2\ii\pi}\oint_{\eta\sone}\Lambda\left(f+xh\right)\frac{\dd x}{x^{p+1}}.
\end{eqnarray*}
Besides if $a\in A$ and $b\in B$ then for all $h\in E$ and all
$\eta>0$ so that $\ball[a][f][\eta]\subset U$ we have\footnote{We point out that the right-hand side can be infinite if the choice
of $a,b$ is not done so that $\norm{\Lambda}b$ is bounded on $\ball[a][f][\eta]$.} 
\begin{eqnarray*}
\norm{\frac{1}{p!}\p[p]{\Lambda}ff\left(h^{\oplus p}\right)}b & \leq & \frac{\norm ha^{p}}{\eta^{p}}\sup_{\norm ua=\eta}\norm{\Lambda\left(f+u\right)}b\,.
\end{eqnarray*}

\end{prop}
Notice that a variation on this presentation of Cauchy's formula can
be used to define completely the values of the $p$-symmetric linear
mapping (see \textbf{\emph{\cite[p229]{Barro}}})
\begin{align}
\p[p]{\Lambda}ff\left(h_{1},\cdots,h_{p}\right) & =\frac{1}{2\ii\pi}\oint_{\nf{\eta}p\ww S^{1}}\cdots\oint_{\nf{\eta}p\ww S^{1}}\Lambda\left(f+\sum_{j=1}^{p}x_{j}h_{j}\right)\prod_{j=1}^{p}\frac{\dd x_{j}}{x_{j}^{2}}\,,\label{eq:Cauchy_multilin}
\end{align}
so that
\begin{align}
\norm{\frac{1}{p!}\p[p]{\Lambda}ff\left(h_{1},\cdots,h_{p}\right)}b & \leq\frac{p^{p}}{p!}\times\frac{\prod_{j=1}^{p}\norm{h_{j}}a}{\eta^{p}}\sup_{\norm ua\leq\eta}\norm{\Lambda\left(f+u\right)}b\,.\label{eq:estim_Cauchy_multilin}
\end{align}
This estimate is optimal, and can still be useful for studying convergence
of power series since $\frac{p^{p}}{p!}$ is sub-geometric.

\subsection{Fréchet- and Gâteaux-holomorphy}

~
\begin{defn}
\label{def_G-_F-_holo}Let $\Lambda~:~U\to F$ be a mapping defined
on a non-empty open set $U\subset E$. In the following $\hat{F}$
represents the topological Hausdorff completion of $F$ (as a uniform
space).
\begin{enumerate}
\item We say that it is \textbf{G-holomorphic} (meaning Gâteaux-holomorphic)
if for all finite dimensional linear subspace $S<E$ the restriction
$\Lambda|_{S\cap U}$ is analytic. This is equivalent (\cite[p51]{Mazette})
to requiring that for all $f\in U$ and $h\in E$ the map $x\mapsto\Lambda\left(f+xh\right)$
be analytic at $0\in\ww C$.
\item We say that it is \textbf{F-holomorphic} (meaning Fréchet-holomorphic)
if it is differentiable at every point $f\in U$ in the following
sense : there exists a continuous linear map $\partial_{f}\Lambda\in\mathcal{L}\left(E\to\hat{F}\right)$
such that for all $b\in B$ we can find $a\in A$ and some function
$\epsilon\,:\,U-f\to\ww R_{\geq0}$ with

\begin{itemize}
\item for all $h\in U-f$ we have 
\begin{eqnarray*}
\norm{\Lambda\left(f+h\right)-\Lambda\left(f\right)-\partial_{f}\Lambda\left(h\right)}b & \leq & \varepsilon\left(h\right)\norm ha\,,
\end{eqnarray*}

\item $\lim_{h\to0}\varepsilon\left(h\right)=0$.
\end{itemize}
\end{enumerate}
\end{defn}
\begin{rem}
If $\Lambda$ is analytic then it is F-holomorphic and
\begin{eqnarray*}
\p[1]{\Lambda}ff & = & \partial_{f}\Lambda.
\end{eqnarray*}
If $\Lambda$ is F-holomorphic then it is continuous.\end{rem}
\begin{thm}
\label{thm:anal_G_F_eqv}\cite[p246]{Barro}\cite[p62]{Mazette}Consider
a mapping $\Lambda~:~U\to F$ with $U$ an open set of $E$. The following
propositions are equivalent
\begin{enumerate}
\item $\Lambda$ is analytic,
\item $\Lambda$ is amply bounded and G-holomorphic,
\item $\Lambda$ is F-holomorphic.
\end{enumerate}
\end{thm}
\begin{rem}
Non-continuous linear maps provide examples of G-holomorphic functions
which are not analytic. It is likewise possible to build two G-holomorphic
mappings whose composition is not G-holomorphic anymore (\cite[p63]{Mazette}).
In view of the results presented earlier in this section these maps
obviously fail to be amply bounded.
\end{rem}

\subsection{Multivariate analyticity}

~

Many good properties of finite dimensional complex analysis persist.
In particular we dispose of Hartogs theorems regarding separate analyticity
for analytic mappings defined on a Cartesian product of finitely many
locally convex spaces equipped with the product topology. This is
indeed a direct consequence of the characterization of Theorem~\ref{thm:anal_G_F_eqv}.
\begin{thm}
\label{thm:anal_Hartogs}\textbf{\emph{(Hartogs lemma)}} Let $\Lambda\,:\,W_{1}\times\cdots\times W_{m}\to F$
be a map from a product of open sets in locally convex spaces. Then
$\Lambda$ is analytic if, and only if, $\Lambda$ is separately analytic,
meaning that for all choices of $p\in\left\{ 1,\cdots,m\right\} $
and of $\mathbf{f}\in\prod_{j}W_{j}$ the mapping $f\mapsto\Lambda\left(\mathbf{f}\haat pf\right)$
is analytic on $W_{p}$.
\end{thm}

\subsection{The case of ultrabornological spaces}

~

The following universal characterization will be useful in practice:
\begin{thm}
\textbf{\emph{\label{thm:inductive_anal}(Universal property of the
analytic inductive limit \cite[p254]{Barro})}} Let $E=\lim_{\rightarrow}\left(E_{k}\right)_{k\in\ww N}$
be ultrabornological, each $E_{k}$ being endowed with its Banach
topology; denote by $\iota_{r}\,:\,E_{r}\hookrightarrow E$ the canonical
(continuous) embedding. Then a map $\Lambda\,:\,U\subset E\to F$,
where $F$ is any locally convex space, is analytic if, and only if,
every $\Lambda\circ\iota_{r}$ is analytic for $r>0$.
\end{thm}

\subsection{\label{sub:analyticity_germs}Specifics for spaces of germs}

~

From now on $\left(E,\left(\norm{\bullet}a\right)_{a\in A}\right)$
(\emph{resp.} $\left(F,\left(\norm{\bullet}b\right)_{b\in B}\right)$)
stands for the vector space $\germ z$ (\emph{resp. $\germ w$}) endowed
with a Hausdorff, locally convex topology such as a sequential topology
or the inductive topology. Before diving into more details, let us
first state once and for all that the analytic structure put on the
finite-dimensional vector subspaces $\pol z{\leq N}\simeq\ww C^{N+1}$
is that induced, in the sense of Definition~\ref{def_anal_space},
by the (continuous) jets projectors $J_{N}\,:\,\germ z\twoheadrightarrow\pol z{\leq N}$.
It coincides with the standard analytic structure, and the space of
analytic functions with respect to $J_{N}$ is the same as the usual
holomorphic functions on open sets of $\ww C^{N+1}$.

\subsubsection{Uniform convergence of the Taylor series}

~
\begin{prop}
\label{prop:ana_cvg_norm_unif} Let $\Lambda$ be analytic on a neighborhood
of some $f\in E$, with Taylor coefficients $\Lambda_{p}\in\mathcal{L}_{p}\left(E\to F\right)$.
For any $b\in B$ there exists $a\in A$ such that the formal power
series $\sum_{p}\Lambda_{p}\left(h^{\oplus p}\right)$ converges $\norm{\bullet}b$-normally
towards $\Lambda\left(f+h\right)$, uniformly in $\norm ha$ small
enough. Besides for such values of $h$
\begin{eqnarray*}
\Lambda\left(f+h\right) & = & \lim_{m\to\infty}\sum_{\mathbf{j}\in\ww N^{m}}\frac{\left|\mathbf{j}\right|!}{\mathbf{j}!}\Lambda_{\left|\mathbf{j}\right|}\left(\oplus_{n}\left(z^{n}\right)^{\oplus j_{n}}\right)\prod_{n}h_{n}^{j_{n}}\\
 & = & \sum_{\mathbf{j}\in\ww N^{\left(\ww N\right)}}\frac{\left|\mathbf{j}\right|!}{\mathbf{j}!}\Lambda_{\left|\mathbf{j}\right|}\left(\oplus_{n}\left(z^{n}\right)^{\oplus j_{n}}\right)\prod_{n}h_{n}^{j_{n}}\,.
\end{eqnarray*}
\end{prop}
\begin{proof}
Take a sequence $\left(f_{n}\right)_{n\in\ww N}\subset\germ z$. Using
the multinomial formula we derive for every $m\in\ww N$ 
\begin{eqnarray*}
\Lambda_{p}\left(\left(\sum_{n=0}^{m}f_{n}\right)^{\oplus p}\right) & = & \sum_{\mathbf{j}\in\ww N^{m}\,,\,\left|\mathbf{j}\right|=p}\frac{p!}{\mathbf{j}!}\Lambda_{p}\left(\oplus_{n=0}^{m}\left(f_{n}\right)^{\oplus j_{n}}\right)\,.
\end{eqnarray*}
Since $\Lambda$ is analytic for all $b\in B$ there exists $\eta>0$
small enough and $a\in A$ such that $\sup_{\norm ua=\eta}\norm{\Lambda\left(f+u\right)}b$
is finite ; let us write $K$ this value. We have for for all $m\in\ww N$
and all $h=\sum_{n}h_{n}z^{n}\in\germ z$, invoking Cauchy's extended
estimate~\eqref{eq:estim_Cauchy_multilin} and Stirling~estimate~\eqref{eq:stirling_estim}:
\begin{eqnarray*}
\norm{\Lambda_{p}\left(\left(\sum_{n=0}^{m}h_{n}z^{n}\right)^{\oplus p}\right)}b & = & \norm{\sum_{\mathbf{j}\in\ww N^{m}\,,\,\left|\mathbf{j}\right|=p}\frac{p!}{\mathbf{j}!}\Lambda_{p}\left(\oplus_{n=0}^{m}\left(h_{n}z^{n}\right)^{\oplus j_{n}}\right)}b\\
 & \leq & K\left(\frac{e}{\eta}\norm ha\right)^{p}\,.
\end{eqnarray*}
The left-hand side thereby admits a limit as $m\to\infty$ and because
$\Lambda_{p}$ is continuous we obtain
\begin{eqnarray*}
\Lambda_{p}\left(h^{\oplus p}\right) & = & \lim_{m\to\infty}\sum_{\mathbf{j}\in\ww N^{m}\,,\,\left|\mathbf{j}\right|=p}\frac{p!}{\mathbf{j}!}\Lambda_{p}\left(\oplus_{n}\left(h_{n}z^{n}\right)^{\oplus j_{n}}\right)\\
 & = & \sum_{j\in\ww N^{\ww N}\,,\,\left|j\right|=p}\frac{p!}{j!}\Lambda_{p}\left(\oplus_{n=0}^{\infty}\left(h_{n}z^{n}\right)^{\oplus j_{n}}\right)
\end{eqnarray*}
with convergence of the right-hand side in $\germ z$, and if $\norm ha<\frac{\eta}{e}$
we have 
\begin{eqnarray*}
\sum_{p}\norm{\Lambda_{p}\left(h^{\oplus p}\right)}b & \leq & K\frac{\eta}{\eta-e\norm ha}\,.
\end{eqnarray*}

\end{proof}

\subsubsection{Quasi-strong analyticity}

~

We introduce now a notion of holomorphy between spaces of germs that
does not require any \emph{a priori }topology on $\germ z$. This
notion is traditionally referred to as «the» notion of holomorphy
in $\germ z$ in the context of dynamical systems.
\begin{defn}
Let $U$ be a non-empty subset of $\germ z$.
\begin{enumerate}
\item A \textbf{germ of a holomorphic} map $\lambda\,:\:\left(\ww C^{m},0\right)\to U$
is a mapping defined on some neighborhood of $\mathbf{0}\in\ww C^{m}$
and such that the map
\begin{eqnarray*}
\lambda^{*}\,:\,\left(\mathbf{x},z\right)\longmapsto\lambda\left(\mathbf{x}\right)\left(z\right)
\end{eqnarray*}
belongs to $\germ{\mathbf{x},z}$. The set of all such germs will
be denoted by $\mathcal{O}\left(\left(\ww C^{m},0\right)\to U\right)$.
\item A map $\Lambda\,:\,U\to\germ w$ is said to be \textbf{quasi-strongly
analytic} on $U$ if for any $m\in\ww N_{>0}$ and any $\lambda\in\mathcal{O}\left(\left(\ww C^{m},0\right)\to U\right)$
the composition $\Lambda\circ\lambda$ belongs to $\mathcal{O}\left(\left(\ww C^{m},0\right)\to\germ w\right)$.
\item We extend in the obvious way these definitions to analytic maps between
analytic spaces as in Definition~\ref{def_anal_space}.
\end{enumerate}
\end{defn}
\begin{rem*}
~
\begin{enumerate}
\item The notion~(2) is called <<strong analyticity>> in \cite{GenzTey}.
Here we reserve this name to the yet stronger notion we introduce
below.
\item Notice that by definition the composition of source/range-compatible
quasi-strongly analytic mappings is again quasi-strongly analytic.
\item Each affine map $x\mapsto f+xh$ is quasi-strongly analytic for fixed
$f\in U$ and $h\in\germ z$.
\item It is \emph{not }sufficient to consider the case $m=1$, as opposed
to what happens for G-holomorphic mappings.
\end{enumerate}
\end{rem*}
We begin with a characterization of quasi-strongly analytic maps in
terms of power-wise holomorphy:
\begin{prop}
\label{prop:coef_holo}Let a mapping $\Lambda~:~U\to F$ be given,
where $U\subset\germ z$, and write $\Lambda\left(f\right)=\sum_{n\geq0}T_{n}\left(\Lambda\left(f\right)\right)w^{n}$.
The following assertions are equivalent:
\begin{enumerate}
\item $\Lambda$ is quasi-strongly analytic,
\item for all $\lambda\in\mathcal{O}\left(\left(\ww C^{m},0\right)\to U\right)$
the following two conditions are fulfilled

\begin{itemize}
\item $T_{n}\left(\Lambda\circ\lambda\right)\in\mathcal{O}\left(\left(\ww C^{m},0\right)\to\ww C\right)$
is analytic on a disc whose size does not depend on $n$,
\item ~
\begin{eqnarray*}
\liminf_{\mathbf{x}\to0}\rad{\Lambda\left(\lambda\left(\mathbf{x}\right)\right)} & > & 0.
\end{eqnarray*}

\end{itemize}
\end{enumerate}
\end{prop}
In that sense quasi-strongly analytic maps are precisely those who
are uniformly power-wise quasi-strongly analytic and who respect lower-boundedness
of the radius of convergence, a fact that was implicitly used in~\cite{GenzTey}.
\begin{proof}
Let $\lambda\in\mathcal{O}\left(\left(\ww C^{m},0\right)\to U\right)$
with $\lambda\left(0\right)=f\in U$ be given. If (1) holds then $\Lambda^{*}\,:\,\left(\mathbf{x},w\right)\mapsto\Lambda\left(\lambda\left(\mathbf{x}\right)\right)\left(w\right)$
is analytic, and
\begin{align*}
T_{n}\left(\Lambda\left(\lambda\left(\mathbf{x}\right)\right)\right) & =\frac{1}{n!}\pp{^{n}\Lambda^{*}}{w^{n}}\left(\mathbf{x},0\right)
\end{align*}
so that $T_{n}\circ\Lambda\,:\,U\to\ww C$ is quasi-strongly analytic
on $\varepsilon\ww D^{m}$ for some $\varepsilon>0$ independent on
$n$. Besides $\inf_{\norm{\mathbf{x}}{}<\varepsilon}\rad{\Lambda\left(\lambda\left(\mathbf{x}\right)\right)}>0$
by definition. Assume now that (2) holds and take $\eta>0$ strictly
less than $\liminf_{\mathbf{x}\to0}\rad{\Lambda\left(\lambda\left(\mathbf{x}\right)\right)}$,
together with $\varepsilon$ for which $T_{n}\circ\Lambda\circ\lambda$
is holomorphic on $\varepsilon\ww D^{m}$. Then there exists $\varepsilon\geq\gamma>0$
such that
\begin{align*}
\inf_{\norm{\mathbf{x}}{}=\gamma}\liminf_{n\to+\infty}\left|T_{n}\circ\Lambda\left(\lambda\left(\mathbf{x}\right)\right)\right|^{-\nf 1n} & >\eta>0\,.
\end{align*}
In particular $\sum_{n\geq0}T_{n}\left(\Lambda\left(\lambda\left(\mathbf{x}\right)\right)\right)w^{n}$
is absolutely convergent on $\gamma\ww D^{m}\times\frac{\eta}{2}\ww D$.
\end{proof}
We can establish a link between this notion of holomorphy and that
of G-holomorphy, which is only natural since the concepts look alike.
\begin{prop}
\label{prop:holo_is_G-anal}Let $\Lambda\,:\,U\to F$ be quasi-strongly
analytic and assume that the ordering between topologies $\igerm z\geq F$
holds. Then $\Lambda$ is G-holomorphic.\end{prop}
\begin{proof}
We exploit the fact that it possible to use the usual Cauchy integral
formula of $\left(x,w\right)\mapsto\Lambda\left(f+xh\right)\left(w\right)$
to derive the Taylor coefficients of $x\mapsto\Lambda\left(f+xh\right)$.
\begin{lem}
\label{lem:strong_ana_Cauchy}Assume $\Lambda\,:\,U\to F$ is quasi-strongly
analytic. Let $f\in U$ and $h\in E$; we can choose $\eta>0$ such
that $f+2\eta\ww Dh\subset U$. Then for all $p\in\ww N$ 
\begin{eqnarray*}
\Lambda_{p}\left(h^{\oplus p}\right) & := & w\mapsto\oint_{\eta\sone}\Lambda\left(f+xh\right)\left(w\right)\frac{\dd x}{x^{p+1}}
\end{eqnarray*}
defines a $p$-linear symmetric mapping, whose values do not depend
on the choice of $\eta$ provided it is kept sufficiently small.
\end{lem}
The $p$-linear mappings $\Lambda_{p}$ need not be continuous and
Theorem~\ref{thm:anal_gene_prop}~(3) will be of help here. Because
of the extended Cauchy estimates in $F$ (Proposition~\ref{prop:anal_Cauchy_integ})
we have for all $b\in B$
\begin{eqnarray*}
\norm{\Lambda_{p}\left(h^{\oplus p}\right)}b & \leq & \eta^{-p}\sup_{\left|x\right|=\eta}\norm{\Lambda\left(f+xh\right)}b\,.
\end{eqnarray*}
Since $F\leq\igerm z$ for every $b\in B$ and every 
\begin{eqnarray*}
0 & <\rho< & \inf_{\left|x\right|=\eta}\rad{\Lambda\left(f+xh\right)}
\end{eqnarray*}
there exists a constant $C=C_{b,\rho}>0$ such that 
\begin{align*}
\norm{\Lambda\left(f+xh\right)}b & \leq C\norm{\Lambda\left(f+xh\right)}{\rho\ww D}\,.
\end{align*}
For example if $F$ is given by a sequential $B$-topology then the
Cauchy estimates in $\ww C$ (Lemma~\ref{lem:Taylor_is_continuous})
states that $C:=\norm{w\mapsto\frac{\rho}{\rho-w}}b$. As a consequence
\begin{eqnarray*}
\norm{\Lambda_{p}\left(h^{\oplus p}\right)}b & \leq & C\eta^{-p}\norm{\left(x,w\right)\mapsto\Lambda\left(f+xh\right)\left(w\right)}{\eta\ww D\times\rho\ww D}
\end{eqnarray*}
defines the general term of a normally convergent power series 
\begin{eqnarray*}
x & \longmapsto & \sum_{p=0}^{\infty}\Lambda_{p}\left(h^{\oplus p}\right)x^{p}\,,
\end{eqnarray*}
uniformly in $\left|x\right|\leq\frac{\eta}{2}$ . By construction
this mapping is analytic at $0$ and its sum is $\Lambda\left(f+xh\right)$.
This is the required property of G-holomorphy. 
\end{proof}

\subsubsection{\label{sub:strong_holomorphy}Strong analyticity}

~

It is tempting to declare that the Taylor coefficient $\Lambda_{p}$
computed in Lemma~\ref{lem:strong_ana_Cauchy} coincides with the
value of the integral
\begin{align*}
\oint_{\eta\ww S^{1}}\Lambda\left(f+xh\right)\frac{\dd x}{x^{p+1}} & .
\end{align*}
Yet nothing guarantees that this integral exists and belongs to $\germ w$.
If we require that $\Lambda$ be amply bounded and quasi-strongly
analytic then it does and equals $\Lambda_{p}\left(h^{\oplus p}\right)$.
For instance its value can be obtained as the limit of the sequence
of Riemann sums
\begin{align*}
I_{N} & :=\frac{2\ii\pi\eta}{N+1}\sum_{j=0}^{N}\frac{\Lambda\left(f+x_{j}h\right)}{x_{j}^{p}}
\end{align*}
where
\begin{align*}
x_{j} & :=\eta\exp\left(j\frac{2\ii\pi}{N+1}\right)\,.
\end{align*}
There exists $\varepsilon>0$ such that $\left(x,w\right)\in\frac{3}{2}\eta\ww D\times\varepsilon\ww D\mapsto\Lambda\left(f+xh\right)\left(w\right)$
is holomorphic and bounded, in which case
\begin{align*}
I_{N}\left(w\right) & =\frac{2\ii\pi\eta}{N+1}\sum_{j=0}^{N}\frac{\Lambda\left(f+x_{j}h\right)\left(w\right)}{x_{j}^{p}}
\end{align*}
converges as $N\to\infty$ toward $\oint_{\eta\sone}\Lambda\left(f+xh\right)\left(w\right)\frac{\dd x}{x^{p+1}}$
on the one hand, while it converges toward $\oint_{\eta\sone}\Lambda\left(f+xh\right)\frac{\dd x}{x^{p+1}}\left(w\right)$
on the other hand. This motivates the
\begin{defn}
\label{def_complete_holo}Assume $F\leq\igerm z$. An amply bounded
map $\Lambda\,:\,U\subset E\to F$ which is also quasi-strongly analytic
will be called \textbf{strongly analytic}.\end{defn}
\begin{rem*}
~
\begin{itemize}
\item Notice again that the composition of source/range compatible strongly
analytic mappings remains strongly analytic.
\item We recall that ample boundedness derives from continuity, and the
latter property will in fact be automatically guaranteed in virtue
of Theorem~\ref{thm:anal_G_F_eqv} and as a consequence of Theorem~C,
which we prove now.
\end{itemize}
\end{rem*}
\begin{prop}
\label{prop:anal_and_radius_minor_iff_strong_anal}Let a map $\Lambda\,:\,U\subset E\to F$
be given, where $U$ is an open set of $E$ and $F\leq\igerm z$.
The following assertions are equivalent:
\begin{enumerate}
\item $\Lambda$ is strongly analytic,
\item $\Lambda$ is analytic and for every $\lambda\in\mathcal{O}\left(\left(\ww C^{m},0\right)\to U\right)$
we have simultaneously
\begin{align*}
\liminf_{\mathbf{x}\to0}\rad{\Lambda\circ\lambda\left(\mathbf{x}\right)} & >0\\
\limsup_{\mathbf{x}\to0\,,\,r\to0}\norm{\Lambda\circ\lambda\left(\mathbf{x}\right)}{r\ww D} & <\infty\,.
\end{align*}

\end{enumerate}
\end{prop}
\begin{proof}
(1)$\Rightarrow$(2) is a direct consequence of the definition of
holomorphy and of the application of Theorem~\ref{thm:anal_G_F_eqv}
to Proposition~\ref{prop:holo_is_G-anal}. Let us prove (2)$\Rightarrow$(1).
Since an analytic map is amply bounded we only need to prove the holomorphy
of $\Lambda$. Let $\lambda\in\mathcal{O}\left(\left(\ww C^{m},0\right)\to U\right)$
be given; without loss of generality we can assume that $\lambda\left(0\right)=0\in U$.
Since $\left(u,\mathbf{x}\right)\mapsto u\lambda\left(\mathbf{x}\right)$
belongs to $\mathcal{O}\left(\left(\ww C^{m+1},0\right)\to U\right)$
we can consider a positive number $0<r'<\liminf_{\left(u,\mathbf{x}\right)\to0}\rad{\Lambda\left(u\lambda\left(\mathbf{x}\right)\right)}$.
Repeating the construction of Lemma~\ref{lem:strong_ana_Cauchy}
and of the beginning of this paragraph we deduce that there exists
$\rho,\,\rho'>0$ such that for all integer $p$ 
\begin{align*}
\left(\forall\left|w\right|<r'\,,\,\forall\norm{\mathbf{x}}{}<\rho'\right)\,\,\,\,\,\Lambda_{p}\left(\lambda\left(\mathbf{x}\right)^{\oplus p}\right)\left(w\right) & =\frac{1}{2\ii\pi}\oint_{\rho\ww S^{1}}\Lambda\left(u\lambda\left(\mathbf{x}\right)\right)\left(w\right)\frac{\dd u}{u^{p+1}}\,,
\end{align*}
where $\Lambda_{p}$ is the Taylor coefficient of $\Lambda$. This
formula implies particularly that $\Lambda_{p}\left(\lambda\left(\mathbf{x}\right)^{\oplus p}\right)\in\mathcal{B}_{r'}$
and 
\begin{align*}
\left|\Lambda_{p}\left(\lambda\left(\mathbf{x}\right)^{\oplus p}\right)\left(w\right)\right| & \leq\rho^{-p}\sup_{\left|u\right|=\rho}\left|\Lambda\left(u\lambda\left(\mathbf{x}\right)\right)\left(w\right)\right|\,.
\end{align*}
By choosing $\rho>\eta>0$, and using that $\lim\sup_{\mathbf{x}\to0,\,r\to0}\norm{\Lambda\left(\lambda\left(\mathbf{x}\right)\right)}{r\ww D}<\infty$
we deduce that, at the expense of decreasing $r'$ and $\rho'$, there
exists $K>0$ such that
\begin{align*}
\left|\Lambda_{p}\left(\left(\eta\lambda\left(\mathbf{x}\right)\right)^{\oplus p}\right)\left(w\right)\right| & \leq K\left(\frac{\eta}{\rho}\right)^{p}\,.
\end{align*}
Hence the functional series $\left(\mathbf{x},w\right)\mapsto\sum_{p\geq0}\Lambda_{p}\left(\lambda\left(\mathbf{x}\right)^{\oplus p}\right)\left(w\right)$
is uniformly convergent and the sum is thereby analytic on $\rho'\ww D\times r'\ww D$. 
\end{proof}
Let us conclude this part by proving the
\begin{prop}
\label{prop:compo_is_ana}Assume that $\germ z$ is endowed with a
useful topology and pick $f\in\germ z$ with radius of convergence
$\rad f$. The right-composition mapping
\begin{align*}
f_{*}~:~\germ{\mathbf{z}} & \longrightarrow\germ{\mathbf{z}}\\
g & \longmapsto f\circ g
\end{align*}
is strongly analytic on the domain $U:=J_{0}^{-1}\left(\rad f\ww D\right)=\left\{ g~:~\left|g\left(0\right)\right|<\rad f\right\} $. 
\end{prop}
We refer to Proposition~\ref{prop:compo_derive} for the computation
of the coefficients of the power series representing $f_{*}$.
\begin{proof}
By definition of a useful topology $f_{*}$ is continuous, so we only
need to show that it is quasi-strongly analytic. The latter property
is trivial because being given $\lambda\in\mathcal{O}\left(\left(\ww C^{m'},0\right)\to U\right)$
with $\lambda\left(0\right)=g\in U$, the map $f_{*}\circ\lambda$
also belongs to $\mathcal{O}\left(\left(\ww C^{m'},0\right)\to\germ{\mathbf{z}}\right)$
as $\left(f_{*}\circ\lambda\right)^{*}\left(\mathbf{x},\mathbf{z}\right)=f\left(\lambda^{*}\left(\mathbf{x},\mathbf{z}\right)\right)$
defines a convergent power series in $\left(\mathbf{x},\mathbf{z}\right)$.
\end{proof}

\section{\label{sec:anal_sets}Analytic sets}

This section is devoted to introducing and studying properties of
analytic sets, as defined by the vanishing locus of analytic maps
$E\to F$ where $\left(E,\left(\norm{\bullet}a\right)_{a\in A}\right)$
and $\left(F,\left(\norm{\bullet}b\right)_{b\in B}\right)$ are Hausdorff,
locally convex spaces. 
\begin{defn}
\label{def_anal_set}~
\begin{enumerate}
\item A closed subset $\Omega\subset E$ is an \textbf{analytic set} if
for all $f\in\Omega$ there exists a domain $U\ni f$ and some collection
$\left(\Lambda_{\iota}\right)_{\iota\in I}$ of analytic mappings
$\Lambda_{\iota}\,:\,U\to F$ such that $\Omega\cap U=\bigcap_{\iota\in I}\Lambda_{\iota}^{-1}\left(0\right)$. 
\item An analytical set is called \textbf{proper} at some point $f\in\Omega$
if for all collection $\left(\left(\Lambda_{\iota}\right)_{i\in I},U\right)$,
with $U\ni f$ and $\Omega\cap U=\bigcap_{\iota\in I}\Lambda_{\iota}^{-1}\left(0\right)$,
we have $\Omega\cap U\neq U$. We say that $\Omega$ is proper if
it is proper at one of its points. 
\item A subset $M\subset E$ is said to be \textbf{analytically meager}
if $M$ is included in a countable union of proper analytic sets.
$M$ is \textbf{countably meager} if it is included in an at-most-countable
union of images by analytic maps of at-most-countable-dimensional
linear spaces.
\item Let $X$ be an analytic space whose analytic structure is induced
by a continuous and onto map $\varphi\,:\,U\subset E\to X$ as in
Definition~\ref{def_anal_space}. A closed subset $\Omega\subset X$
is an analytic set if $\varphi^{-1}\left(\Omega\right)$ is an analytic
set of $U$. We likewise extend to this setting the notion of analytically
meager subsets of $X$, and say that $X$ is an \textbf{analytical
Baire space} if every analytically meager subset of $X$ has empty
interior.
\end{enumerate}
\end{defn}
\begin{rem}
~
\begin{enumerate}
\item There is no \emph{a priori }hypothesis on the cardinality of the set
$I$ in the definition of analytic sets.
\item It is obvious from the definition that finite unions and unspecified
intersections of analytic sets still are analytic sets.
\item We will show shortly that a proper analytic set is actually proper
at each one of its points (Proposition~\ref{prop:proper_anal_set}).
\item In (4) above we expressly do \emph{not} define analytic sets\textbf{
}of $X$ as the vanishing loci of analytic maps. This condition is
indeed far too restrictive, as will be illustrated in Section~\ref{sec:application_analysis}
when speaking about the analytic space of meromorphic functions (Remark~\ref{rem_mero}).
\end{enumerate}
\end{rem}
Because taking the $N^{\tx{th}}$-jet of a germ is an analytic operation,
each subspace of polynomials with given upper-bound $N$ on their
degree 
\begin{eqnarray*}
\pol z{\leq N} & = & \ker\left(\sum_{n}a_{n}z^{n}\longmapsto\sum_{n>N}a_{n}z^{n}\right)=\ker\left(\id-J_{N}\right)
\end{eqnarray*}
is a proper analytic set. The countable union $\pol z{}$ of all these
subspaces is consequently analytically meager (and dense) in $\germ z$,
and therefore $\pol z{}$ is not an analytical Baire space.

\bigskip{}

A nice feature of analytic sets is that they are negligible, in the
sense that the removable singularity theorem holds:
\begin{thm}
\label{thm:removable_sing}\textbf{\emph{(Removable Singularity Theorem
\cite{Mazette})}} Let $U\subset E$ be an open set, $\Omega\subset U$
a proper analytic set and $\Lambda$ be analytic on $U\backslash\Omega$.
We assume that for every $a\in A$ and every $f\in\Omega$ there exists
a neighborhood $V\subset U$ of $f$ such that $\norm{\Lambda|_{V\backslash\Omega}}a$
is bounded. Then $\Lambda$ admits a unique analytic extension to
$U$.
\end{thm}
\bigskip{}

The aim of this section is mainly to show the next result:
\begin{thm}
\label{thm:anal_baire}Assume there exists a Banach space $\mathcal{V}$
and a continuous, linear mapping $\iota\,:\,\mathcal{V}\to E$ such
that $\iota\left(\mathcal{V}\right)$ is dense in $E$. Then any analytic
space $X$ modeled on $E$ is an analytical Baire space.
\end{thm}
We deduce Theorem~A from this statement and the next trivial lemma:
\begin{lem}
Let $\mathcal{V}$ be the vector subspace of $\mathbb{C}\left\{ z\right\} $
defined by
\[
\mathcal{V}:=\left\{ \sum_{n}f_{n}z^{n}\,:\,\left|f_{n}\right|\textup{ is bounded}\right\} 
\]
equipped with the norm $\left|\left|\bullet\right|\right|_{\infty}$
:
\[
\left|\left|f\right|\right|_{\infty}:=\sup_{n}\left|f_{n}\right|\,.
\]
Then $\left(\mathcal{V},\left|\left|\bullet\right|\right|_{\infty}\right)$
is a Banach space and the inclusion $\left(\mathcal{V},\norm{\bullet}{\infty}\right)\hookrightarrow\germ z_{\mathcal{A}}$
is continuous. More precisely for each $f\in\mathcal{V}$ and each
$a\in\mathcal{A}$ 
\begin{eqnarray*}
\norm fa & \leq & \norm f{\infty}\norm{\frac{1}{1-z}}a\,.
\end{eqnarray*}

\end{lem}
We also intend to prove the Theorem~B:
\begin{thm}
\label{thm:anal_range_meager}Let $X$ be an analytic space modeled
on $\germ z$ and consider an analytic map $\Lambda\,:\,U\to X$ on
an open set $U\subset E$. The image by $\Lambda$ of the trace on
$U$ of any linear subspace of $E$ with at most countable dimension
is analytically meager. In particular countably meager sets are analytically
meager.\end{thm}
\begin{rem}
~
\begin{enumerate}
\item It is worth noticing that the analyticity condition on $\Lambda$
can be loosened a little. Indeed we are only considering the countable
union of range of analytic maps restricted to finite-dimensional open
sets, so that only G-holomorphy (or quasi-strong analyticity) is actually
required. In this form the spine of the proof is borrowed from \cite{GenzTey},
but some flesh is added to cover the more general context.
\item Be careful that <<analytically meager>> does not imply <<empty
interior>> in general, according to Theorem~\ref{thm:anal_baire}.
\end{enumerate}
\end{rem}

\subsection{Proper analytic sets}

~
\begin{prop}
\label{prop:proper_anal_set}Let $\Omega$ be a connected, non-empty
analytic set. The following assertions are equivalent :
\begin{enumerate}
\item $\Omega$ is proper,
\item $\Omega$ is proper at any one of its points,
\item the interior of $\Omega$ is empty,
\item for all $f\in\Omega$ there exists a domain $U\ni f$ and a non-empty
collection $\left(\Lambda_{\iota}\right)_{\iota\in I}$ such that
$U\cap\Omega=\bigcap_{\iota\in I}\Lambda_{\iota}^{-1}\left(0\right)$
while $\Lambda_{\iota}\neq0$,
\item there exists a domain $U$ meeting $\Omega$ and a non-empty collection
$\left(\Lambda_{\iota}\right)_{\iota\in I}$ such that $U\cap\Omega=\bigcap_{\iota\in I}\Lambda_{\iota}^{-1}\left(0\right)$
while $\Lambda_{\iota}\neq0$.
\end{enumerate}
\end{prop}
A consequence of this result is that proper analytical meager subsets
of $\germ z$ are genuine meager subsets in the sense of Baire. 
\begin{proof}
The implications (3)$\Rightarrow$(2)$\Rightarrow$(4)$\Rightarrow$(5)$\Rightarrow$(1)
are trivial. Because of the identity theorem (Corollary~\ref{cor:anal_PPA})
we also have (4)$\Rightarrow$(3). Let us show (1)$\Rightarrow$(2)
by a connectedness argument; for this consider the set
\begin{eqnarray*}
C & := & \left\{ f\in\Omega\,:\,\Omega\mbox{ is not proper at }f\right\} \,,
\end{eqnarray*}
which is an open subset of $\Omega$. Let us consider a sequence $\left(f_{n}\right)_{n\in\ww N}\subset C$
converging in $\Omega$ to some $f_{\infty}$ and prove that $f_{\infty}\in C$.
Pick a domain $U\ni f_{\infty}$ such that $U\cap\Omega=\bigcap_{\iota\in I}\Lambda_{\iota}^{-1}\left(0\right)$
for some collection of maps $\left(\Lambda_{\iota}\right)_{\iota\in I}$
analytic on $U$, acknowledging that for $n$ big enough we have $f_{n}\in U$.
Therefore $\Lambda_{\iota}$ must vanish on some small ball around
$f_{n}$ and the identity theorem applies once more proving that $\Lambda_{\iota}$
vanishes on the whole $U$, which is $f_{\infty}\in C$. Because by
hypothesis $\Omega$ is connected and $\Omega\backslash C\neq\emptyset$
we showed $C=\emptyset$, which is (2), as required.
\end{proof}

\subsection{Analytical Baire property : proof of Theorem~\ref{thm:anal_baire}}

By definition of the analytic sets of $X$ it is sufficient to show
the property for $E$. Assume then that there exists a countable collection
of analytic sets $\left(C_{n}\right)_{n\in\ww N}$ such that $\bigcup_{n}C_{n}$
has non-empty interior $W$, and show that at least one $C_{n}$ is
not proper. Pick some point $f$ in $W$; the affine subspace 
\begin{eqnarray*}
\mathcal{V}_{f}: & = & f+\iota\left(\mathcal{V}\right)
\end{eqnarray*}
is Baire. Since $W$ is open the intersection $W\cap\mathcal{V}_{f}$
is a non-void open subset in $\mathcal{V}_{f}$ (\emph{i.e. }for the
topology induced by $\norm{\bullet}{\mathcal{V}}$), and $C_{n}\cap\mathcal{V}_{f}$
is closed in $\mathcal{V}_{f}$. Therefore at least one $C_{n}\cap\mathcal{V}_{f}$
has a non-empty interior $W_{n}$. Let $\Lambda\,:\,U\to F$ be an
analytic map on a neighborhood $U$ of some point $\tilde{f}\in W_{n}$
such that $C_{n}\cap U\subset\Lambda^{-1}\left(0\right)$. Because
$\Lambda|_{\mathcal{V}_{f}}$ is analytic for the analytic structure
of $\mathcal{V}_{f}$ inherited from its Banach topology, necessarily
$\Lambda$ vanishes on $\mathcal{V}_{f}$. By denseness of $\mathcal{V}_{f}$
in $E$ we deduce that $\Lambda=0$ and $C_{n}$ cannot be proper,
according to the characterization given in Proposition~\ref{prop:proper_anal_set}.

\subsection{Range of a finite-dimensional analytic map : proof of Theorem~B}

~
\begin{prop}
\label{prop:finite_range_is_analytic}Take an analytic map $\lambda\,:\,V\to\germ w$
defined on some neighborhood $V$ of $\mathbf{0}\in\ww C^{m}$ such
that $\tx{rank}\left(\partial_{\bb 0}\lambda\right)=m$. There exists
an open neighborhood $W\subset V$ of the origin such that the following
properties hold:
\begin{enumerate}
\item $\lambda|_{W}$ is one-to-one,
\item $\lambda\left(W\right)$ is the trace of a proper analytic set on
some domain $U\subset\germ w$ containing $\lambda\left(\mathbf{0}\right)$,
defined by the vanishing of a non-zero analytic mapping $\Lambda~:~U\to\germ w$.
\end{enumerate}
\end{prop}
\begin{proof}
We can assume, without loss of generality, that $\lambda\left(\mathbf{0}\right)=0$. 
\begin{enumerate}
\item We adapt here the proof done in \cite{GenzTey}. If there exists $N\in\ww N$
such that the map $J_{N}\circ\lambda$ is one-to-one, then so is $\lambda$.
We claim that a family $f_{1},\ldots,f_{k}\in\germ w$ is free (over
$\mathbb{C}$) if, and only if, there exists $N\in\mathbb{N}$ such
that their $N$-jets are free (as elements of the vector space $\pol w{\leq N}$).
Suppose that for any $N\in\mathbb{N}$ there exists a non-trivial
relation 
\[
L_{N}:=\left(\lambda_{1,N},\ldots,\lambda_{k,N}\right)\neq0
\]
 for the family $C_{N}:=\left(J_{N}\left(f_{1}\right),\ldots,J_{N}\left(f_{k}\right)\right)$,
that is 
\begin{eqnarray*}
J_{N}\left(\sum_{j=1}^{k}\lambda_{j,N}f_{j}\right) & = & 0\,.
\end{eqnarray*}
Up to rescaling of $L_{N}$ one can suppose that it belongs to the
unit sphere of $\mathbb{C}^{k}$ and thereby one can consider some
accumulation point $L_{\infty}:=\left(\lambda_{1,\infty},\ldots,\lambda_{k,\infty}\right).$
Because $J_{p+1}\left(f\right)=0$ implies $J_{p}\left(f\right)=0$,
by taking the limit $N\to\infty$ while fixing an arbitrary $p$ we
obtain that $L_{\infty}$ is a non-trivial relation for $C_{p}$ (by
continuity of $J_{N}$), and thus is a non-trivial relation for $\left(f_{1},\cdots,f_{k}\right)$,
proving our claim. If $\lambda$ is of maximal rank at $\mathbf{0}$,
\emph{i.e.} the rank of $\partial_{\mathbf{0}}\lambda$ is $m$, there
accordingly exists $N\in\mathbb{N}$ such that the mapping $J_{N}\circ\lambda$
is of maximal rank. Therefore the mapping $J_{N}\circ\lambda$ is
locally one-to-one around $\mathbf{0}$, say on some domain $V'\subset V$,
and so is $\lambda$.
\item Let $N\in\ww N$ and $V'$ be given as above, so that $J_{N}\circ\lambda|_{V'}$
is one-to-one. Let $\left(b_{0},\cdots,b_{m-1}\right)\subset\pol z{\leq N}$
be a basis of $\tx{im}\left(\partial_{0}\left(J_{N}\circ\lambda\right)\right)$,
which we complete with a basis $\left(b_{m},\cdots,b_{N}\right)\subset\pol z{\leq N}$
of the cokernel of $\partial_{0}\left(J_{N}\circ\lambda\right)$,
so that $\pol z{\leq N}=\ww C\left\langle b_{0},\cdots,b_{N}\right\rangle $.
Write $J_{N}\circ\lambda\left(\mathbf{x}\right)=\sum_{j=0}^{N}\lambda_{j}\left(\mathbf{x}\right)b_{j}$
with $\lambda_{j}~:~V'\to\ww C$ holomorphic. According to the local
inversion theorem there exists a domain $W\subset V'$ containing
$\mathbf{0}$ such that $\tilde{\lambda}\,:\,\mathbf{x}\in W\mapsto\sum_{j<m}\lambda_{j}\left(\mathbf{x}\right)b_{j}$
is a biholomorphism onto its image. Let us write $\Phi$ the inverse
map, holomorphic from the open set $\tilde{\lambda}\left(W\right)$
of $\ww C\left\langle b_{0},\ldots,b_{m-1}\right\rangle $ into $W$
and introduce the natural (continuous) projector
\begin{eqnarray*}
P\,:\,\germ w & \longrightarrow & \ww C\left\langle b_{0},\ldots,b_{m-1}\right\rangle \\
\sum_{n\leq N}f_{n}b_{n}+\sum_{n>N}f_{n}\id^{n} & \longmapsto & \sum_{n<m}f_{n}b_{n}\,,
\end{eqnarray*}
so that $P\circ\lambda=\tilde{\lambda}$. Consider the (analytic)
mapping
\begin{eqnarray*}
\Lambda\,:\,J_{N}^{-1}\left(\tilde{\lambda}\left(W\right)\right)\subset\germ w & \longrightarrow & \germ w\\
f & \longmapsto & f-\lambda\circ\Phi\circ P\left(f\right)\,.
\end{eqnarray*}
By construction $\Lambda\left(f\right)=0$ if, and only if, $f\in\lambda\left(W\right)$.
In particular $\lambda\left(W\right)$ is closed in $U:=J_{N}^{-1}\left(\tilde{\lambda}\left(W\right)\right)$.
\end{enumerate}
\end{proof}
\begin{cor}
Let $\lambda\,:\,V\subset\ww C^{m}\to\germ w$ be analytic. Then $\lambda\left(V\right)$
is analytically meager.
\end{cor}
Since a countable union of analytically meager subsets remains analytically
meager Theorem~B derives from this result.
\begin{proof}
It is done by induction on $m$, the result being trivially true for
$m=0$. The reasoning is essentially the same as in \cite{GenzTey},
and consists in splitting $V$ into at most countably many sets on
which the restriction of $\lambda$ has relative maximal rank. Outside
a proper analytic subset $\Sigma$ of $V$ the rank of the map $\lambda$
is constant; let us call $\mu\left(\lambda\right)$ this generic rank.
If $\mu\left(\lambda\right)<m$ then we can find at most countably
many analytic discs $D_{k}\subset V$ of dimension $\mu\left(\lambda\right)$
such that $\bigcup_{k}\lambda\left(D_{k}\right)=\lambda\left(V\right)$.
Thanks to our induction hypothesis this case has already been dealt
with and $\lambda\left(V\right)$ is analytically meager. We therefore
assume that $\mu\left(\lambda\right)=m$. The set $\Sigma$ admits
a decomposition $\Sigma=\bigcup_{k}C_{k}$ into at most countably
many analytic discs of dimension $p_{k}$ with $0\leq p_{k}<m$ (see
\emph{e.g. }\cite{coste}). We apply the induction hypothesis to each
$\lambda|_{C_{k}}$ as before to obtain that $\lambda\left(\Sigma\right)$
is analytically meager. We can apply Proposition~\ref{prop:finite_range_is_analytic}
at each point of the open set $V\backslash\Sigma$, which is therefore
covered by countably many domains $\left(W_{k}\right)_{k}$ such that
each $\lambda\left(W_{k}\right)$ is included in a proper analytic
set. This settles the induction.
\end{proof}

\subsection{Remarks on tangent spaces}

~

As in the finite dimensional case we can build (at least) two candidate
tangent spaces.
\begin{defn}
\label{def_tang_space}Let $\Omega$ be a non-empty analytic set and
pick $f\in\Omega$, together with a collection of locally defining
maps $\left(\Lambda_{\iota}\right)_{\iota}$.
\begin{enumerate}
\item We define the \textbf{algebraic tangent space }of $\Omega$ at $f$
\begin{align*}
\tang fA{\Omega} & :=\bigcap_{\iota}\ker\p{\Lambda^{\iota}}ff\,.
\end{align*}

\item We define the \textbf{geometric tangent space} of $\Omega$ at $f$
\begin{align*}
\tang fG{\Omega} & :=\left\{ \lambda'\left(0\right)\,:\,\lambda\mbox{ holomorphic }\left(\ww C,0\right)\to\left(\Omega,f\right)\right\} \,.
\end{align*}

\end{enumerate}
\end{defn}
~

Obviously $\tang fG{\Omega}<\tang fA{\Omega}$ and equality is equivalent,
for finite-dimensional spaces and irreducible $\Omega$, to the fact
that $f$ is a regular point of $\Omega$. Therefore we propose the
following definition:
\begin{defn}
\label{def_reg_point}Assume $\Omega$ is irreducible. We say that
$f$ is a \textbf{regular point} of $\Omega$ if $\tang fG{\Omega}=\tang fA{\Omega}$.
\end{defn}
We conjecture that regular points of (at least strongly) analytic
sets enjoy «nice» geometric features, maybe reaching as far as the
existence of a local analytic parameterization.

\section{\label{sec:calculus}Fréchet calculus}

In this section we fix a choice of a useful $A$-topology on $\germ z$.
For $a\in\germ z$ the notation $a\times\bullet$ stands for the endomorphism
$h\mapsto a\times h$ of $\germ z$. More generally in the formulas
bellow the symbol <<$\bullet$>> will stand for the argument of
a continuous linear mapping.

\subsection{Actually computing derivatives}

~

Computing actual derivatives of <<simple>> operations can be performed
easily. In fact one can compute them in many less simple cases using
the following formula, which is only a consequence of the continuity
of the derivative:
\begin{lem}
Let $\Lambda\,:\,U\to\germ w$ be analytic and $f\in U$. Then 
\begin{eqnarray*}
\p{\Lambda}ff\left(\sum_{n\geq0}h_{n}z^{n}\right) & = & \sum_{n\geq0}h_{n}\p{\Lambda}ff\left(z^{n}\right)\\
 & = & \sum_{n\geq0}h_{n}\frac{\partial\Lambda\left(f+x\times z^{n}\right)}{\partial x}\left(0\right)\,.
\end{eqnarray*}

\end{lem}
We shall present an example: differentiating the strongly analytic
map (Proposition~\ref{prop:compo_is_ana}) $g\mapsto f\circ g$ with
respect to $g$ (\emph{i.e.} for fixed $f$). This allows to compute
$\p{f\circ g}gg$ for any $g\in U$:
\begin{eqnarray*}
\p{f\circ g}g{\left(f,g\right)}\left(z^{n}\right) & = & \frac{\partial f\circ\left(g+xz^{n}\right)}{\partial x}\left(0\right)=z^{n}\times f'\circ g
\end{eqnarray*}
so that
\begin{prop}
\label{prop:compo_derive}Let $f\in\germ z$ be given. The right-composition
mapping $g\mapsto f\circ g$ is analytic on the domain $J_{0}^{-1}\left(\rad f\ww D\right)$
and for all $g$ within this domain we have 
\begin{eqnarray}
\p{f\circ g}g{\left(f,g\right)} & = & \left(f'\circ g\right)\times\bullet\,.\label{eq:deriv_compos}
\end{eqnarray}
More generally we have, for all $p\in\ww N$, 
\begin{eqnarray*}
\p[p]{f\circ g}gg\left(\bullet^{\oplus p}\right) & = & \left(f^{\left(p\right)}\circ g\right)\times\bullet^{p}\,.
\end{eqnarray*}

\end{prop}

\subsection{The chain rule}

~

A corollary of Theorem~\ref{thm:anal_Hartogs}, and of the usual
rules of calculus, is that if $\Lambda\,:\,W_{1}\times W_{2}\subset\germ z\times\germ z\to\germ w$
is (strongly) analytic, and if $\Lambda_{j}\,:\,U\to W_{j}$ are (strongly)
analytic maps, $j\in\left\{ 1,2\right\} $, then $\Lambda\left(\Lambda_{1},\Lambda_{2}\right)$
is also (strongly) analytic and, for all $f\in U$, 
\begin{eqnarray*}
\p{\Lambda\left(\Lambda_{1},\Lambda_{2}\right)}ff & = & \p{\Lambda}{f_{1}}{\left(\Lambda_{1}\left(f\right),\Lambda_{2}\left(f\right)\right)}\left(\p{\Lambda_{1}}ff\right)+\p{\Lambda}{f_{2}}{\left(\Lambda_{1}\left(f\right),\Lambda_{2}\left(f\right)\right)}\left(\p{\Lambda_{2}}ff\right)\,.
\end{eqnarray*}

Let us give applications of this calculation and of Proposition~\ref{prop:compo_derive}.
\begin{prop}
\label{prop:differentiate_usual_op}~
\begin{enumerate}
\item The diffeomorphing map
\begin{eqnarray*}
\mathcal{D}\,:\,g\in\germ z & \longmapsto & z\times\exp\circ g\in\diff{}\,,
\end{eqnarray*}
 is strongly analytic and for all $g\in\germ z$ 
\begin{eqnarray*}
\p{\mathcal{D}}gg & = & \mathcal{D}\left(g\right)\times\bullet\,.
\end{eqnarray*}
More generally, with convergence in $\germ z$,
\begin{eqnarray*}
\mathcal{D}\left(g\right) & = & z\times\sum_{p=0}^{\infty}\frac{g^{p}}{p!}\,.
\end{eqnarray*}

\item The inversion map 
\begin{eqnarray*}
\iota\,:\,\germ z & \longrightarrow & \diff{}\\
g & \longmapsto & \mathcal{D}\left(g\right)^{\circ-1}
\end{eqnarray*}
is strongly analytic and for all $g\in\germ z$ we have 
\begin{eqnarray*}
\p{\iota}gg & = & -\left(\frac{z}{1+z\times g'}\times\bullet\right)\circ\iota\left(g\right)\,.
\end{eqnarray*}

\item There exists a unique strongly analytic map 
\begin{eqnarray*}
\mathcal{H}\,:\,\germ z & \longrightarrow & \germ z
\end{eqnarray*}
such that $\mathcal{H}\left(0\right)=0$ and 
\begin{eqnarray*}
\iota & = & \mathcal{D}\circ\mathcal{H}\,.
\end{eqnarray*}
 If $\tx{Log}$ denotes the principal determination of the logarithm
then 
\begin{eqnarray*}
\mathcal{H}\left(g\right) & = & \tx{Log}\circ\frac{\iota\left(g\right)}{z}
\end{eqnarray*}
for all $g\in\germ z$. We have also
\begin{eqnarray*}
\p{\iota}gg & = & \iota\left(g\right)\times\p{\mathcal{H}}gg\,,
\end{eqnarray*}
that is
\begin{eqnarray*}
\p{\mathcal{H}}gg & = & -\left(\frac{\bullet}{1+z\times g'}\right)\circ\iota\left(g\right)\,.
\end{eqnarray*}
Every other holomorphic map $\tilde{\mathcal{H}}$ such that $\iota=\mathcal{D}\circ\tilde{\mathcal{H}}$
is obtained from $\mathcal{H}$ by adding an element of $2i\pi\ww Z$.
\end{enumerate}
\end{prop}
\begin{proof}
We use the rules established in the previous sections. For the sake
of clarity we use the notation $\mathcal{C}_{f}$ to denote the right
composing mapping $g\mapsto f\circ g$.
\begin{enumerate}
\item We have 
\begin{eqnarray*}
\p{\mathcal{D}}gg & = & z\times\p{\exp g}gg=\mathcal{D}\left(g\right)\times\bullet\,.
\end{eqnarray*}

\item Using the previous lemma to differentiate with respect to $g$ the
relation
\begin{eqnarray*}
\mathcal{C}_{\iota\left(g\right)}\left(\mathcal{D}\left(g\right)\right) & =z
\end{eqnarray*}
 we derive the formula:
\begin{eqnarray*}
0 & = & \p{\mathcal{C}_{f}\left(\mathcal{D}\left(g\right)\right)}f{\iota\left(g\right)}\left(\p{\iota}gg\right)+\p{\mathcal{C}_{\iota\left(g\right)}\left(h\right)}h{\mathcal{D}\left(g\right)}\left(\p{\mathcal{D}}gg\right)\\
 & = & \mathcal{C}_{\p{\iota}gg}\left(\mathcal{D}\left(g\right)\right)+\mathcal{C}_{\iota\left(g\right)'}\left(\mathcal{D}\left(g\right)\right)\times\mathcal{D}\left(g\right)\times\bullet\,.
\end{eqnarray*}
The same relation differentiated with respect to $z$ yields the usual
formula 
\begin{eqnarray*}
\mathcal{D}\left(g\right)'\times\mathcal{C}_{\iota\left(g\right)'}\left(\mathcal{D}\left(g\right)\right) & = & 1\,,
\end{eqnarray*}
completing the proof since $\mathcal{D}\left(g\right)'=\left(1+\id\times g'\right)\times\exp\circ g$.
\item The equality regarding the derivatives, whenever they exist, comes
from differentiating $\iota=\mathcal{D}\circ\mathcal{H}$. Let us
now prove the existence of $\mathcal{H}$. Consider
\begin{eqnarray*}
L\,:\,h\in\germ z & \longmapsto & \sum_{p>0}\frac{\left(-1\right)^{p+1}}{p}h^{p}\,,
\end{eqnarray*}
which is a convergent series for the topology of $\germ z$. Obviously
$\mathcal{C}_{\exp}\circ L\left(h\right)=1+h$ for all $h\in\fgrm h$.
Therefore setting
\begin{eqnarray*}
\mathcal{H}\left(g\right) & := & L\left(\frac{\iota\left(g\right)}{z}\right)-1
\end{eqnarray*}
 does the trick.
\end{enumerate}
\end{proof}
\begin{cor}
\label{cor:diffeo_univ_covering}The universal covering of $\diff{}$
can be represented by the analytic covering
\begin{eqnarray*}
\mathcal{D}\,:\,\germ z & \longrightarrow & \diff{}\\
g & \longmapsto & z\exp\circ g\,.
\end{eqnarray*}
The fiber is canonically isomorphic to $2i\pi\ww Z$: the Galois group
of the covering is generated by the shift $g\mapsto g+2i\pi$.
\end{cor}

\section{\label{sec:application_analysis}Application to complex analysis}

For a family of $k$ germs $F:=\left(f_{\ell}\right)_{1\leq\ell\leq k}$
we let $\mathcal{I}\left(F\right)$ be the ideal of $\germ{\mathbf{z}}$
spanned by the family. Define the \textbf{Milnor number} of $F$ as
\begin{align*}
\mu\left(F\right) & :=\dim_{\ww C}\nf{\germ{\mathbf{z}}}{\mathcal{I}\left(F\right)}\,\,\,\,\,\in\,\ww N\cup\left\{ \infty\right\} \,.
\end{align*}
It is a well-known consequence of the \emph{Nullstellensatz} for complex
analytic functions that the following statements are equivalent:
\begin{itemize}
\item $\mu\left(F\right)<\infty$ ,
\item $\mathcal{I}\left(F\right)$ contains a power of $\frak{Z}$, the
maximal ideal of $\germ{\mathbf{z}}$,
\item no element of $\frak{Z}$ is a common factor to all the $f_{\ell}$'s.
\end{itemize}
As it turns out the set of collections $F\in\germ{\mathbf{z}}^{m}$
with infinite Milnor number has the structure of an analytic set.
\begin{thm}
\label{thm:Milnor_analytic}Let $m$ and $k$ be integers greater
than $1$. We understand $\mu$ here as the mapping
\begin{align*}
\mu\,:\,\germ{\mathbf{z}}^{k} & \longrightarrow\ww N\cup\left\{ \infty\right\} \\
F & \longmapsto\mu\left(F\right)\,.
\end{align*}
Endow each factor $\germ{\mathbf{z}}$ with a sequential topology.
Then the set $\frak{I}:=\mu^{-1}\left(\infty\right)$ is a proper
analytic set of $\frak{Z}^{k}$.
\end{thm}
We actually show that $\mu^{-1}\left(\infty\right)$ is «algebraic»
in the sense that for each $N\in\ww N$ there exists a polynomial
map $\Lambda_{N}\,:\,\pol{\mathbf{z}}{\leq N}^{k}\to\ww C^{d\left(N\right)}$
such that 
\begin{align*}
\mu^{-1}\left(\infty\right) & =\bigcap_{N\in\ww N}\left(\Lambda_{N}\circ J_{N}\right)^{-1}\left(0\right)\,,
\end{align*}
where $J_{N}$ is the Cartesian product of the $N^{\tx{th}}$-jet
operator of each copy of $\germ{\mathbf{z}}$. From this particular
decomposition stems the fact that the theorem derives from the same
result in the factorial ring $\frml{\mathbf{z}}$. 

We show, in the rest of this section, that determining if $k\geq2$
formal power series have a non-trivial common factor, \emph{i.e. }belonging
to the maximal ideal (still written) $\frak{Z}$ of $\frml{\mathbf{z}}$,
is an «algebraic» condition. Observe the set $\frak{I}\subset\frak{Z}^{k}$
formed by non-coprime families is never empty since $\left(f,\cdots,f\right)\in\frak{I}$
whenever $f\in\frak{Z}$. The case $k=2$ encompasses all technical
difficulties so it is completed first, in Section~\ref{sub:Coprime-families}.
We finally reduce the general case $k\geq2$ to the latter study,
in Section~\ref{sub:general_case}. In Section~\ref{sub:Computations}
we present an effective computable process which stops in finite time
if, and only if, the given family $\left(f_{\ell}\right)_{\ell\leq k}$
is not coprime, using a growing sequence of Macaulay-like matrices.

\bigskip{}

Before performing all these tasks we shall use the above structure
theorem to equip the space of germs of meromorphic functions with
an analytic structure.

\subsection{The analytic space of meromorphic germs}

~

Let $\mero{\mathbf{z}}$ stand for the space of germs of meromorphic
functions (for short, a meromorphic germ) at the origin of $\ww C^{m}$.
This space is by definition the fractions field of the ring $\germ{\mathbf{z}}$,
that is the set of equivalence classes of couples $\left(P,Q\right)\in\germ{\mathbf{z}}\times\left(\germ{\mathbf{z}}\backslash\left\{ 0\right\} \right)$
given by
\begin{align}
\left(P_{1},Q_{1}\right)\leftrightsquigarrow\left(P_{2},Q_{2}\right) & \Longleftrightarrow P_{1}Q_{2}=P_{2}Q_{1}\,.\label{eq:eqv_mero}
\end{align}
Of course we write $\frac{P}{Q}$ the equivalence class of $\left(P,Q\right)$,
and name $\quot$ the canonical projection 
\begin{align*}
\quot\,:\,\germ{\mathbf{z}}\times\left(\germ{\mathbf{z}}\backslash\left\{ 0\right\} \right) & \longrightarrow\mero{\mathbf{z}}\\
\left(P,Q\right) & \longmapsto\frac{P}{Q}\,.
\end{align*}

\begin{defn}
~\label{def_pure_mero}
\begin{enumerate}
\item We say that $\left(P,Q\right)$ is a \textbf{proper representative}
of $\frac{P}{Q}$ if $\mu\left(P,Q\right)<\infty$. Two proper representatives
differ by the multiplication with a multiplicatively invertible holomorphic
germ.
\item We say that a meromorphic germ is \textbf{purely meromorphic} if it
does not admit a representative of the form $\left(P,1\right)$ or
$\left(1,Q\right)$. This is equivalent to requiring that any proper
representative belongs to $\mathfrak{Z}^{2}$. The set of all purely
meromorphic germs is written 
\begin{align*}
\pmero{\mathbf{z}} & :=\left\{ \frac{P}{Q}\,:\,P\left(0\right)=Q\left(0\right)=0\text{ },\text{ }\mu\left(P,Q\right)<\infty\right\} \,.
\end{align*}

\end{enumerate}
\end{defn}
Since $\germ{\mathbf{z}}$ is factorial any meromorphic germ admits
a proper representative. Therefore we can equip the space $\mero{\mathbf{z}}$
with the analytic structure (as in Definition~\ref{def_anal_space})
induced by the quotient map $\tx{Quot}$ restricted to the complement
of $\mu^{-1}\left(\infty\right)$, which is a nonempty open set by
Theorem~\ref{thm:Milnor_analytic}. 
\begin{rem}
\label{rem_mero}~
\begin{enumerate}
\item Working with proper representative is quite natural since a lot of
constructions involving meromorphic germs (\emph{e.g.} analytic maps
$\Lambda$ from some open set $U\subset\mero{\mathbf{z}}$) fail to
be possible for non-proper representatives of meromorphic germs. Besides
if $\quot^{*}\Lambda$ is amply bounded near $\mu^{-1}\left(\infty\right)$
then Theorem~\ref{thm:removable_sing} applies and allows to extend
analytically $\quot^{*}\Lambda$ to the whole $\quot^{-1}\left(U\right)$. 
\item Let us say a few words about the definition of analytic sets of $\mero{\mathbf{z}}$
as in Definition~\ref{def_anal_set}~(4). If analytic sets\textbf{
}of $\mero{\mathbf{z}}$ were defined solely as the vanishing loci
of analytic maps then the set of germs $\frac{P}{Q}$ where, for instance,
$J_{N}\left(P\right)=J_{N}\left(Q\right)=0$ would not form an analytic
set, since the natural map $j_{N}~:~\left(P,Q\right)\mapsto\left(J_{N}\left(P\right),J_{N}\left(Q\right)\right)$
is neither left- nor right-invariant by $\quot$. Yet the vanishing
locus of $j_{N}$ is $\quot$-invariant, in the sense that $J_{N}\left(P\right)=J_{N}\left(Q\right)=0$
implies $J_{N}\left(uP\right)=J_{N}\left(uQ\right)=0$ for any $u\in\germ{\mathbf{z}}$. 
\item The map
\begin{align*}
M\,:\,\germ{\mathbf{z}}\times\left(\germ{\mathbf{z}}\backslash\ww C\right)\backslash\mu^{-1}\left(\infty\right) & \longrightarrow\germ{\mathbf{z}}\times\left(\germ{\mathbf{z}}\backslash\ww C\right)\\
\left(P,Q\right) & \longmapsto\left(P-J_{0}\left(P\right),Q-J_{0}\left(Q\right)\right)
\end{align*}
is analytic and onto, so we can equip $\pmero{\mathbf{z}}$ with the
analytic structure induced by 
\begin{align*}
M_{0} & :=\tx{Quot}\circ M\text{ .}
\end{align*}

\end{enumerate}
\end{rem}

\subsection{\label{sub:Coprime-families}Coprime pairs of $\protect\frml{\mathbf{z}}^{2}$}

~

\subsubsection{Reduction of the proof of Theorem~\ref{thm:Milnor_analytic} when
$k=2$}

~

The result is «only» an exercise in linear algebra. Assume there exists
$\left(h_{1},h_{2}\right)\in\frml{\mathbf{z}}^{k}$ such that 
\begin{align}
f_{1}h_{1} & =f_{2}h_{2}\label{eq:composite}
\end{align}
with $\nu\left(h_{1}\right)<\nu\left(f_{2}\right)$ and $\nu\left(h_{2}\right)<\nu\left(f_{1}\right)$,
where $\nu$ is the valuation associated to the gradation by homogeneous
degree on $\frml{\mathbf{z}}$. We say in that case, for short, that
the collection $\left(f_{1},f_{2}\right)$ is \textbf{composite}.
It follows that $\frak{I}$ is the set of all composite collections;
notice in particular that if $\nu\left(h_{\ell}\right)<\nu\left(f_{3-\ell}\right)$
for some $\ell\in\left\{ 1,2\right\} $ then the relation equally
holds for the other one. We set
\begin{align}
\nu & :=\min\left(\nu\left(f_{1}\right),\nu\left(f_{2}\right)\right)>0,\label{eq:def_of_mu}
\end{align}
and write $f_{\ell}=:\sum_{\mathbf{n}\in\ww N^{m}}f_{\ell,\mathbf{n}}\mathbf{z^{n}}$
with $f_{\ell,\mathbf{n}}:=0$ for $\left|\mathbf{n}\right|<\nu$.
We set likewise $h_{\ell}=:\sum_{\mathbf{n}\in\ww N^{m}}h_{\ell,\mathbf{n}}\mathbf{z^{n}}$
then express the relationship~\eqref{eq:composite} coefficients-wise:
\begin{align}
\left(\forall\mathbf{n}\in\ww N^{m}\right)\,\,\,\,\,\,\,\,\,\sum_{\mathbf{p}+\mathbf{q}=\mathbf{n}}\left(f_{1,\mathbf{p}}\times h_{1,\mathbf{q}}-f_{2,\mathbf{p}}\times h_{2,\mathbf{q}}\right) & =0.\label{eq:composite_coeff}
\end{align}
We say that $\left(f_{1},f_{2}\right)$ is \textbf{composite at rank}
$d\in\ww N$ if, and only if, one can find a tuple $\left(h_{\ell,\mathbf{q}}\right)_{\left|\mathbf{q}\right|\leq d,\,\ell\leq k}$
such that 
\begin{itemize}
\item the previous relation holds for every $\left|\mathbf{n}\right|\leq d+1$,
\item at least one $h_{\ell,\mathbf{q}}$ is non-zero for $\left|\mathbf{q}\right|<\nu$.
\end{itemize}
If a collection $\left(f_{\ell}\right)_{\ell}$ is composite then
it is composite at every rank $d\in\ww N$. If a collection fails
to be composite at some rank then it also does at every other bigger
rank. 

We work with the family of linear spaces $E_{d}:=\pol{\mathbf{z}}{\leq d}^{k}$,
indexed by $d\in\ww N$. Define the linear mapping 
\begin{align*}
\varphi_{d}\text{ : }E_{d}\longrightarrow & \pol{\mathbf{z}}{\leq d+1}\cap\frak{Z}\\
\left(h_{1},h_{2}\right)\longmapsto & J_{d+1}\left(f_{1}h_{1}-f_{2}h_{2}\right)\,.
\end{align*}

\begin{prop}
\label{prop:composite_eqv}The following properties are equivalent:
\begin{enumerate}
\item The family $\left(f_{\ell}\right)_{\ell}$ is composite.
\item The family $\left(f_{\ell}\right)_{\ell}$ is composite at every rank
$d\geq0$.
\item For all $d\geq0$ we have
\begin{align*}
\rank{\varphi_{d}} & \leq\binom{m+d}{d}\,.
\end{align*}

\end{enumerate}
\end{prop}
We deduce that composite collections form a proper analytic subset
of $\frml{\mathbf{z}}^{2}$, expressing for all $d$ the vanishing
of every minor of size greater than $\binom{m+d}{d}$ of a given matrix
of $\varphi_{d}$.

\subsubsection{Proof of (2)$\Rightarrow$(1)}

~

Only the case $f_{1}f_{2}\neq0$ is non-trivial, for which $0<\nu<\infty$
(see~\eqref{eq:def_of_mu}). We prove that if $\left(f_{1},f_{2}\right)$
is composite at every rank $d\in\ww N$ then there exists two formal
power series $h_{1},\text{ }h_{2}\in\frml{\mathbf{z}}$ such that
$f_{1}h_{1}=f_{2}h_{2}$ and $\nu\left(h_{\ell}\right)<\nu\left(f_{3-\ell}\right)$.
Consider the restriction $j_{d}\,:\,E_{d+1}\twoheadrightarrow E_{d}$
of $J_{d}$, each one yielding a canonical section $E_{d}\hookrightarrow E_{d}\oplus\ker j_{d}=E_{d+1}$.
For that identification we have 
\begin{align}
\varphi_{d+1}|_{E_{d}} & =\varphi_{d}\label{eq:matrix_embed}
\end{align}
 and $\varphi_{d+1}\left(\ker j_{d}\right)<\ker j_{d+1}$. Write $\kappa_{d}\text{ : }\ker\varphi_{d+1}\rightarrow\ker\varphi_{d}$
the restriction of $j{}_{d}$. We choose a direct system of one-to-one
maps $\left(\kappa_{d}\,:\,K_{d+1}\hookrightarrow K_{d}\right)_{d\in\ww N}$
of complement subspaces of $\ker\varphi_{d}\cap J_{d}\left(\ker J_{\nu-1}\right)$
in $\ker\varphi_{d}$ in the following fashion:
\begin{itemize}
\item $K_{d}:=\ker\varphi_{d}$ if $d\leq\nu-1$
\item define $K:=\kappa_{d}^{-1}\left(K_{d}\right)$ and, observing that
$K\cap\ker\kappa_{d}=\ker j_{d}\cap\ker\kappa_{d}$, define $K_{d+1}$
as some complement in $K$ of $K\cap\ker\kappa_{d}$.\end{itemize}
\begin{fact}
\label{fact_comp_d_dim_Kd}A collection $\left(f_{\ell}\right)_{\ell}$
is composite at some rank $d$ if, and only if, 
\begin{align*}
\dim K_{d} & \neq0\,.
\end{align*}

\end{fact}
Now if $\left(f_{1},f_{2}\right)$ is composite at every rank $d\in\ww N$
then the inverse limit $K_{\infty}$ is embeddable as a linear subspace
of finite positive dimension in $\frml{\mathbf{z}}^{2}$, solving
our problem.

\subsubsection{Proof of (1)$\Rightarrow$(3)}

~

If $f_{1}=f_{2}=0$ then $\left(f_{1},f_{2}\right)$ is composite
and $\varphi_{d}=0$. Assume now that, say, $f_{1}\neq0$ and $\nu=\nu\left(f_{1}\right)$.
We introduce a notation regarding the number of multi-indexes whose
length satisfy some property $\mathcal{P}$ 
\begin{align*}
\coef{\mathcal{P}} & :=\#\left\{ \mathbf{n}\in\ww N^{m}\,:\mathcal{P}\left(\left|\mathbf{n}\right|\right)\right\} \,.
\end{align*}
For instance the number of multi-indexes of length less than or equal
to $d$ is
\begin{align*}
\coef{\leq d} & =\binom{m+d}{d}\text{ ,}
\end{align*}
which incidentally is the sought bound on $\rank{\varphi_{d}}$ for
composite collections. We observe that $E_{d}$ has dimension $2\coef{\leq d}$,
hence the
\begin{fact}
\label{fact_compd_with_ker}Condition~(3) holds if, and only if,
for all $d\in\ww N$
\begin{align*}
\dim\ker\varphi_{d} & \geq\coef{\leq d}\,.
\end{align*}

\end{fact}
Because $E_{d}\cap\ker J_{d+1-\nu}<\ker\kappa_{d}$ contributes for
a subspace of dimension $2\coef{d+2-\nu\leq\bullet\leq d}$ we are
only interested in studying the dimension of a complement $A_{d}$
of $E_{d}\cap\ker J_{d+1-\nu}$ in $\ker\varphi_{d}$. Notice also
that
\begin{align*}
\left(\forall d<2\left(\nu-1\right)\right)\qquad\,\,\,\,\,\,2\coef{d+2-\nu\leq\bullet\leq d} & =\coef{\leq d}+\left(\coef{\leq d}-2\coef{\leq d+1-\nu}\right)\\
 & \geq\coef{\leq d}+\coef{\leq2\left(\nu-1\right)}-2\coef{\leq\nu-1}\\
 & \geq\coef{\leq d}
\end{align*}
so that (3) holds for all such $d$. In the sequel we thus assume
$d\geq2\left(\nu-1\right)$. As before we choose:
\begin{itemize}
\item two direct systems $\left(\kappa_{d}\,:\,K_{d+1}\hookrightarrow K_{d}\right)_{d\in\ww N}$
and $\left(\kappa_{d}\,:\,A_{d+1}\to A_{d}\right)_{d\in\ww N}$,
\item for each $d$, a subspace $C_{d}<\ker\varphi_{d}\cap\ker J_{\nu-1}$
such that 
\begin{align}
A_{d} & =\kappa_{d}\left(A_{d+1}\right)\oplus C_{d}\,.\label{eq:comp_decomp_ker}
\end{align}
\end{itemize}
\begin{lem}
Assume that $K_{\infty}\neq0$ and $d\geq2\left(\nu-1\right)$. Then
\begin{align*}
\dim A_{d} & \geq\coef{\leq d-2\left(\nu-1\right)}\,.
\end{align*}

\end{lem}
Now we have
\begin{align*}
\ker\varphi_{d}\geq2\coef{d+2-\nu\leq\bullet\leq d}+\coef{\leq d-2\left(\nu-1\right)} & =\coef{\leq d}+\sum_{j=d+2-\nu}^{d}\left(\coef{=j}-\coef{=j-\nu+1}\right)\\
 & \geq\coef{\leq d}
\end{align*}
as expected.
\begin{proof}
We begin with the trivial remark that $J_{d+1-\nu}\left(\frak{Z}K_{\infty}\right)\cap\ker J_{\nu-1}<A_{d}\cap\ker J_{\nu-1}$
then bound from below the dimension of the leftmost space. The «worst-case
scenario» corresponds to the support of $J_{\nu-1}\left(K_{\infty}\right)$
consisting in a single point $\mathbf{n}_{0}$ located on the homogeneous
segment $\left\{ \mathbf{n}\,:\,\left|\mathbf{n}\right|=\nu-1\right\} $.
Yet in that case $\frak{Z}K_{\infty}$ contains an embedding of the
ideal $\mathbf{z}^{\mathbf{n}_{0}}\frak{Z}$ so that
\begin{align*}
\dim J_{d+1-\nu}\left(\frak{Z}K_{\infty}\right)\cap\ker J_{\nu-1}\geq\dim J_{d+1-\nu}\left(\mathbf{z}^{\mathbf{n}_{0}}\frak{Z}\right) & =\sum_{p=1}^{d-2\left(\nu-1\right)}\coef{=j}\\
 & =\coef{\leq d-2\left(\nu-1\right)}-1\,.
\end{align*}
Finally 
\begin{align*}
\dim A_{d} & \geq\left(\coef{\leq d-2\left(\nu-1\right)}-1\right)+\dim K_{d}\\
 & \geq\coef{\leq d-2\left(\nu-1\right)}\,.
\end{align*}

\end{proof}

\subsubsection{Proof of (3)$\Rightarrow$(2)}

~

This is the most involved part of the proof. We start from the decomposition~\eqref{eq:comp_decomp_ker}.
Because the restriction $\hat{\kappa}_{d}:=\kappa_{d}|_{A_{d}}$ is
still a projector we have $A_{d+1}=\ker\hat{\kappa}_{d}\oplus\img{\hat{\kappa}_{d}}$.
Whence:
\begin{fact}
\label{fact_ker_decompos}For any collection $\left(f_{1},f_{2}\right)$
and any $d\in\ww N$
\begin{align*}
A_{d}\oplus\ker\hat{\kappa}_{d} & =A_{d+1}\oplus C_{d}\,.
\end{align*}
In particular
\begin{align*}
\dim A_{d}+\dim\ker\hat{\kappa}_{d} & =\dim A_{d+1}+\dim C_{d}\,.
\end{align*}

\end{fact}
As before we would like to use the image $Z_{p}:=J_{d+p+1-\nu}\left(\frak{Z}C_{d}\right)$
to bound from below $\dim C_{d+p}$, but $Z_{p}$ need not be included
in $C_{d+p}$:
\begin{itemize}
\item an element $h\in Z_{p}$ might not lie in $\ker J_{d+1-\nu}$, and
$J_{d+2-\nu}\left(\varphi_{d+p}\left(h\right)\right)$ has no reason
to vanish,
\item contrary to what happened in the previous paragraph's lemma, we do
not dispose of a global object like $K_{\infty}$ and there is no
reason why the question 
\begin{align*}
0=J\,_{d+p+2-\nu}\left(g\varphi_{d}\left(h\right)\right) & \overset{?}{=}\varphi_{d+p}\left(J_{d+p+1-\nu}\left(gh\right)\right)
\end{align*}
should admit a positive answer when $g\in\frak{Z}$ and $h\in C_{d}$.
\end{itemize}
Yet both conditions can be resolved if on the one hand 
\begin{align*}
p & \geq\hat{d}:=d+1-\nu
\end{align*}
 and on the other hand we replace $\frak{Z}$ by its trace $\frak{Z}_{=p}$
on the space of homogeneous polynomials of degree $d$. 
\begin{lem}
If $\dim C_{d}>0$ then $\frak{Z}_{=p}C_{d}$ contains at least a
vector space of dimension $\coef{=p}$, since each map
\begin{align*}
\tau_{\mathbf{n}}\,:\,C_{d} & \longrightarrow C_{d+p}\\
h & \longmapsto\mathbf{z}^{\mathbf{n}}h
\end{align*}
is a monomorphism whenever $\left|\mathbf{n}\right|=p$. \end{lem}
\begin{proof}
Take any $h\in C_{d}\backslash0$ and show that $h_{\mathbf{n}}:=\mathbf{z}^{\mathbf{n}}h\in C_{d+p}\backslash0$.
Suppose now, by contradiction, that $h_{\mathbf{n}}\notin C_{d+p}$:
there exists $\hat{h}_{\mathbf{n}}\in A_{d+p+1}$ such that $\kappa_{d+p}\left(\hat{h}_{\mathbf{n}}\right)=h_{\mathbf{n}}$.
Because $h_{\mathbf{n}}$ has support inside $\left\{ \mathbf{m}\,:\,m_{j}\geq n_{j}\right\} $
we may require without loss of generality that $\hat{h}_{\mathbf{n}}$
also has. Then $\hat{h}_{\mathbf{n}}=\mathbf{z}^{\mathbf{n}}\hat{h}$
for some $\hat{h}\in A_{d+1}$, as indeed $\mathbf{z}^{\mathbf{n}}\varphi_{d+1}\left(\hat{h}\right)=\varphi_{d+\alpha+1}\left(\hat{h}_{\mathbf{n}}\right)=0$.
We reach the contradiction $h=\kappa_{d}\left(\hat{h}\right)\in C_{d}\backslash0$. 
\end{proof}
Assume now that (2) does not hold: there exists $d\geq2\left(\nu-1\right)$
such that $\dim K_{d+1}<\dim K_{d}$, so that $\dim C_{d}>0$. Then
for all $p\geq\hat{d}$
\begin{align*}
\dim A_{d+p} & \leq\dim A_{\hat{d}}+\sum_{j=\hat{d}}^{p}\left(\dim\ker\hat{\kappa}_{d+j-1}-\coef{=j}\right)\,.
\end{align*}

\begin{lem}
For every $d\geq2\left(\nu-1\right)$ the following estimate holds:
\begin{align*}
\dim\ker\hat{\kappa}_{d} & \leq\coef{=d+2-\nu}\,.
\end{align*}
\end{lem}
\begin{proof}
This is clear from the matrix representation of $\varphi_{d}$ given
if Section~\ref{sub:Computations}.
\end{proof}
Now we can write for $p\geq\hat{d}$
\begin{align*}
\dim A_{d} & \leq\coef{\leq\hat{d}+p}-\coef{\leq p}+c\,,
\end{align*}
where $c$ is some term constant with respect to $p$. Whence 
\begin{align*}
\dim\ker\varphi_{d+p} & \leq2\coef{\hat{d}+p+1\leq\bullet\leq d+p}+\coef{\leq\hat{d}+p}-\coef{\leq p}+c\\
 & \leq\coef{\leq d+p}+\left(\coef{\hat{d}+p+1\leq\bullet\leq d+p}-\coef{\leq p}+c\right)\,.
\end{align*}
Because $\coef{\leq p}=\frac{p^{m}}{p!}+O\left(p^{m-1}\right)$ and
$\coef{\hat{d}+p+1\leq\bullet\leq d+p}=O\left(p^{m-1}\right)$ we
deduce that the parenthesized, rightmost term is eventually negative.

\subsection{\label{sub:Computations}Computations ($k=2$)}

~

The set $\ww N^{m}$ of multi-indexes of given dimension $m$ comes
with the lexicographic order $\lll$, that is $\left(0,2\right)\lll\left(1,0\right)\lll\left(1,1\right)$.
Define the «lexicomogeneous» order $\preceq$ on $\ww N^{m}$ by 
\begin{align*}
\mathbf{a}\preceq\mathbf{b} & \Longleftrightarrow\begin{cases}
 & \left|\mathbf{a}\right|<\left|\mathbf{b}\right|\\
\tx{or} & \left|\mathbf{a}\right|=\left|\mathbf{b}\right|\mbox{ and }\mathbf{a}\lll\mathbf{b}
\end{cases}
\end{align*}
so that \emph{e.g.} $\left(2,0\right)\succeq\left(1,1\right)\succeq\left(0,2\right)\succeq\left(1,0\right)\succeq\left(0,1\right)\succeq\left(0,0\right)$.
For $d\in\ww N$ we form the matrix of multi-indexes 
\begin{align*}
M_{d} & :=\left[\mathbf{n}_{\mathbf{p},\mathbf{q}}\right]_{0\leq\left|\mathbf{p}\right|\leq d\text{ , }0<\left|\mathbf{q}\right|\leq d+1}
\end{align*}
defined by the relation
\begin{align*}
\mathbf{n}_{\mathbf{p},\mathbf{q}} & :=\begin{cases}
\mathbf{q}-\mathbf{p} & \mbox{if }\mathbf{q}-\mathbf{p}\in\ww N^{m}\\
\mathbf{0} & \mbox{otherwise}
\end{cases}
\end{align*}
where $\mathbf{p}$, $\mathbf{q}$ are ordered counter-lexicomogeneously
with $\mathbf{0}\preceq\mathbf{p}\preceq d\oplus\mathbf{0}$ and $\mathbf{0}\prec\mathbf{q}\preceq\left(d+1\right)\oplus\mathbf{0}$.
See Table~\ref{tab:matrix_index} below for an example. This matrix
is upper $\left(d+1\right)\times\left(d+1\right)$ block-triangular,
the blocks $\left[\mathbf{n}_{\mathbf{p},\mathbf{q}}\right]$ corresponding
to constant lengths $\left|\mathbf{p}\right|$ and $\left|\mathbf{q}\right|$.
Its size is $\coef{\leq d+1}\times\coef{\leq d}$. Each $M_{d}$ is
naturally nested within $M_{d+1}$ as the right-bottom-most sub-matrix
formed by the last $d$ blocks both vertically and horizontally, accounting
for the relation~\eqref{eq:matrix_embed}.

\begin{table}[H]
\begin{raggedright}
\hfill{}%
\begin{tabular}{|llllll|lll|l|}
\hline 
$(1,0,0)$ 
 &  &  &  &  &  & $(2,0,0)$ 
 &  &  & $(3,0,0)$ 
\tabularnewline
$(0,1,0)$ 
 & $(1,0,0)$ 
 &  &  &  &  & $(1,1,0)$ 
 & $(2,0,0)$ 
 &  & $(2,1,0)$ 
\tabularnewline
$(0,0,1)$ 
 &  & $(1,0,0)$ 
 &  &  &  & $(1,0,1)$ 
 &  & $(2,0,0)$ 
 & $(2,0,1)$ 
\tabularnewline
 & $(0,1,0)$ 
 &  & $(1,0,0)$ 
 &  &  & $(0,2,0)$ 
 & $(1,1,0)$ 
 &  & $(1,2,0)$ 
\tabularnewline
 & $(0,0,1)$ 
 & $(0,1,0)$ 
 &  & $(1,0,0)$ 
 &  & $(0,1,1)$ 
 & $(1,0,1)$ 
 & $(1,1,0)$ 
 & $(1,1,1)$ 
\tabularnewline
 &  & $(0,0,1)$ 
 &  &  & $(1,0,0)$ 
 & $(0,0,2)$ 
 &  & $(1,0,1)$ 
 & $(1,0,2)$ 
\tabularnewline
 &  &  & $(0,1,0)$ 
 &  &  &  & $(0,2,0)$ 
 &  & $(0,3,0)$ 
\tabularnewline
 &  &  & $(0,0,1)$ 
 & $(0,1,0)$ 
 &  &  & $(0,1,1)$ 
 & $(0,2,0)$ 
 & $(0,2,1)$ 
\tabularnewline
 &  &  &  & $(0,0,1)$ 
 & $(0,1,0)$ 
 &  & $(0,0,2)$ 
 & $(0,1,1)$ 
 & $(0,1,2)$ 
\tabularnewline
 &  &  &  &  & $(0,0,1)$ 
 &  &  & $(0,0,2)$ 
 & $(0,0,3)$ 
\tabularnewline
\hline 
 &  &  &  &  &  & $(1,0,0)$ 
 &  &  & $(2,0,0)$ 
\tabularnewline
 &  &  &  &  &  & $(0,1,0)$ 
 & $(1,0,0)$ 
 &  & $(1,1,0)$ 
\tabularnewline
 &  &  &  &  &  & $(0,0,1)$ 
 &  & $(1,0,0)$ 
 & $(1,0,1)$ 
\tabularnewline
 &  &  &  &  &  &  & $(0,1,0)$ 
 &  & $(0,2,0)$ 
\tabularnewline
 &  &  &  &  &  &  & $(0,0,1)$ 
 & $(0,1,0)$ 
 & $(0,1,1)$ 
\tabularnewline
 &  &  &  &  &  &  &  & $(0,0,1)$ 
 & $(0,0,2)$ 
\tabularnewline
\cline{7-10} 
 &  &  &  &  & \multicolumn{1}{l}{} &  &  &  & $(1,0,0)$ 
\tabularnewline
 &  &  &  &  & \multicolumn{1}{l}{} &  &  &  & $(0,1,0)$ 
\tabularnewline
 &  &  &  &  & \multicolumn{1}{l}{} &  &  &  & $(0,0,1)$ 
\tabularnewline
\hline 
\end{tabular}\hfill{}
\par\end{raggedright}

\protect\caption{\label{tab:matrix_index}The matrix $M_{2}$ when $m=3$. Null entries
are not shown. Homogeneous blocks are drawn in solid lines.}
\end{table}

We build a matrix $M_{d}\left(f_{1},f_{2}\right)$ from $M_{d}$,
that of the mapping $\varphi_{d}$. It has twice as many rows and
is obtained from the latter by replacing each column $\left(\mathbf{n}_{\mathbf{p},\mathbf{q}}\right)_{0<\left|\mathbf{q}\right|\leq d+1}$
with the pair of columns whose respective entries are $\left(-1\right)^{\ell+1}f_{\ell,\mathbf{n}_{\mathbf{p},\mathbf{q}}}$
for $\ell\in\left\{ 1,2\right\} $. 

\begin{table}[H]
\begin{raggedright}
\hfill{}%
\begin{tabular}{|llllll|llll|ll|}
\hline 
0 & -1 & 0 & 0 & 0 & 0 & -1 & 0 & 0 & 0 & 1 & 0\tabularnewline
0 & -1 & 0 & -1 & 0 & 0 & -2 & 3 & -1 & 0 & 2 & 0\tabularnewline
0 & 0 & 0 & -1 & 0 & -1 & -1 & 3 & -2 & 3 & 1 & 0\tabularnewline
0 & 0 & 0 & 0 & 0 & -1 & 0 & 0 & -1 & 3 & 0 & 0\tabularnewline
\hline 
0 & 0 & 0 & 0 & 0 & 0 & 0 & -1 & 0 & 0 & -1 & 0\tabularnewline
0 & 0 & 0 & 0 & 0 & 0 & 0 & -1 & 0 & -1 & -2 & 3\tabularnewline
0 & 0 & 0 & 0 & 0 & 0 & 0 & 0 & 0 & -1 & -1 & 3\tabularnewline
\cline{7-12} 
0 & 0 & 0 & 0 & 0 & \multicolumn{1}{l}{0} & 0 & 0 & 0 & 0 & 0 & -1\tabularnewline
0 & 0 & 0 & 0 & 0 & \multicolumn{1}{l}{0} & 0 & 0 & 0 & 0 & 0 & -1\tabularnewline
\hline 
\end{tabular}\hfill{}
\par\end{raggedright}

\protect\caption{\label{tab:matrix_fumctions}The matrix $M_{2}\left(\left(x-1\right)\left(x+y\right)^{2}\text{ , }\left(1-3y\right)\left(x+y\right)\right)$
when $m=2$. Its rank is $6=\rho\left(2,3,2\right)$ and $\nu=1$.}
\end{table}

\begin{rem}
When $m=2$ the first non-trivial case occurs for $d=2\left(\nu-1\right)$.
The condition $\dim\ker\varphi_{d}\geq\binom{m+d}{d}$ is equivalent
to the vanishing of the determinant of the matrix occupying the homogeneous
bloc $\left(2\nu-1\right)\times\nu-1$. This matrix is nothing but
the Sylvester matrix (up to columns/rows permutation) of the homogeneous
part of degree $\nu$ of $\left(f_{1},f_{2}\right)$. Indeed if $\left(f_{1},f_{2}\right)$
is composite then also is its part of lowest homogeneous degree, which
can be written $\frac{1}{x^{\nu}}\left(f_{1}\left(1,t\right),f_{2}\left(1,t\right)\right)$
with $t=\frac{y}{x}$. This situation persists for other values of
$m$ as the Macaulay matrix of the homogeneous part of $\left(f_{1},f_{2}\right)$
is embedded in $M_{2\left(\nu-1\right)}\left(f_{1},f_{2}\right)$.
I am aware that the topic discussed here is well-known in the setting
of commutative algebra on the $\ww C$-module of polynomials, but
I have not been able to locate a similar construction in the framework
of formal power series.
\end{rem}

\subsection{\label{sub:general_case}Coprime families of $\protect\frml{\mathbf{z}}^{k}$}

~

The general case is not more difficult. Proposition~\ref{prop:composite_eqv}
holds actually for any $k\geq2$ as we explain now.
\begin{fact}
The collection $f=\left(f_{\ell}\right)_{1\leq\ell\leq k}$ is composite
if, and only if, there exists $\left(h_{1},\ldots,h_{k}\right)\in\frml{\mathbf{z}}^{k}$
such that $f_{\ell}h_{\ell}$ does not depend on $\ell$ and for every
$1\leq\ell\leq k$
\begin{align*}
\nu\left(h_{\ell}\right) & <\varepsilon_{\ell}\left(f\right):=\sum_{j\neq\ell}\nu\left(f_{j}\right)-\nu\left(\frac{\prod_{j=1}^{k}f_{j}}{\tx{gcd}\left(f\right)^{k-1}\tx{lcm}\left(f\right)}\right)\,.
\end{align*}
Notice that when $k=2$ the right-hand side equals $\nu\left(f_{3-\ell}\right)$
as before.
\end{fact}
The result follows from the study of the linear maps 
\begin{align*}
\hat{\varphi}_{d}\,:\,E_{d} & \longrightarrow\frml{\mathbf{z}}_{\leq d+1}^{\times\binom{k}{2}}\\
\left(h_{1},\ldots,h_{k}\right) & \longmapsto\left(J_{d+1}\left(f_{p}h_{p}-f_{q}h_{q}\right)\right)_{1\leq p<q\leq k}
\end{align*}
and the property of being composite is expressed in terms of the dimension
of $\hat{K}_{d}:=\nf{\ker\hat{\varphi}_{d}}{\ker J_{\varepsilon}}$
where the extra parameter
\begin{align*}
\varepsilon & :=\min_{\ell}\varepsilon_{\ell}\left(f\right)
\end{align*}
may in general not be equal to $\nu=\min_{\ell}\nu\left(f_{\ell}\right)$.
This in turn is equivalent to $\dim\ker\hat{\varphi}_{d}\geq\left(k-1\right)\binom{m+d}{d}$,
as can be recovered from repeating the arguments developed in the
case $k=2$.

\section{\label{sec:application_ODE}Application to differential equations}

In this section every space of germs is given the factorial topology
(Definition~\ref{def_useful_topo}). In particular $\germ{\mathbf{z}}$
becomes a metric space.

\subsection{Analyticity of the flow of a vector field: proof of Theorem~C}

~
\begin{thm}
\label{thm:flow_is_strong_anal}Fix $m\in\ww N$ and consider the
space $\mathtt{VF}$ of germs at $\mathbf{0}\in\ww C^{m}$ of a holomorphic
vector field, identified with $\germ{\mathbf{z}}^{m}$. For $X\in\mathtt{VF}$
we name $\flow X{}$ the flow of $X$, that is the unique germ of
a holomorphic mapping near $\left(\mathbf{0},0\right)$ 
\begin{align*}
\flow X{}\,:\,\ww C^{m}\times\ww C & \longrightarrow\ww C^{m}\\
\left(\mathbf{p},t\right) & \longmapsto\Phi_{X}^{t}\left(\mathbf{p}\right)
\end{align*}
that is solution of the differential system 
\begin{align*}
\begin{cases}
\dot{\mathbf{z}}\left(\mathbf{p},t\right) & =X\left(\mathbf{z}\left(\mathbf{p},t\right)\right)\\
\mathbf{z}\left(\mathbf{p},0\right) & =\mathbf{p}\,.
\end{cases}
\end{align*}
Then the <<flow mapping>>
\begin{align*}
\mathtt{VF} & \longrightarrow\germ{\mathbf{z},t}^{m}\\
X & \longmapsto\flow X{}\,,
\end{align*}
where the target space is also given the factorial topology (see Remark~\ref{rem_facto_topo_vector}),
is strongly analytic.
\end{thm}
I believe this theorem should hold for every useful topology. Yet
the proof of this result is simplified by the fact that we know the
norms of differentiation operators in the factorial topology.
\begin{proof}
Let $\chi\in\mathcal{O}\left(\left(\ww C^{q},0\right)\to\mathtt{VF}\right)$
as in Definition~\ref{def_complete_holo}; for the sake of keeping
notations unobtrusive we write $\chi_{\mathbf{x}}$ instead of $\chi\left(\mathbf{x}\right)$.
From the differential system $\dot{\mathbf{z}}=\chi_{\mathbf{x}}\left(\mathbf{z}\right)$
we build a new differential system incorporating the extra parameter
$\mathbf{x}\in\left(\ww C^{q},0\right)$ in the new variable $\mathbf{w}=\mathbf{z}\oplus\mathbf{x}$
\begin{align*}
\begin{cases}
\dot{\mathbf{w}}\left(\mathbf{p},\mathbf{x},t\right) & =\chi_{\mathbf{x}}\left(\mathbf{z}\right)\\
\mathbf{w}\left(\mathbf{p},\mathbf{x},0\right) & =\left(\mathbf{p},\mathbf{x}\right)
\end{cases} & .
\end{align*}
By assumption on $\chi$ the vector field $\left(\mathbf{z},\mathbf{x}\right)\mapsto\chi_{\mathbf{x}}\left(\mathbf{z}\right)$
is holomorphic on a neighborhood of $\mathbf{0}\oplus\mathbf{0}$.
Therefore the Cauchy-Lipschitz theorem applies to this system and
yields a flow which is a germ of a holomorphic mapping with respect
to $\left(\mathbf{p},\mathbf{x},t\right)$. In particular $\chi^{*}\flow{\bullet}{}\,:\,\mathbf{x}\mapsto\flow{\chi_{\mathbf{x}}}{}$
belongs to $\mathcal{O}\left(\left(\ww C^{q},0\right)\to\germ{\mathbf{z},t}\right)$
and the flow mapping is quasi-strongly analytic in the sense of Definition~\ref{def_complete_holo}. 

To complete the proof that the flow is strongly analytic on $\mathtt{VF}$
we need to establish that $\chi^{*}\flow{\bullet}{}$ is amply bounded
for the factorial topology. In fact we show that $\flow{\bullet}{}$
is amply bounded on the whole $\mathtt{VF}$, so that Theorem~\ref{thm:removable_sing}
will automatically apply. This can be achieved using Lie's formula
for the flow:
\begin{align*}
\flow Xt\left(\mathbf{p}\right) & =\left(\sum_{k=0}^{\infty}\frac{t^{k}}{k!}X\cdot^{k}\id\right)\left(\mathbf{p}\right)
\end{align*}
with uniform convergence, as a power series in $\left(\mathbf{p},t\right)$,
on a neighborhood of $\mathbf{0}\times0$. The iterated Lie derivatives
are given for germs of functions by $X\cdot^{0}f=f$ and $X\cdot^{k+1}:=X\cdot\left(X\cdot^{k}f\right)$,
and then extended to vectors of functions by acting component-wise.
The factorial topology on $\germ{\mathbf{z},t}$ is induced by $\mathbf{a}\left(\bullet\right)=a\left(\bullet\right)\oplus\left(k!^{-\bullet}\right)_{k\in\ww N}$
while $a\left(\bullet\right)$ induces the factorial topology on $\mathtt{VF}$.
We thus have for $\alpha>0$
\begin{align*}
\norm{\flow X{}}{\mathbf{a}\left(\alpha\right)} & =\sum_{k=0}^{\infty}\frac{1}{\left(k!\right)^{\alpha+1}}\norm{X\cdot^{k}\id}{a\left(\alpha\right)}
\end{align*}
and want to show that this quantity is uniformly bounded when $X$
lies in some $a\left(\beta\right)$-ball for $\alpha>\beta>0$. In
all the following the letter $C$ denotes a positive number depending
only on $\alpha$ and $\beta$, whose exact value does not matter
much and may vary from place to place.

Let us show the claim when $m=1$, the general case being not very
more difficult but hampered with cumbersome notations. We write 
\begin{align*}
X\left(z\right) & :=\xi\left(z\right)\pp{}z\,.
\end{align*}
It is easy to check that $X\cdot^{k}z$ is of the form, for $k>0$,
\begin{align*}
X\cdot^{k}z & =\xi\pp{}z\left(\xi\pp{}z\left(\cdots\xi\right)\right)\\
 & =\xi\sum_{\mathbf{j}\in J_{k}}\prod_{\ell}\left(\pp{^{j_{\ell}}\xi}{z^{j_{\ell}}}\right)
\end{align*}
with $J_{k}$ a finite set of multi-indexes $\mathbf{j}\in\ww N^{k-1}$
of length $\left|\mathbf{j}\right|=k-1$. Proposition~\ref{prop:facto_useful}~(3)
taken into account we can write
\begin{align*}
\norm{\left(\pp{^{j_{\ell}}\xi}{z^{j_{\ell}}}\right)}{a\left(\alpha\right)} & \leq C\exp\left(\sigma\left(\alpha-\beta\right)j_{\ell}^{\nf{\beta+1}{\alpha+1}}-\left(\beta+1\right)j_{\ell}\right)\sigma^{\left(\alpha+1\right)j_{\ell}}j_{\ell}^{\left(\beta+1\right)j_{\ell}}\norm{\xi}{a\left(\beta\right)}
\end{align*}
for some positive $\sigma$. We use Lemma~\ref{lem:facto_estim},
more specifically the relation $\left|\mathbf{j}\right|!\geq\mathbf{j}!$,
and relation~\eqref{eq:stirling_estim} to derive (with the convention
$0^{0}=1$)
\begin{align*}
\norm{\flow X{}}{a\left(\alpha\right)} & \leq\sum_{k=0}^{\infty}\left(C\norm{\xi}{a\left(\beta\right)}\right)^{k}\sum_{\left|\mathbf{j}\right|=k-1}\left(\prod_{\ell}j_{\ell}^{j_{\ell}}\right)^{\beta-\alpha}\,.
\end{align*}
 Define 
\begin{align*}
E_{n,k} & :=\sum_{\mathbf{j}\in\ww N^{k}\,,\,\left|\mathbf{j}\right|=n}\prod_{\ell}j_{\ell}^{\left(\beta-\alpha\right)j_{\ell}}\,,
\end{align*}
the $n^{\tx{th}}$ Taylor coefficient of $\varphi^{k}$, where $\varphi$
is the entire function
\begin{align*}
\varphi\,:\,z\in\ww C & \longmapsto\sum_{n\geq0}z^{n}n^{\left(\beta-\alpha\right)n}\,.
\end{align*}
For any $\rho>0$ Cauchy's formula yields
\begin{align*}
\rho^{n}E_{n,k} & \leq\varphi\left(\rho\right)^{k}\\
E_{k-1,k-1} & \leq\left(\frac{\varphi\left(\rho\right)}{\rho}\right)^{k-1}\,.
\end{align*}
Consequently 
\begin{align*}
\norm{\flow X{}}{a\left(\alpha\right)} & \leq\sum_{k=0}^{\infty}\left(C\norm{\xi}{a\left(\beta\right)}\right)^{k}\,,
\end{align*}
with convergence of the right-hand side if $\norm{\xi}{a\left(\beta\right)}$
is small enough. 

We have just proved that the flow is amply bounded near the null vector
field. Let us show now that this property implies that of the flow
being amply bounded near any point of $\mathtt{VF}$. Since for any
$\lambda\in\ww C$ 
\begin{align*}
\flow{\lambda X}{\bullet} & =\flow X{\lambda\bullet}
\end{align*}
we can rescale any given vector field $X$ so that $\lambda X$ belongs
to a convenient neighborhood of the null vector field. Since the composition
$\lambda^{*}\,:\,\flow X{\bullet}\mapsto\flow X{\lambda\bullet}$
is continuous (Proposition~\ref{prop:facto_useful}) we derive the
sought property of ample boundedness near $X$.
\end{proof}

\subsection{\label{sec:foliations}Solvability of first-order, ordinary differential
equations: proof of Corollary~C}

~

In this section we consider a germ, defined near $\left(0,0\right)\in\ww C^{2}$,
of an ordinary differential equation
\begin{align*}
y' & =f\left(x,y\right)\tag{\ensuremath{\lozenge}}
\end{align*}
where $f\in\mero{x,y}$ is a germ of a meromorphic function. We refer
to Remark~\ref{rem_mero} for the description of the analytic structure
we put on $\mero{x,y}$.

The geometric object underlying this equation is the foliation it
defines: local solutions of $(\lozenge)$ define the local leaves
of the foliation. If $\left(P,Q\right)$ is a proper representative
of $f$ the singular locus of the germ of a foliation is $P^{-1}\left(0\right)\cap Q^{-1}\left(0\right)$
and consists of at most the origin. As was discussed in the introduction
we are only concerned with singular germs of foliations whose first
jet vanishes at $\left(0,0\right)$, that is 
\begin{itemize}
\item $f$ is purely meromorphic (\emph{i.e. }$P\left(0,0\right)=Q\left(0,0\right)=0$),
see Definition~\ref{def_pure_mero},
\item the matrix
\begin{align*}
L\left(P,Q\right) & :=\left[\begin{array}{cc}
\frac{\partial P}{\partial x}\left(0,0\right) & \frac{\partial P}{\partial y}\left(0,0\right)\\
\frac{\partial Q}{\partial x}\left(0,0\right) & \frac{\partial Q}{\partial y}\left(0,0\right)
\end{array}\right]
\end{align*}
is zero. In particular it is non-reduced in the sense that $L\left(P,Q\right)$
is nilpotent.
\end{itemize}
Notice that these conditions do not depend on the choice of a proper
representative of $f$. In particular the set $\mathtt{ZLP}$ (for
Zero Linear Part) of those foliations can be identified with the analytic
subset $\mero{x,y}_{1}$ of $\pmero{x,y}$ defined by the vanishing
locus of
\begin{align*}
\frac{P}{Q}\in\mero{x,y} & \longmapsto\left(J_{1}\left(P\right),J_{1}\left(Q\right)\right)\in\pol{x,y}{}^{2}\,,
\end{align*}
 which is the range of the analytic mapping
\begin{align*}
M_{1}\,:\,\left(P,Q\right)\in\germ{x,y}\times\left(\germ{x,y}\backslash\left\{ 0\right\} \right)\backslash\mu^{-1}\left(\infty\right) & \longmapsto\frac{P-J_{1}\left(P\right)}{Q-J_{1}\left(Q\right)}\in\mero{x,y}_{1}\,,
\end{align*}
and we equip from now $\mathtt{ZLP}$ with the analytic structure
induced by this map (Definition~\ref{def_anal_space}).

\bigskip{}

The usual process when one is faced with such a singularity is to
study a foliated complex surface obtained by pulling-back the germ
of a foliation by the standard blow-up of the origin. This foliation
has <<simpler>> singularities, and the repetition of the process
at any successive singular point eventually stops at the stage where
every singular point is reduced~\cite{Seiden}. We are interested
in those foliations for which the reduction procedure stops after
the first step. Let us describe this process in more details. The
complex manifold $\mathcal{M}$ obtained by blowing-up the origin
of $\ww C^{2}$ is defined by the gluing of 
\begin{align*}
\mathcal{M}_{x} & :=\left\{ \left(x,u\right)\in\ww C^{2}\right\} \\
\mathcal{M}_{y} & :=\left\{ \left(v,y\right)\in\ww C^{2}\right\} 
\end{align*}
through the map
\begin{align*}
\mathcal{M}_{x}\backslash\left\{ u=0\right\}  & \longrightarrow\mathcal{M}_{y}\backslash\left\{ v=0\right\} \\
\left(x,u\right) & \longmapsto\left(\frac{1}{u},xu\right)\,.
\end{align*}
The blow-up morphism $\pi$ is therefore given in these charts by
\begin{align*}
\pi_{x}\,:\,\mathcal{M}_{x} & \longrightarrow\ww C^{2}\\
\left(x,u\right) & \longmapsto\left(x,xu\right)
\end{align*}
and
\begin{align*}
\pi_{y}\,:\,\mathcal{M}_{y} & \longrightarrow\ww C^{2}\\
\left(v,y\right) & \longmapsto\left(vy,y\right)\,.
\end{align*}
The exceptional divisor $\mathcal{E}$ of $\mathcal{M}$ is the projective
line $\pi^{-1}\left(0\right)$ of self-intersection $-1$, and $\mathcal{M}\backslash\mathcal{E}$
is biholomorphic to $\ww C^{2}\backslash\left\{ 0\right\} $. The
pulled-back foliation $\mathcal{F}$ is well-defined on a neighborhood
of $\mathcal{E}$ in $\mathcal{M}$, which we still call $\mathcal{M}$
for the sake of simplicity. $\mathcal{F}$ is said dicritic if $\mathcal{E}$
is not a leaf, otherwise $\mathcal{F}$ only admits a finite number
of singularities on $\mathcal{E}$. A finite set of algebraic conditions
in the second jet of $P$ and $Q$ can be included to ensure that
$\mathcal{F}$ is reduced and non-dicritic:
\begin{prop}
\label{prop:red-1_is_zariski}The set $\mathtt{RND}$ (for Reduced
Non-Dicritic) of non-dicritic foliations reduced after one blow-up
is Zariski-full and open in $\mathtt{ZLP}$. Besides the following
conditions describe a Zariski-full open set $\mathtt{RND}^{*}\subset\mathtt{RND}$:
\begin{itemize}
\item $\mathcal{F}$ admits exactly three distinct singular points on $\mathcal{E}$
and $\xi_{*}:=\left(0,0\right)\in\mathcal{M}_{x}$ is not among them,
\item the linear part of the foliation at every singular point admits two
non-zero eigenvalues.
\end{itemize}
\end{prop}
We postpone the proof of this property till Subsection~\ref{sub:proof_prop_red-1_zariski}
below. Since $\mathcal{F}$ is transverse, outside $\mathcal{E}$
and maybe three regular analytic curves\footnote{They correspond to the local separatrices near singular points on
$\mathcal{E}$.}, to the fibers of the natural projection $\sigma\,:\,\left(x,u\right)\in\mathcal{M}_{x}\mapsto u$,
we can build the holonomy group $\hol f$ of $\mathcal{F}$ based
at $\xi_{*}$, which is canonically identified with a subgroup of
$\diff{}$ spanned by two generators $\holo 1$ and $\holo 2$. Simply
identifying the group $\left\langle \holo 1,\holo 2\right\rangle $
with the pair $\left(\holo 1,\holo 2\right)\in\diff{}^{2}$ will not
do, since one can follow a loop in $\mathtt{RND}^{*}$ in order to
exchange $\holo 1$ and $\holo 2$. This difficulty is overcome by
considering the quotient $\nf{\diff{}^{2}}{\frak{S}_{2}}$ of the
pairs $\left(a,b\right)$ \emph{modulo} the action of the $2$-symmetric
group
\begin{align*}
\left(a,b\right) & \longmapsto\left(b,a\right)\,.
\end{align*}
This quotient space is endowed with the induced analytic structure.
\begin{thm}
\label{thm:holo_is_analytic}The map
\begin{align*}
\hol{\bullet}\,:\,\mathtt{RND}^{*} & \longrightarrow\nf{\diff{}^{2}}{\frak{S}_{2}}\\
f & \longmapsto\left\{ \holo 1,\holo 2\right\} 
\end{align*}
is strongly analytic. Its image contains the open set $\left\{ \holo 1'\left(0\right)\notin\exp\left(\ii\ww R\right)\,\,\mbox{or}\,\,\holo 2'\left(0\right)\notin\exp\left(\ii\ww R\right)\right\} $.\end{thm}
\begin{rem}
The map $\hol{\bullet}$ is not onto according to a result by \noun{Y.~Il'Yashenko}~\cite{Ilya},
which states that every group $\left\langle \holo 1,\holo 2\right\rangle $
generated by non-linearizable, small-divisors-impaired diffeomorphisms
(such that $\holo 1\circ\holo 2$ is equally not linearizable) cannot
be realized as the projective holonomy of any germ of a foliation
at the origin of $\ww C^{2}$. The range of $\hol{\bullet}$ however
contains a non-empty open set, so we say for short that it is \textbf{quasi-onto}.
\end{rem}
The proof of the theorem splits into three parts:
\begin{itemize}
\item first we recall how the <<projective>> holonomy group of $\mathcal{F}$
is built in Subsection~\ref{sub:holonomy},
\item we then study in Subsection~\ref{sub:holo_is_loc_anal} the local
strong analyticity of the holonomy mapping with respect to $\frac{P}{Q}$
by restricting ourselves to a small neighborhood $V\subset\mathtt{RND}^{*}$
of some $\frac{P^{0}}{Q^{0}}$,
\item we perform next in Subsection~\ref{sub:holo_is_anal_and_onto} its
analytic continuation to obtain a globally strongly analytic, quasi-onto
map on $\mathtt{RND}^{*}$.
\end{itemize}
This theorem is the key to Theorem~D: for $\mathtt{RND}$ foliations
the differential equation~$(\lozenge)$ is solvable by quadratures
(in the sense of Liouville) only if its image by $\hol{\bullet}$
is solvable. It is indeed classical that a <<solution>> lying in
a Liouvillian extension of the differential field $\mero{x,y}$ must
have a solvable monodromy, property which translates into a solvable
projective holonomy. Using Corollary~A we then deduce that the former
condition defines a proper analytic set of $\mathtt{RND}$ (since
the image of $\hol{\bullet}$ contains a non-empty open set) whose
complement is a Zariski-full open set of $\mathtt{ZLP}$ and consists
only of non-solvable equations.

\subsubsection{\label{sub:proof_prop_red-1_zariski}Proof of Proposition~\ref{prop:red-1_is_zariski}}

~

Let $f\in\mathtt{ZLP}$ be given and fix a proper representative $\left(P,Q\right)$
of $f$; for the sake of clarity we write $P\left(x,y\right):=\sum_{p+q>1}P_{p,q}x^{p}y^{q}$
and $Q\left(x,y\right):=\sum_{p+q>1}Q_{p,q}x^{p}y^{q}$. To compute
the pull-back $\mathcal{F}$ of the foliation defined by $\left(\lozenge\right)$
we consider the vector field
\begin{align*}
X_{\mathcal{F}}\left(x,y\right) & :=Q\left(x,y\right)\pp{}x+P\left(x,y\right)\pp{}y
\end{align*}
whose integral curve $t\in\left(\ww C,0\right)\mapsto\flow Xt\left(x,y\right)$
coincides with the local leaf of $\mathcal{F}$ passing through $\left(x,y\right)$.
Its pull-back by \emph{e.g. }$\pi_{x}$ is given by
\begin{align*}
\pi_{x}^{*}X_{\mathcal{F}} & =xQ\left(x,xu\right)\pp{}x+\left(P\left(x,xu\right)-uQ\left(x,xu\right)\right)\pp{}u
\end{align*}
whose components belong to $\germ{x,u}$. The affine trace $\left\{ x=0\right\} $
of $\mathcal{E}$ on $\mathcal{M}_{x}$ is a line of singularities
for $\pi_{x}^{*}X_{\mathcal{F}}$, since both of its components belong
to the ideal $\mathcal{I}\left(x^{2}\right)<\germ{x,u}$ spanned by
$x^{2}$. We indeed have
\begin{align*}
xQ\left(x,xu\right) & =x^{3}\left(Q_{0,2}u^{2}+Q_{1,1}u+Q_{2,0}\right)+O\left(x^{4}\right)\\
uQ\left(x,xu\right)-P\left(x,xu\right) & =x^{2}\left(Q_{0,2}u^{3}+\left(Q_{1,1}-P_{0,2}\right)u^{2}+\left(Q_{2,0}-P_{1,1}\right)u-P_{2,0}\right)+O\left(x^{3}\right)\,.
\end{align*}
If $\mathcal{F}$ is dicritic then 
\begin{align*}
uQ\left(x,xu\right)-P\left(x,xu\right) & \in\mathcal{I}\left(x^{3}\right)\,.
\end{align*}
If we require that 
\begin{align}
Q_{0,2} & =0\label{eq:dicritic_cond_1}
\end{align}
does not hold then we ensure that the foliation is non-dicritic. Notice
that $\mathcal{F}$ is given in the chart $\mathcal{M}_{x}$ by the
holomorphic vector field with isolated singularities
\begin{align}
X & :=\frac{1}{x^{2}}\pi^{*}X_{\mathcal{F}}\,.\label{eq:blown-up_vector-field}
\end{align}
This vector field admits three singularities (counted with multiplicity)
on $\mathcal{E}$ whose affine coordinates are given by the roots
of the polynomial
\begin{align}
\phi\left(u\right):= & -Q_{0,2}u^{3}+\left(P_{0,2}-Q_{1,1}\right)u^{2}+\left(P_{1,1}-Q_{2,0}\right)u+P_{2,0}\,.\label{eq:dicritic_poly}
\end{align}
 The discriminant of $\phi$ is a polynomial in the variables $\left(P_{\mathbf{n}},Q_{\mathbf{n}}\right)_{\left|\mathbf{n}\right|=2}$:
requiring that it vanishes describes the analytic set defined by
\begin{align}
Q_{0,2}\left(27P_{0,2}Q_{0,2}-4\left(Q_{2,0}-P_{1,1}\right)^{3}-18\left(Q_{1,1}-P_{0,2}\right)\left(Q_{2,0}-P_{1,1}\right)P_{2,0}\right)\nonumber \\
+\left(Q_{1,1}-P_{0,2}\right)^{2}\left(\left(Q_{2,0}-P_{1,1}\right)^{2}+4\left(Q_{1,1}-P_{0,2}\right)P_{2,0}\right) & =0\,.\label{eq:dicritic_cond_3}
\end{align}
Under the condition that this quantity does not vanish we particularly
derive that the linearized part of $\mathcal{F}$ at any singular
point of $\mathcal{E}\cap\mathcal{M}_{x}$, the matrix
\begin{align*}
\left[\begin{array}{cc}
Q_{0,2}u^{2}+Q_{1,1}u+Q_{2,0} & \star\\
0 & \phi'\left(u\right)
\end{array}\right] & ,
\end{align*}
cannot be nilpotent. The condition $Q_{0,2}\neq0$ also ensures that
$\mathcal{F}$ has no singularity at $\left(0,\infty\right)=\mathcal{E}\backslash\mathcal{M}_{x}$,
and the foliation is reduced. As announced above, if we assume that
\begin{align}
P_{2,0} & =0\label{eq:dicritic_cond_2}
\end{align}
is not fulfilled then $0$ is never a root of $\phi$. 

The union of the analytic sets defined by each one of the three conditions~\ref{eq:dicritic_cond_1},
\ref{eq:dicritic_cond_3} and \ref{eq:dicritic_cond_2} forms a proper
analytic subset $\Omega$ of $\left\{ \left(P,Q\right)\,:\,J_{1}\left(P\right)=J_{1}\left(Q\right)=0\right\} $.
Since any other proper representative of $\frac{P}{Q}$ defines the
same foliation on a (maybe smaller) neighborhood of $\mathcal{E}$,
in particular the roots (and their multiplicity) of the polynomial
$\phi$ remain unchanged. Therefore $\Omega$ corresponds to a proper
analytic set of $\mathtt{ZLP}$, whose complement $\mathtt{RND}^{*}$
is Zariski-full.

\subsubsection{\label{sub:holonomy}Building the holonomy}

~

\begin{figure}[H]
\includegraphics[width=0.75\columnwidth]{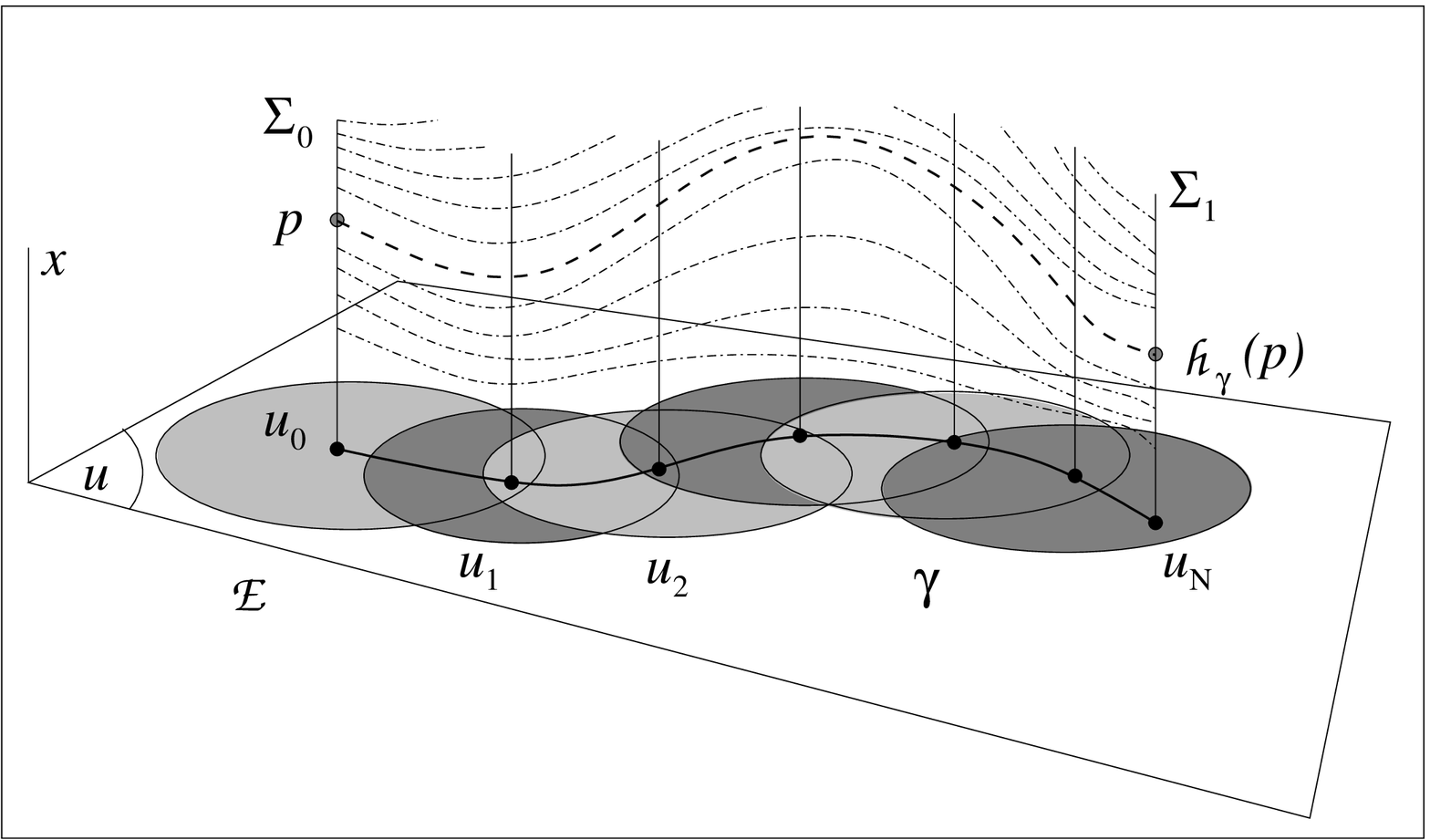}

\protect\caption{\label{fig:holo_construction}Computing the holonomy by composing
local flows.}
\end{figure}
Choose a vector field $X=a\pp{}x+b\pp{}u$ holomorphic on a neighborhood
of $\mathcal{E}$ and a path $\gamma\,:\,\left[0,1\right]\to\mathcal{E}\backslash\sing X$
of base-point $\xi_{*}=\left(0,u_{0}\right)$, such that $X$ is tangent
to $\mathcal{E}$. We require that $X$ be transverse to the lines
$\left\{ u=\tx{cst}\right\} $ outside its singular locus and define
\begin{align*}
\hat{X} & :=\frac{a}{b}\pp{}x+\pp{}u\,,
\end{align*}
whose restriction to $\mathcal{E}$ is holomorphic outside $\sing X$.
It defines the same foliation as $X$. The image of $\gamma$ does
not meet any singularity of $X$ therefore at each point $\gamma\left(s\right)$
there exists a polydisc $\Sigma_{s}\times D_{s}$ on which the mapping
\begin{align}
h_{s}\,:\,\left(x,u\right) & \longmapsto\flow{\hat{X}}u\left(x,u\left(s\right)\right)=\left(\eta_{s}\left(x,u\right),u+u\left(s\right)\right)\label{eq:local_holonomy}
\end{align}
 is holomorphic. From the open cover $\left(\Sigma_{s}\times D_{s}\right)_{s\in\left[0,1\right]}$
of the image of $\gamma$ we can extract a finite sub-cover corresponding
to discs centered at points $\gamma\left(s_{\ell}\right)$ for $0\leq\ell\leq N$
with
\begin{itemize}
\item $s_{0}=0$ and $s_{N}=1$,
\item $s_{\ell+1}>s_{\ell}$,
\item $\gamma\left(s_{\ell+1}\right)\in\left\{ 0\right\} \times D_{s_{\ell}}$.
\end{itemize}
We say that the collection of times $\left(s_{\ell}\right)_{0\leq\ell\leq N}$
is \textbf{adapted} to $\left(X,\gamma\right)$, and write
\begin{align*}
\gamma\left(s_{\ell}\right) & =\left(0,u_{\ell}\right)\,.
\end{align*}
Since each $h_{s}$ is holomorphic there exists $\varepsilon>0$ such
that the mapping 
\begin{align*}
\holo{\gamma}\,:\,\Sigma_{0}\times\left\{ u_{0}\right\}  & \longrightarrow\Sigma_{1}\times\left\{ u_{N}\right\} \\
p=\left(x,u_{0}\right) & \longmapsto\mbox{\ensuremath{\holo{\gamma}}}\left(p\right):=h_{s_{N-1}}\left(h_{s_{N-2}}\left(\cdots h_{0}\left(x,u_{1}-u_{0}\right)\cdots\right),u_{N}-u_{N-1}\right)
\end{align*}
is a well-defined mapping. We call this germ of a function the \textbf{holonomy}
\textbf{of $X$} (or of the underlying foliation) \textbf{along $\gamma$
with respect to the projection} $\left(x,u\right)\mapsto u$. Although
$\holo{\gamma}$ should depend on the particular choice of an adapted
collection this is not so, as is asserted in the well-known folk theorem:
\begin{thm}
\textbf{\emph{(Holonomy theorem)}} \label{thm:holonomy_theorem}Let
a vector field $X$ be holomorphic on a neighborhood of $\mathcal{E}$,
transverse to the fibers of the projection $\left(x,u\right)\mapsto u$
outside the singular locus. For every path $\gamma$ of $\mathcal{E}\backslash\sing X$
the following properties hold:
\begin{enumerate}
\item the holonomy $\holo{\gamma}$ of $X$ along $\gamma$ depends only
on the homotopy class of $\gamma$ in $\mathcal{E}\backslash\sing X$,
\item $\holo{\gamma}$ is a germ of a biholomorphism of $\Sigma:=\Sigma_{0}\times\left\{ u_{0}\right\} $
and $\holo{\gamma}\left(0,u_{0}\right)=\left(0,u_{N}\right)$,
\item if $\xi_{*}$ is a given base-point in $\mathcal{E}\backslash\sing X$
the mapping
\begin{align*}
\holo{\bullet}\,:\,\pi_{1}\left(\mathcal{E}\backslash\sing X,\xi_{*}\right) & \longrightarrow\tx{Diff}\left(\Sigma,\xi_{*}\right)\simeq\diff{}\\
\gamma & \longmapsto\holo{\gamma}
\end{align*}
is a group morphism. Its image is the \textbf{holonomy group} of $X$
(or better, of the induced foliation).
\end{enumerate}
\end{thm}

\subsubsection{\label{sub:holo_is_loc_anal}Local study}

~

\begin{figure}[H]
\includegraphics[width=0.75\columnwidth]{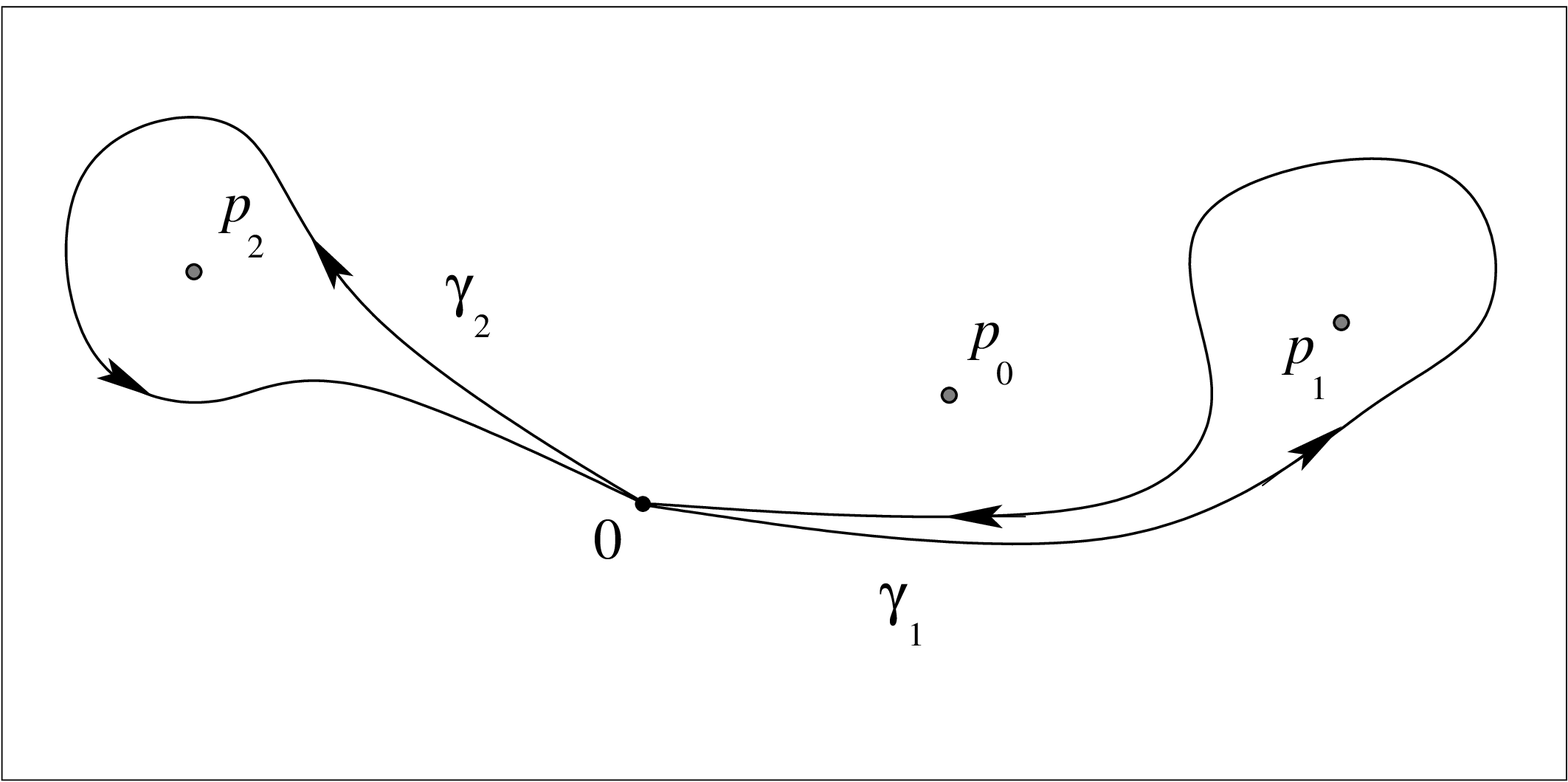}

\protect\caption{\label{fig:holo_generators}Generators of the fundamental group of
$\protect\ww P_{1}\left(\protect\ww C\right)\backslash\left\{ p_{0},p_{1},p_{2}\right\} $.}
\end{figure}

We continue to work in the affine coordinates $\left(x,u\right)$
through the chart $\pi_{x}$ and we fix the base-point $\zeta_{*}$
whose coordinates in $\mathcal{M}_{x}$ is $\left(0,0\right)$. For
$\frac{P}{Q}\in\mathtt{RND}^{*}$ we fix a proper representative $\left(P,Q\right)$
and consider the holomorphic vector field $X$ defined in~\eqref{eq:blown-up_vector-field}
by
\begin{align*}
X\left(x,u\right) & :=\frac{Q\left(x,xu\right)}{x}\pp{}x+\frac{P\left(x,xu\right)-uQ\left(x,xu\right)}{x^{2}}\pp{}u\,.
\end{align*}
Its integral curves define the foliation $\mathcal{F}$. The (isolated)
singular set $\sing X$ of $X$ coincides with that of $\mathcal{F}$,
and because $X\left(0,u\right)=\phi\left(u\right)\pp{}u$ the singularities
located on $\mathcal{E}$ are given by the three simple roots of the
polynomial $\phi$ defined by~\eqref{eq:dicritic_poly}. It is obvious
that the map
\begin{align*}
\left(P,Q\right)\in\quot^{-1}\left(\mathtt{RND}^{*}\right) & \longmapsto X\in\mathtt{VF}
\end{align*}
is strongly analytic, and from now on we only work with vector fields
$X$ instead of foliations $\mathcal{F}$.

In the sequel the superscript <<$0$>> refers to objects computed
from a fixed proper representative $\left(P^{0},Q^{0}\right)$ of
a meromorphic germ belonging to $\mathtt{RND}^{*}$. We also fix a
particular set of generators $\gamma_{1}$, $\gamma_{2}$ of the fundamental
group $\pi_{1}\left(\mathcal{E}\backslash\sing{X^{0}},\zeta_{*}\right)$,
where $\sing{X^{0}}=\left\{ p_{0}^{0},p_{1}^{0},p_{2}^{0}\right\} $
is the singular locus of $X^{0}$. For instance one can choose two
simple loops, each turning once around the singular point $p_{j}^{0}$.
Because $\mathtt{RND}^{*}$ does not cross the vanishing locus of
the discriminant~\eqref{eq:dicritic_cond_3} there exists a biholomorphic
map in the variables $\left(P_{\mathbf{n}}\right)_{\left|\mathbf{n}\right|=2}$
and $\left(Q_{\mathbf{n}}\right)_{\left|\mathbf{n}\right|=2}$, well-defined
in a neighborhood of $\left(P_{\mathbf{n}}^{0},Q_{\mathbf{n}}^{0}\right)_{\left|\mathbf{n}\right|=2}$,
whose range is the collection of roots of $\phi$. It is in particular
possible to find a neighborhood $U'$ of $\left(P_{\mathbf{n}}^{0},Q_{\mathbf{n}}^{0}\right)_{\left|\mathbf{n}\right|=2}\in\ww C^{6}$
so that the roots of $\phi$ does not cross the images of $\gamma_{1}$
and $\gamma_{2}$ for values of the second jet of $\left(P,Q\right)$
in $U'$. Let $V':=J_{2}^{-1}\left(U'\right)$ be the corresponding
neighborhood of $\left(P^{0},Q^{0}\right)$. 

Take $\left(P,Q\right)\in V'$; by construction $\gamma_{1}$ and
$\gamma_{2}$ generates $\pi_{1}\left(\mathcal{E}\backslash\sing{\mathcal{F}},\zeta_{*}\right)$.
Therefore the holonomy group of $\mathcal{F}$ is spanned by $\holo 1:=\holo{\gamma_{1}}$
and $\holo 2:=\holo{\gamma_{2}}$. In the neighborhood $V'$ of $\left(P^{0},Q^{0}\right)$
it is possible to keep track of the pair $\left(\holo 1,\holo 2\right)$,
with ordering, since the paths $\gamma_{1}$ and $\gamma_{2}$ do
not vary. The mapping
\begin{align*}
\left(P,Q\right)\in V' & \longmapsto\left(\holo 1,\holo 2\right)\in\diff{}^{2}
\end{align*}
is thereby well-defined. Let us show it is strongly analytic on a
maybe smaller $V\subset V'$. The way the holonomy is constructed
guarantees that it is $\quot$-invariant, meaning that the map $\frac{P}{Q}\in\quot\left(V'\right)\mapsto\left(\holo 1,\holo 2\right)$
also is well-defined and strongly analytic.
\begin{lem}
Let $\gamma$ be a loop of $\mathcal{E}\backslash\sing{X^{0}}$ with
base-point $\xi_{*}$ and consider a sequence $\left(s_{\ell}\right)_{0\leq\ell\leq N}$
adapted to $\left(X^{0},\gamma\right)$. There exists a neighborhood
$V\subset V'$ of $\left(P^{0},Q^{0}\right)$ such that the sequence
$\left(s_{\ell}\right)_{0\leq\ell\leq N}$ continues to be adapted
to $\left(X,\gamma\right)$.
\end{lem}
Before we prove this lemma we describe briefly how to derive Theorem~\ref{thm:holo_is_analytic}
from it. Consider a neighborhood $V\subset V'$ of $\left(P^{0},Q^{0}\right)$
such that $\gamma_{1}$ and $\gamma_{2}$ both admit an adapted covering
uniform with respect to $\left(P,Q\right)\in V$. Each holonomy generator
$\holo 1$ and $\holo 2$ is obtained by composing $N$ local flows,
and because $N$ does not depend on $\left(P,Q\right)$ this composition
is strongly analytic with respect to $\left(P,Q\right)$ whenever
each flow is. But the former statement precisely is the content of
Theorem~\eqref{thm:flow_is_strong_anal}, which ends our proof.
\begin{proof}
Since this property is local, and because there is only finitely many
times $\left(s_{k}\right)_{0\leq k<N}$, we prove it only for the
first point $u_{0}=0$. We can request without loss of generality
that $\eta>\left|u_{1}\right|$ be chosen so small that 
\begin{align*}
\sup_{\left|u\right|=\eta}\left|\phi\left(u\right)-\phi\left(0\right)\right| & <\frac{1}{4}\left|\phi\left(0\right)\right|\,,
\end{align*}
with a suitable, fixed choice of a thinner adapted covering of the
image of $\gamma$. Since $\phi$ depends only on the second jet of
$\left(P,Q\right)$ we can find a neighborhood $V\subset V'$ such
that the above estimate holds for fixed $\eta$ and all $\left(P,Q\right)\in V$.

For $\left(P,Q\right)\in V$ let us introduce 
\begin{align*}
\hat{X}\left(x,u\right) & :=\xi\left(x,xu\right)x\pp{}x+\pp{}u\\
\xi\left(x,u\right) & :=\frac{Q\left(x,xu\right)}{P\left(x,xu\right)-uQ\left(x,xu\right)}=\frac{Q\left(x,xu\right)}{\phi\left(u\right)+x\left(f\left(x,xu\right)+ug\left(x,xu\right)\right)}\,.
\end{align*}
Because of the form of the arguments in the functions making $\xi$
up, we can always find $\varepsilon>0$ sufficiently small so that
$\xi$ is holomorphic on $\varepsilon\ww D\times\eta\ww D$, in particular
because there exists $\varepsilon>0$ such that for all $\left|x\right|<\varepsilon$
and all $\left|u\right|\leq\eta$ we have
\begin{align*}
\left|\phi\left(u\right)-\phi\left(0\right)+x\left(f\left(x,xu\right)+ug\left(x,xu\right)\right)\right| & <\frac{1}{2}\left|\phi\left(0\right)\right|\,.
\end{align*}
Let us define
\begin{align*}
C & :=\norm{\xi}{\varepsilon\ww D\times\eta^{0}\ww D}\,.
\end{align*}
We need to prove that the parameterization of the integral curves
of $\hat{X}$ by the flow, starting from points of $\left\{ u=0\,,\,\left|x\right|<\varepsilon\right\} $,
is holomorphic on a domain which contains the disc $\eta\ww D$ for
a maybe smaller $\varepsilon>0$. Let $t\mapsto z\left(t,x,u\right)$
be the $x$-component of the flow $\flow{\hat{X}}t\left(x,0\right)$,
\emph{i.e.} the solution of
\begin{align*}
\begin{cases}
\dot{z}\left(t\right) & =z\left(t\right)\xi\left(z\left(t\right),tz\left(t\right)\right)\\
z\left(0\right) & =x
\end{cases} & .
\end{align*}
Setting $t:=e^{\ii\theta}\tau$ with $\tau,\,\theta\in\ww R_{\geq0}$
and differentiating $\left|z\right|^{2}=z\overline{z}$ with respect
to $\tau$ brings the equation
\begin{align*}
\frac{\dd{\left|z\left(t\right)\right|^{2}}}{\dd{\tau}} & =2\left|z\left(t\right)\right|^{2}\re{e^{\ii\theta}\xi\left(z\left(t\right),tz\left(t\right)\right)}\,,
\end{align*}
therefore
\begin{align*}
\left|z\left(t\right)\right| & \leq\left|x\right|\exp\left(C\left|t\right|\right)\,.
\end{align*}
In particular the integral curve $t\mapsto\flow{\hat{X}}t\left(x,0\right)$
does not escape the polydisc $\varepsilon\ww D\times\eta\ww D$ provided
$\left|x\right|<\varepsilon\exp\left(-\eta C\right)$ and $\left|t\right|<\eta$,
which means that the local holonomy $h_{0}$ defined in~\eqref{eq:local_holonomy}
is holomorphic on $\Sigma_{0}\times D_{0}$ where $\Sigma_{0}:=\varepsilon\exp\left(-\eta C\right)\ww D$
and $D_{0}:=\eta\ww D$. Since $\eta$ depends only on $V$ and not
on a particular choice of $\left(P,Q\right)\in V$ the result is proved. 
\end{proof}

\subsubsection{\label{sub:holo_is_anal_and_onto}$\protect\hol{\bullet}$ is globally
analytic and quasi-onto}

~

The local strong-analyticity we just established implies that $\hol{\bullet}$
is a well-defined strongly analytic map on the universal covering
$\mathcal{C}\,:\,\widetilde{\mathtt{RND}^{*}}\to\mathtt{RND}^{*}$.
We perform the analytic continuation of $\hol{\bullet}$ by deforming
continuously the loops $\gamma_{1}$ and $\gamma_{2}$ so that no
singular point ever crosses the image of any loop. The fiber of $\mathcal{C}$
thereby corresponds to foliations with same singular points on the
exceptional divisor but with generators of the holonomy group that
may not appear in the same order. Yet the group generated is the same,
so the mapping $\hol{\bullet}$ is well-defined as a map from $\mathtt{RND}^{*}$
to the orbifold quotient $\nf{\diff{}^{2}}{\frak{S}_{2}}$.

\bigskip{}

To show the holonomy map is quasi-onto we use a result by \noun{A.\,Lins-Neto}:
\begin{thm}
\label{thm:realisation_Lins-Neto}\cite{LinsNeto} Any finitely-generated
subgroup $\left\langle \holo 1,\cdots,\holo n\right\rangle $ of $\diff{}$
such that $\bigcirc_{k=1}^{n}\holo k=\id$ and at least one generator
is analytically linearizable, can be obtained as the projective holonomy
group computed along the exceptional divisor of a foliation reduced
after one blow-up.
\end{thm}
The construction, based on Grauert's theorem, also ensures that the
foliation realizing a given subgroup $\left\langle \holo 1,\holo 2\right\rangle $
belongs to $\mathtt{RND}^{*}$. Besides it is well known that if $\frac{1}{2\ii\pi}\log\holo 1'\left(0\right)$
or $\frac{1}{2\ii\pi}\log\holo 2'\left(0\right)$ is not real then
$\holo 1$ or $\holo 2$ is hyperbolic and therefore linearizable.

{\footnotesize{}\bibliographystyle{preprints}
\bibliography{biblio}
}{\footnotesize \par}
\end{document}